\documentclass[12pt]{amsart}
\usepackage{amsmath,amssymb,amsthm,ascmac}
\usepackage[dvipdfmx]{graphicx,xcolor}
\usepackage[all]{xy}
\usepackage{proof}

\theoremstyle{plain}
\newtheorem{theorem}{Theorem}[section]
\newtheorem{lemma}[theorem]{Lemma}
\newtheorem{prop}[theorem]{Proposition}
\newtheorem{fact}[theorem]{Fact}

\newtheorem{obs}[theorem]{Observation}

\newtheorem{question}{Question}

\theoremstyle{definition}
\newtheorem*{ack}{Acknowledgements}
\newtheorem{example}[theorem]{Example}
\newtheorem{definition}[theorem]{Definition}

\newtheorem*{remark}{Remark}

\newcommand{\N}{\mathbb{N}}
\newcommand{\R}{\mathbb{R}}

\newcommand{\lemr}{\mbox{-}{\sf LEM}_\R}
\newcommand{\dmlr}{\mbox{-}{\sf DML}_\R}
\newcommand{\dner}{\mbox{-}{\sf DNE}_\R}

\newcommand{\om}{\omega}
\newcommand{\leqW}{\leq_{\sf W}}
\newcommand{\leqcW}{\leq^c_{\sf W}}

\newcommand{\leqcGW}{\leq^c_{\gW}}
\newcommand{\jump}{\mathbf{j}}
\newcommand{\rea}{\Vdash}
\newcommand{\reap}{\Vdash_\mathbf{j}}
\newcommand{\reak}{\Vdash_K}

\newcommand{\upto}{\upharpoonright}
\newcommand{\pcolon}{\colon\!\!\!\subseteq}
\newcommand{\tto}{\rightrightarrows}

\newcommand{\lamq}{\lambda}

\newcommand{\dom}{{\rm dom}}

\newcommand{\ul}[1]{\underline{#1}}

\newcommand{\bfN}{\mathbf{N}}
\newcommand{\set}{\mathbf{Set}}
\newcommand{\foralle}{\forall^+}

\newcommand{\wklQ}{{\sf WKL}_{=2}}

\newcommand{\spair}[2]{(#1\Vdash #2)}
\newcommand{\sspair}[2]{(#1,#2)}

\newcommand{\dia}{\diamondsuit}
\newcommand{\ep}{\varepsilon}

\newcommand{\num}[1]{\underline{ #1}}

\newcommand{\limN}{{\sf Lim}_\om}

\newcommand{\one}{\mathbf{1}}
\newcommand{\fr}{{}^\smallfrown}

\newcommand{\gW}{{\sf gW}}
\newcommand{\clo}{\dia}

\newcommand\tboldsymbol[1]{%
\protect\raisebox{0pt}[0pt][0pt]{%
$\underset{\widetilde{}}{\mathbf{#1}}$}\mbox{\hskip 1pt}}

\newcommand{\ifthen}[3]{
{\sf if}\ {#1}\ {\sf iszero}\ {\sf then}\ {#2}\ {\sf else}\ {#3}
}

\newcommand{\tpbf}[1]{\tboldsymbol{#1}}

\newcommand{\com}[1]{
\color{magenta}\emph{#1}
\color{black}}

\title[Realizability and reverse mathematics]
{Degrees of incomputability, realizability and constructive reverse mathematics}

\author{Takayuki Kihara}

\date{}

\begin{document}
\maketitle

\begin{abstract}
We introduce a method for assigning a realizability notion to each degree of incomputability.
In our setting, we make use of Weihrauch degrees (degrees of incomputability/discontinuity of partial multi-valued functions) to obtain Lifschitz-like relative realizability predicates.
We present sample examples on how to lift some separation results on Weihrauch degrees to those over intuitionistic Zermelo-Fraenkel set theory ${\bf IZF}$.
\end{abstract}

\section{Introduction}

\subsection{Summary}
This article is a contribution to constructive reverse mathematics initiated by Ishihara \cite{Is05,Ish06}.
{\em Reverse mathematics} is the field that seeks to find the minimum axioms necessary to prove a given mathematical theorem \cite{SOSOA:Simpson}.
Classical reverse mathematics measures the complexity of {\em induction} and {\em comprehension} principles which are necessary for proving a given theorem; the former is particularly important from a foundational perspective.
Constructive reverse mathematics measures the complexity of the law of {\em excluded middle} which is necessary for proving a given theorem.

For the overview of previous works on constructive reverse mathematics, we refer the reader to Diener \cite{Die18} (and also \cite{Ish23}).
However, contrary to \cite{Die18}, in our article, we do not include the axiom of countable choice ${\sf AC}_{\om}$ in our base system of constructive reverse mathematics.
Why do we think that the axiom of countable choice ${\sf AC}_{\om}$ should not be implicitly assumed?
This is because including ${\sf AC}_{\om}$ makes it difficult to compare the results with classical reverse mathematics \`a la Friedman-Simpson \cite{SOSOA:Simpson}.
For instance, weak K\"onig's lemma ${\sf WKL}$ is equivalent to the intermediate value theorem ${\sf IVT}$ over Biship-style constructive mathematics ${\bf BISH}$ (which entails countable choice), whereas ${\sf WKL}$ is strictly stronger than ${\sf IVT}$ in classical reverse mathematics.

In order to avoid this difficulty, several base systems other than ${\bf BISH}$ have been proposed.
For instance, Troelstra's elementary analysis ${\bf EL}_0$ is widely used as a base system of constructive reverse mathematics \cite{Nem23}.
Here, ``${\bf EL}_0$ plus the law of excluded middle'' is exactly the base system ${\bf RCA}_0$ of classical reverse mathematics; and moreover, ``${\bf EL}_0$ plus the axiom of countable choice'' is considered as Biship-style constructive mathematics ${\bf BISH}$.
Thus, ${\bf EL}_0$ lies in the intersection of ${\bf RCA}_0$ and ${\bf BISH}$.

Our aim is to separate various non-constructive principles which are equivalent under countable choice ${\sf AC}_{\om}$, and our main tool is (a topological version of) Weihrauch reducibility (cf.~\cite{pauly-handbook}).
The relationship between Weihrauch reducibility and intuitionistic systems are extensively studied, e.g.~in \cite{Fu21,Kuy17,Uf18,Yos2}.
Thus, some of separations may be automatically done using the previous works.
Our main purpose in this article is to develop an intuitive and easily customizable framework constructing ${\sf AC}_{\om}$-free models separating such non-constructive principles.
In practice, establishing a convenient framework for separation arguments is sometimes more useful than giving a completeness result (for instance, look at classical reverse mathematics, where computability-theorists make use of Turing degree theory in order to construct $\om$-models separating various principles over ${\bf RCA}_0$, cf.~\cite{Sho10}).
Our work gives several sample separation results, not just over ${\bf EL}_0$, but even over intuitionistic Zermelo-Fraenkel set theory ${\bf IZF}$.

\subsubsection{Variants of weak K\"onig's lemma}\label{sec:between-llpo-wkl}

In this article, we discuss a hierarchy between ${\sf LLPO}$ and ${\sf WKL}$ which collapses under the axiom ${\sf AC}_{\om,2}$ of countable choice (even with binary values).
Here, ${\sf AC}_{X,Y}$ is the statement that 
\[\forall x\in X\exists y\in Y\varphi(x,y)\longrightarrow\exists f\colon X\to Y\forall x\in X\varphi(x,f(x)),\]
and $\om$ is the set of all natural numbers.
Consider the following three principles:

\begin{itemize}
\item The binary expansion principle ${\sf BE}$ states that every regular Cauchy real has a binary expansion.
\item The intermediate value theorem ${\sf IVT}$ states that for any continuous function $f\colon[0,1]\to[-1,1]$ if $f(0)$ and $f(1)$ have different signs then there is a regular Cauchy real $x\in[0,1]$ such that $f(x)=0$.
\item Weak K\"onig's lemma ${\sf WKL}$ states that every infinite binary tree has an infinite path.
\end{itemize}

Here, a regular Cauchy real is a real $x$ which is represented by a sequence $(q_n)_{n\in\om}$ of rational numbers such that $|q_n-q_m|<2^{-n}$ for any $m\geq n$.
As mentioned above, under countable choice ${\sf AC}_{\om,2}$, we have
\[{\sf WKL}\leftrightarrow{\sf IVT}\leftrightarrow{\sf BE}\leftrightarrow{\sf LLPO}.\]

Even if countable choice, Markov's principle, etc.~are absent, we have the forward implications; see Berger et al.~\cite{BIKN}.
Ishihara and Nemoto (in private communication) asked what happens to the backward implications under the absence of the axiom of countable choice.

\subsubsection*{Robust division principle ${\sf RDIV}$}
We also examine the {\em division} principles for reals.
It is known that {\em Markov's principle} (double negation elimination for $\Sigma^0_1$-formulas) is equivalent to the statement that, given regular Cauchy reals $x,y\in[0,1]$, we can divide $x$ by $y$ whenever $y$ is nonzero; that is, if $y$ is nonzero then $z=x/y$ for some regular Cauchy real $z$.

However, it is algorithmically undecidable if a given real $y$ is zero or not.
To overcome this difficulty, we consider any regular Cauchy real $z$ such that if $y$ is nonzero then $z=x/y$ as a value of the division.
If we require $x\leq y$ then such a $z$ always satisfies $x=yz$ whatever $y$ is (since $y=0$ and $x\leq 0$ implies $x=0$, so any $z$ satisfies $0=0z$).
To avoid the difficulty of deciding if $y$ is nonzero or not, we always assume this additional requirement, and then call it the {\em robust division} principle ${\sf RDIV}$.
In other words, ${\sf RDIV}$ is the following statement:
\[(\forall x,y\in[0,1])\;\left[x\leq y\rightarrow\;(\exists z\in[0,1])\;x=yz\right].\]

One may think that the condition $x\leq y$ is not decidable, but we can always replace $x$ with $\min\{x,y\}$ without losing anything.
Namely, the robust division principle ${\sf RDIV}$ is equivalent to the following:
\[(\forall x,y\in[0,1])(\exists z\in[0,1])\;\min\{x,y\}=yz.\]

The principle ${\sf RDIV}$ is known to be related to problems of finding Nash equilibria in bimatrix games \cite{Pa10} and of executing Gaussian elimination \cite{pauly-kihara2-mfcs}.
The following implications are known:
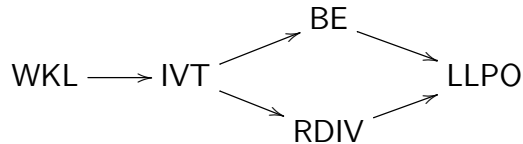
\[
\xymatrix@R=6pt{
 & & {\sf BE} \ar[dr] & \\
{\sf WKL} \ar[r] & {\sf IVT} \ar[dr] \ar[ur] & & {\sf LLPO} \\
 & & {\sf RDIV} \ar[ur] & 
}
\]

\subsection{Main Theorem}

In this article, to construct ${\sf AC}_{\om,2}$-free models of ${\bf IZF}$, we incorporate Weihrauch reducibility \cite{pauly-handbook} into Chen-Rathjen's Lifschitz-like realizability \cite{ChRa12}.
In general, we introduce the notion of a {\em $j$-operator} on a relative partial combinatory algebra.
We call the pair of a relative partial combinatory algebra $\mathbb{P}$ and a $j$-operator $\jump$ a {\em site}.
Then we show that given a site $(\mathbb{P},\jump)$ always yields a realizability predicate which validates ${\bf IZF}$.
As sample results of our method, we describe realizability interpretations of the following:

\begin{theorem}\label{thm:main-theorem}
Each of the following is realizable (with respect to the realizability interpretation over a suitable site).
\begin{enumerate}
\item ${\bf IZF}+{\sf LLPO}+\neg{\sf RDIV}+\neg{\sf BE}$.\label{item:main1b}
\item ${\bf IZF}+{\sf RDIV}+\neg{\sf BE}$.\label{item:main2b}
\item ${\bf IZF}+{\sf BE}+\neg{\sf RDIV}$.\label{item:main3b}
\item ${\bf IZF}+{\sf RDIV}+{\sf BE}+\neg{\sf IVT}$.\label{item:main4b}
\item ${\bf IZF}+{\sf IVT}+\neg{\sf WKL}$.\label{item:main5b}
\end{enumerate}
\end{theorem}

This solves the Ishihara-Nemoto problem.

\section{Background and key ideas}

\subsection{Constructive reverse mathematics}

We need benchmark principles to measure the complexity of mathematical theorems.
In constructive reverse mathematics, one such benchmark is the arithmetical hierarchy of the law of excluded middle.

\subsubsection{Law of excluded middle}\label{sec:subsec-lem-coll}

We here consider central nonconstructive principles, the {\em law of excluded middle} {\sf LEM}: $P\lor\neg P$, the {\em double negation elimination} {\sf DNE}: $\neg\neg P\to P$, and {\em de Morgan's law} {\sf DML}: $\neg(P\land Q)\to\neg P\lor \neg Q$.

For each principle ${\sf L}\in\{{\sf LEM},{\sf DNE},{\sf DML}\}$, let $\Gamma$-${\sf L}$ denote the principle ${\sf L}$ restricted to $\Gamma$-formulas $P,Q$.
Since what deal with in this paper is a principle on the reals, we also deal with the hierarchy of the law of excluded middle that takes (expressions of) reals as parameters.
When emphasizing this, the principle ${\sf L}$ is written as ${\sf L}_\mathbb{R}$.
For example, $\Sigma^0_1$-${\sf LEM}_\mathbb{R}$ is the scheme $\varphi(\bar{x})\lor\neg\varphi(\bar{x})$ for $\Sigma^0_1$ formulas $\varphi(\bar{x})$ with real parameters $\bar{x}$.

More precisely, we examine the arithmetical hierarchy for these nonconstructive principles (cf.~Akama et al.~\cite{Aka}) for formulas of second order arithmetic with function variables $\bar{v}$ (i.e, variables from $\N^\N$).
For instance, we consider the class $\Sigma^0_n$ of formulas  of the form 
\[\exists a_1\forall a_2\dots {\sf Q}a_n R(a_1,a_2,\dots,a_n,\bar{v}),\]
where ${\sf Q}=\exists$ if $n$ is odd, ${\sf Q}=\forall$ if $n$ is even, each $a_i$ ranges over $\N$, and $R$ is a decidable formula.
Note that every quantifier-free formula is decidable.
These principles are also described as nonconstructive principles on the regular reals $\mathbb{R}$.
For instance:
\begin{itemize}
\item $\Sigma^0_1\dner$ is known as the {\em Markov principle} ${\sf MP}$, which is also equivalent to the following statement: For any regular Cauchy reals $x,y\in\R$, if $y\not=0$ then you can divide $x$ by $y$; that is, there is a regular Cauchy real $z\in\R$ such that $x=yz$. 
\item $\Sigma^0_1\dmlr$ is equivalent to the {\em lessor limited principle of omniscience} ${\sf LLPO}$, which states that for any regular Cauchy reals $x,y\in\mathbb{R}$, either $x\leq y$ or $y\leq x$ holds.
\item $\Pi^0_1\lemr$ is equivalent to the {\em limited principle of omniscience} ${\sf LPO}$, which states that for any regular Cauchy reals $x,y\in\mathbb{R}$, either $x=y$ or $x\not=y$ holds.
\item $\Pi^0_1\lemr\leftrightarrow\Sigma^0_1\lemr$ holds under Markov's principle.
\item $\Sigma^0_2\dmlr$ is equivalent (assuming Markov's principle, cf.~Proposition \ref{prop:dmlr-rt12}) to the {\em pigeonhole principle} for infinite sets (or ${\sf RT}^1_2$ in the context of reverse mathematics).
\end{itemize}

%

The arithmetical hierarchy of the law of excluded middle is first studied in the context of limit computable mathematics (one standpoint of constructive mathematics which allows {\em trial and errors} \cite{HaNa02}, which may be considered as a special case of a realizability interpretation of constructive mathematics w.r.t.~a suitable partial combinatory algebra, see also Section \ref{sec:realizability-pca}).
After that, the arithmetical hierarchy of law of excluded middle became extensively studied in various contexts; see e.g.~\cite{Aka,Ber06,BeSt17,FuKu21,Fuku22,Nakata}.

\begin{example}
The equivalence between $\Sigma^0_2\dner$ and Hilbert's basis theorem (Dickson's lemma), see \cite[Section 5]{HaNa02}, can also be considered as an early result in constructive reverse mathematics.

In \cite{BeSt17}, it is shown that some variant of $\Sigma^0_3$-${\sf DML}_\R$ is equivalent to Ramsey's theorem for computable coloring of pairs, ${\sf RT}^2_2(\Sigma^0_0)$, over Heyting arithmetic ${\bf HA}$.
\end{example}

\begin{figure}[t]
\[{\small
\xymatrix{
	&	\Sigma^0_2\lemr \ar[ddl] \ar[dr] \ar@/^1.8pc/[rrrdd]	&	&	&		{\sf WKL} \ar[dd]	\\
	& 	&	\Pi^0_2\lemr	\ar[d] 	&	&		\\
\Sigma^0_2\dner\ar[dr] \ar@/^1.8pc/[rrrdd]	&	&	\Sigma^0_2\dmlr \ar[dl]|\hole	&		&	{\sf IVT}\ar[dd]\ar[ddl]	\\
	&	\Delta^0_2\lemr \ar[d] & &  &\\
	&	\Sigma^0_1\lemr \ar[dr]\ar[ddl] & & {\sf IVT}_{\rm lin} \ar[d]\ar[dr] & {\sf BE} \ar[d]\\
	&	& \Pi^0_1\lemr \ar[r] \ar[d] &   {\sf RDIV} \ar[dr] & {\sf BE}_\mathbb{Q}\ar[d] \\
\Sigma^0_1\dner\ar[dr] & & \Sigma^0_1\dmlr\ar[rr]\ar[dl] & & {\sf LLPO}\ar[ll] \\
	& \mbox{base system} & & & 
}}
\]
\caption{Nonconstructive principles over a base intuitionistic system}\label{figure:principles-over-izf}
\end{figure}

For the relationship among non-constructive principles, see also Figure \ref{figure:principles-over-izf}, where ${\sf IVT}_{\sf lin}$ and ${\sf BE}_\mathbb{Q}$ are some auxiliary (less important) principles introduced in Section \ref{sec:Auxiliary}.
We explain some nontrivial implications:
For $\Sigma^0_2\lemr\to{\sf IVT}$, the standard proof of the intermediate value theorem employs the nested interval argument with the conditional branching on whether $f(x)=0$ for some $x\in\mathbb{Q}$.
As $\mathbb{Q}$ is $\Sigma^0_2$, this shows that $\Sigma^0_2\lemr$ implies ${\sf IVT}$, and hence ${\sf BE}$.
For $\Pi^0_1\lemr\to{\sf RDIV}$, one can decide if $y$ is zero or not by using $\Pi^0_1\lemr$.
If $y=0$ then put $z=0$; otherwise define $z=x/y$.
Thus, $\Pi^0_1\lemr$ implies ${\sf RDIV}$.
For other nontrivial implications, see Section \ref{section:implications}.

\subsection{Function presentations of theorems}

\subsubsection{Theorems as multivalued functions}
A simple, but important, observation is that most principles which appear in Theorem \ref{thm:main-theorem} are of the $\forall\exists$-forms.
\[\varphi\equiv\forall x\in X\ [\eta(x)\to\exists y\in Y.\,\theta(x,y)]\]

Here, $\eta$ and $\theta$ may also involve quantifiers.
The basic idea is to regard this type of formula $\varphi$ as a (multivalued) function $F_\varphi$ that takes $x$ as input and outputs $y$.
In other words, a (possibly false) statement $\varphi\equiv\forall x\in X\;[\eta(x)\rightarrow\exists y\in Y.P(x,y)]$ is transformed into the following $F_\varphi$:
\begin{align*}
{\rm dom}(F_\varphi)&=\{x\in X:\eta(x)\}\\
F_\varphi(x)&=\{y\in Y:\theta(x,y)\}
\end{align*}

Think of $F_\varphi$ as a multivalued function:
Each $y\in F_\varphi(x)$ is regarded as an output of $F_\varphi(x)$.
In this way, we view each $\forall\exists$-principle $\varphi$ as a {\em partial multi-valued function} ({\em multifunction}).
Here, we consider formulas as partial multifunctions rather than relations in order to distinguish a {\em hardest} instance $F_\varphi(x)=\emptyset$ (corresponding to a {\em false sentence}), and an {\em easiest} instance $x\in X\setminus{\dom}(F_\varphi)$ (corresponding to a {\em vacuous truth}).
Only if $F_\varphi(x)$ is a singleton for any $x\in{\rm dom}(F_\varphi)$, we think of $F_\varphi$ as a single-valued function.
We use the symbol $F\pcolon X\tto Y$ to mean that $F$ is a partial multifunction from $X$ to $Y$.

The basic idea for measuring the complexity of a theorem $\varphi$ is to analyze how complicated the function $F_\varphi$ is.
In a mathematical system based on constructive logic, the complexity of an $\forall\exists$-theorem can be predicted to be dominated by the complexity of constructing an existential witness.
Therefore, it seems appropriate to measure the complexity of a {\em choice} function $f$ of $F_\varphi$, i.e., a function $f$ such that $f(x) \in F_\varphi(x)$ for all $x \in {\rm dom}(F_\varphi)$.
This is a function choosing an existential witness for $\varphi$.

\subsubsection{Degrees of discontinuity}\label{sec:degree-discontinuity}

If $X,Y$ are topological spaces, we can discuss whether $F_\varphi$ has a continuous choice function, or if not, how discontinuous a choice function is.
All principles mentioned above, except for Markov's principle, are {\em discontinuous} principles; that is, there are no continuous choice functions for these principles.
Thus, all of these principles fail in a well-known model of ${\bf IZF}$ (or a topos) in which all functions on $\om^\om$ are continuous.
Here, it is important to note that, for example, a statement that ``for any regular Cauchy real $\ldots$'' refers to the space of rational sequences (which is zero-dimensional) as the domain.

\begin{example}
The law of excluded middle $\Pi^0_2\lemr$ for $\Pi^0_2$ formulas (as a single-valued function) clearly corresponds to {\em Dirichlet's nowhere continuous function $\chi_\mathbb{Q}$}.
\end{example}

\begin{example}
For the binary expansion principle ${\sf BE}$, any total binary expansion $\mathbb{R}\to 2^\om$ is discontinuous, and indeed, there is no continuous function which, given a regular Cauchy representation of a real, returns its binary expansion.
\end{example}

\begin{example}
As already mentioned, some proof of the intermediate value theorem ${\sf IVT}$ can be given by determining whether a real is rational or not, i.e., using Dirichlet's function, and the rest can be done by continuous reasoning.
However, the standard proof of ${\sf IVT}$ is more reminiscent of Thomae's function than Dirichlet's function.
\[
{\rm Thomae}(x)=
\begin{cases}
\frac{1}{q} & \mbox{($x$ is of the form $\frac{p}{q}$ for coprime $p,q$)}\\
0 & \mbox{($x$ is irrational)}
\end{cases}
\]

Thomae's function has the property of being {\em discontinuous on the rationals} and {\em continuous on the irrationals}.
The proof of ${\sf IVT}$ (i.e., for any continuous $f$, if $f(0)f(1)<0$ then $f(x)=0$ for some $x\in[0,1]$) is given by the so-called interval-halving method, but if a solution $x$ is a dyadic rational $m/2^{n}$, the process terminates in a finite number of steps discontinuously, and if $x$ is a dyadic irrational, the process continues continuously forever.
In other words, the standard proof of ${\sf IVT}$ has the same property as Thomae's function, which is {\em discontinuous on the dyadic rationals} and {\em continuous on the dyadic irrationals}.

Indeed, using this argument, one can formally prove that an upper bound on the discontinuity of ${\sf IVT}$ is given by Thomae's function (see \cite{KiPa19}).
Of course, the discontinuity of the standard proof is only as complicated as that of Thomae's function, so there may be other proofs with lower discontinuity, making this merely an upper bound.
\end{example}

There are differences among the {\em degrees of discontinuity} for these principles.

\begin{remark}
From the end of the 19th century to the beginning of the 20th century, mathematical analysts introduced various notions for measuring the degree of discontinuity:
A function is of Baire class $0$ if it is continuous.
A function is of Baire class $n+1$ if it is the pointwise limit of functions of Baire class $n$.
A function is $\Gamma$-measurable if the preimage of an open set is in $\Gamma$.
A function is $\sigma$-continuous if some countable partition of the domain makes it continuous when restricted to each piece.
By the Lebesgue-Hausdorff theorem, on a good topological space (such as the Baire space $\om^\om$), being of Baire $n$ and being $\tpbf{\Sigma}^0_{n+1}$-measurable are equivalent (see e.g.~\cite[Theorem 24.3]{KechrisBook}).
\end{remark}

\begin{example}
Dirichlet's function is $\sigma$-continuous (i.e., decomposable into countably many continuous functions), and of Baire class $2$ (i.e., a double pointwise limit of continuous functions), but not of Baire class $1$.
On the other hand, Thomae's function is of Baire class $1$, so Thomae's function is a more precise upper bound for {\sf IVT}.
\end{example}

\begin{example}
Weak K\"onig's lemma ${\sf WKL}$ has no $\sigma$-continuous choice, but has a choice of Baire class $1$.
Note also that ${\sf WKL}$ (even ${\sf IVT}$) has no $\tpbf{\Delta}^0_2$-measurable choice, where there are functions of Baire class $1$ (equivalently $F_\sigma$-measurable) which are not $\tpbf{\Delta}^0_2$-measurable.
On the other hand, ${\sf RDIV}$ has a $\tpbf{\Delta}^0_2$-measurable choice.
\end{example}

These observations indicate that these principles are all different.
However, one of the questions is, given a level $\mathbf{d}$ of discontinuity, how one can construct a model of constructive mathematics in which the functions on $\om^\om$ are exactly those with discontinuity level at most $\mathbf{d}$.

%
%

\subsubsection{Theorems as realizability problems}

The discussion in Section \ref{sec:degree-discontinuity} can be justified using Kleene's functional realizability interpretation.
Realizability is the act of keeping track of witnesses (e.g., existential witnesses) for the correctness of a formula at all subformula levels (see e.g.~\cite{Tr98}).
For example, a realizer of an $\forall\exists$-theorem $\varphi\equiv\forall x\exists y\theta(x,y)$ contains information about a function which, given $x$, returns an existential witness $y$.
In other words, a realizer always provides a choice function for $F_\varphi$.

There are several variants of realizability.
Roughly speaking, Kleene's number realizability only validates the principles which have computable choices and refutes other non-computable principles, and Kleene's functional realizability validates the principles which have continuously choices and refutes other discontinuous principles.
For the formal definition of realizability, see Section \ref{sec:realizability-interpretation-def}.

Then, can we formalize the discussion in Section \ref{sec:degree-discontinuity} as a variant of realizability?
For example, instead of realizability by continuous functions, one may consider realizability by Baire $1$ functions or $\sigma$-continuous functions.
At least for the former, it is clear that this does not yield a reasonable realizability since Baire $1$ functions are not closed under composition.
Furthermore, while the idea of considering the degrees of discontinuity is very promising, classical notions such as the Baire class are too coarse.
A more useful tool is the notion of reducibility (relative computability/continuity), which has developed in {\em computability theory}.

What is dealt with in computability theory is the degree of difficulty of ``problems.''
We consider an $\forall\exists$-theorem as a problem of finding an existential witness.
In other words, consider a multifunction $F\pcolon X \tto Y$ as the problem of finding at least one solution $y \in F(x)$ for each instance $x\in{\rm dom}(F)$.
Then one can consider the degree of difficulty of the problem $F$.
This is given by a preorder relation $G \leq F$, which roughly means that $F$ is more difficult than $G$ or of equal difficulty.
In other words, $G$ is $F$-relatively solvable (if one knows a solution to $F$, one can also solve $G$).
This computability-theoretic preorder relation provides a much more delicate and precise measure than classical topological notions such as the Baire class.

\subsection{Summary}
Our theoretical goal is to define a notion of realizability relative to a given problem $F$.
The process is as follows: In the following sections, we formalize the notion of $F$-relative computation (in game-theoretic terms; Section \ref{subsec:weihrauch-reducibility}), then introduce the notion of universal computation $F^{\clo}$ for $F$-relative computability (Definition \ref{def:relative-computation-F}), which becomes a $j$-operator (Definition \ref{def:j-operator}), from which we construct a site (Definition \ref{def:site}), and finally, the realizability over the site provides an interpretation of constructive mathematics (Theorem \ref{thm:IZF-realizability}), which moreover corresponds to $F$-relative computability (Lemma \ref{lem:main-lemma}) if $F$ is well-behaved.

At a concrete level, we use powerful techniques on computability-theoretic games to show that there are differences in {\em computability-theoretic strength} between various $\forall\exists$ principles (Theorem \ref{thm:main-sub-GW}).
Then, using our theory, we transfer the computability-theoretic results to separation results over {\bf IZF} and complete the proof of Main Theorem \ref{thm:main-theorem} in Section \ref{sec:main-section-proof}.



\section{Key technical notions}

\subsection{Notation}

Let $\om$ be the set of all natural numbers.
Let $\om^{<\om}$ be the set of all finite strings.
For strings $\sigma,\tau\in\om^{<\om}$, we write $\sigma\preceq\tau$ if $\tau$ extends $\sigma$.
The concatenation of $\sigma$ and $\tau$ is denoted by $\sigma\fr\tau$.
For a string $\langle n\rangle$ of length $1$, the concatenation $\sigma\fr\langle n\rangle$ is often abbreviated as $\sigma\fr n$ or $\sigma n$.
For a string $\sigma$, the length of $\sigma$ is denoted as $|\sigma|$. 
For an infinite sequence $p\in\om^\om$, we use $p\upto n$ to denote the restriction of $p$ up to the length $n$.
A tree is a set $T\subseteq\om^{<\om}$ which is downward closed with respect to $\preceq$; that is, $\tau\preceq\sigma\in T$ implies $\sigma\in T$.
An infinite path through $T$ is an infinite sequence $p\in\om^\om$ such that $p\upto n\in T$ for any $n\in\om$.
We write $[T]$ for the set of all infinite paths through $T$.

The notation $f\pcolon X\to Y$ indicates that $f$ is a partial function from $X$ to $Y$.
A partial continuous function $\Phi\pcolon\om^\om\to\om$ is determined by a partial function $\varphi\pcolon\om^{<\om}\to\om$ such that, for any $f,c$, we have $f\in{\rm dom}(\Phi)$ and $\Phi(f)=c$ iff there is $n\in\om$ such that $f\upto n\in{\rm dom}(\varphi)$ and $\varphi(f\upto n)=c$ for some $n\in\om$.
We may assume that ${\rm dom}(\varphi)$ is upward closed with respect to $\preceq$.
Note that $\sigma,\tau\in{\rm dom}(\varphi)$ and $\sigma\preceq\tau$ implies $\varphi(\sigma)=\varphi(\tau)$.
We often identify $\Phi$ with $\varphi$.
We write $\Phi(\sigma)\downarrow$ if $\sigma\in{\rm dom}(\varphi)$.
Otherwise, we write $\Phi(\sigma)\uparrow$.

The power set of a set $X$ is written as $\mathcal{P}X$.
A partial function $F\pcolon X\to\mathcal{P}Y$ is called a {\em partial multifunction} from $X$ to $Y$, and sometimes written as $F\pcolon X\tto Y$.
If $x\in{\rm dom}(F)$, we sometimes write $F(x)\downarrow$; otherwise, $F(x)\uparrow$.

\subsection{Weihrauch reducibility}

The main tool for measuring the degrees of difficulty of partial multifunctions in this article is Weihrauch reducibility (see e.g.~\cite{pauly-handbook}).
Below we assume that ${\bf X}$ is either $\om$ or $\om^\om$.

\subsubsection{Weihrauch reducibility}\label{subsec:weihrauch-reducibility}


Let $F,G\colon{\bf X}\tto{\bf X}$ be partial multifunctions.
We say that {\em $F$ is Weihrauch reducible to $G$} (written $F\leq_{\sf W}G$) if there are partial computable functions $\Phi$ and $\Psi$ such that, for any $x\in{\dom}(F)$, $y\in G(\Phi(x))$ implies $\Psi(x,y)\in F(x)$.
For basics on Weihrauch reducibility, we refer the reader to \cite{pauly-handbook}.
The definition of Weihrauch reducibility $F\leq_{\sf W}G$ can be viewed as the following perfect information two-player game:
\[
\begin{array}{rccccc}
{\rm I}\colon	& x_0\in{\dom}(F)	&		& x_1\in G(y_0)	&		\\
{\rm II}\colon	&		& y_0\in{\dom}(G)	&		& y_1\in F(x_0)	
\end{array}
\]

More precisely, each player chooses an element from ${\bf X}$ at each round, and Player II wins if there is a computable strategy $\tau$ for II which yields a play described above.
Note that $y_0$ depends on $x_0$, and $y_1$ depends on $x_0$ and $x_1$, and moreover, a computable strategy $\tau$ for II yields partial computable maps $\Phi\colon x_0\mapsto y_0$ and $\Psi\colon (x_0,x_1)\mapsto y_1$.
Usually, $\Phi$ is called an {\em inner reduction} and $\Psi$ is called an {\em outer reduction}.
If reductions $\Phi$ and $\Psi$ are allowed to be continuous (when ${\bf X}=\om^\om$), then we say that {\em $F$ is continuously Weihrauch reducible to $G$} (written $F\leqcW G$).

In (classical or intuitionistic) logic, when showing $P\to Q$, the contraction rule allows us to use the antecedent $P$ more than once.
However, the above reducibility notion essentially requires us to show $Q$ by using $P$ only once.
To overcome this difficulty, there is an operation $H\star G$ on partial multifunctions $G,H$ which allows us to use $H$ after the use of $G$.
For partial multi-valued functions $G$ and $H$, define the composition $H\circ G$ as follows:
\[H\circ G(x)=
\begin{cases}
\bigcup\{H(y):y\in G(x)\}&\mbox{ if }G(x)\downarrow\subseteq{\dom}(H),\\
\uparrow&\mbox{ otherwise}.
\end{cases}
\]

Note that the above definition of the composition of partial multifunctions is different from the usual definition of the composition of relations (so the partial multifunctions are completely different from the relations).
Then it is shown that $\min_{\leq_{\sf W}}\{H_0\circ G_0:G_0\leq_{\sf W}G\;\land\;H_0\leq_{\sf W}H\}$ always exists, and a representative is denoted by $H\star G$; see Brattka-Pauly \cite{paulybrattka4}.
There is an explicit description of a representative $H\star G$, and it is indeed needed to prove results on $\star$.
To explain the explicit definition, we consider the following play for ``$F\leq_{\sf W}H\star G$'':
\[
{\small
\begin{array}{rccccccc}
{\rm I}\colon	& x_0\in{\dom}(F)	&		& x_1\in G(y_0)	&		& x_2\in H(y_1) \\
{\rm II}\colon	&		& y_0\in{\dom}(G)	&		& y_1\in {\dom}(H)	& & y_2\in F(x_0)
\end{array}
}
\]

Again, a computable strategy for Player II yields functions $y_0=\Phi_0(x_0)$, $y_1=\Phi_1(x_0,x_1)$, and $y_2=\Psi(x_0,x_1,x_2)$.
That is, II's strategy yields two inner reductions $\Phi_0,\Phi_1$ and an outer reduction $\Psi$, where $\Phi_0$ makes a query to $G$, and then, after seeing a result $x_1\in G(y_0)$, the second reduction $\Phi_1$ makes a query to $H$.

Now, the partial multifunction $H\star G$ takes, as an input, Player I's first move $x_0$ and a code for Player II's strategy $\tau$ (a code for $\langle \Phi_0,\Phi_1,\Psi\rangle$) such that II's play according to the strategy $\tau$ obeys the rule unless I's play violates the rule at some previous round, where we say that Player I obeys the rule if $x_1\in G(y_0)$ and $x_2\in H(y_1)$, and Player II obeys the rule if $y_0\in{\dom}(G)$ and $y_1\in{\dom}(H)$.
Then, $H\star G(x_0,\tau)$ returns $\Psi(x_0,x_1,x_2)$ for I's some play $(x_1,x_2)$ which obeys the rule.
Note that there are many possible values for $\Psi(x_0,x_1,x_2)$, which means that $H\star G$ is multi-valued.
Without loss of generality, one can remove $x_0$ from the definition of $H\star G$ by considering a continuous strategy $\tau$.

This operation $(G,H)\mapsto H\star G$ is called the {\em compositional product}; see also Brattka-Pauly \cite{paulybrattka4}.
For game-theoretic descriptions, see also Hirschfeldt-Jockusch \cite{HiJo16} and Goh \cite{Goh19}.
This idea extends to reduction games introduced below.


\subsubsection{$\star$-closure and reduction games}\label{sec:def-game-Weihrauch}

In order to express ``arbitrary use of an antecedent,'' Hirschfeldt-Jockusch \cite{HiJo16} introduced {\em reduction games} and {\em generalized Weihrauch reducibility}.

\begin{definition}
For partial multifunctions $F,G\pcolon{\bf X}\rightrightarrows{\bf X}$, let us consider the following perfect information two-player game $\mathfrak{G}(F,G)$:
\[
\begin{array}{rccccccc}
{\rm I}\colon	& x_0	&		& x_1	&		& x_2	&	& \dots \\
{\rm II}\colon	&		& y_0	&		& y_1	& 		& y_2	& \dots
\end{array}
\]

Each player chooses an element from ${\bf X}$ at each round.
Here, Players I and II need to obey the following rules.
\begin{itemize}
\item First, Player I chooses $x_0\in{\dom}(F)$.
\item At the $n$th round, Player II reacts with $y_n=\langle a,u_n\rangle$.
\begin{itemize}
\item The choice $a=0$ indicates that Player II makes a new query $u_n$ to $G$.
In this case, we require $u_n\in{\dom}(G)$.
\item The choice $a=1$ indicates that Player II declares victory with $u_n$.
\end{itemize}
\item At the $(n+1)$th round, Player I responds to the query made by Player II at the previous stage.
This means that $x_{n+1}\in G(u_n)$.
\end{itemize}

Then, {\em Player II wins the game $\mathfrak{G}(F,G)$} if either Player I violates the rule before Player II violates the rule or Player II obeys the rule and declares victory with $u_n\in F(x_0)$.

Hereafter, we require that Player II's moves are chosen in a continuous manner (if ${\bf X}=\om^\om$).
In other words, Player II's strategy is a code $\tau$ of a partial {\em continuous} function $h_\tau\pcolon{\bf X}^{<\om}\to 2\times{\bf X}\simeq{\bf X}$.
On the other hand, Player I's strategy is any partial function $\sigma\pcolon{\bf X}^{<\om}\to{\bf X}$ (which is not necessarily continuous).
Given such strategies $\sigma$ and $\tau$ yield a play $\sigma\otimes\tau$ in the following manner:
\begin{align*}
(\sigma\otimes\tau)(0)&=\sigma(\langle\rangle),\\
(\sigma\otimes\tau)(2n+1)&=h_\tau(\langle(\sigma\otimes\tau)(2m)\rangle_{m\leq n}),\\
(\sigma\otimes\tau)(2n+2)&=\sigma(\langle(\sigma\otimes\tau)(2m+1)\rangle_{m\leq n}).
\end{align*}

Player II's strategy $\tau$ is {\em winning} if Player II wins along $\sigma\otimes\tau$ whatever Player I's strategy $\sigma$ is.
We say that $F$ is {\em game-Weihrauch reducible to $G$} (written $F\leq_{\gW}G$) if Player II has a computable winning strategy for $\mathfrak{G}(F,G)$.
We write $F\leq_{\gW}^cG$ if Player II has a (continuous) winning strategy for $\mathfrak{G}(F,G)$.
\end{definition}

Hirschfeldt-Jockusch \cite{HiJo16} showed that the relation $\leq_{\gW}$ is transitive.

\begin{remark}
In a traditional context, the notion $\leq_{\gW}$ is also called generalized Weihrauch reducibility.
However, this notion $\leq_{\gW}$ should be considered to be the most fundamental reducibility notion, which performs a relative computation using a finite number of queries to an oracle, and perhaps a {\em simpler name} would be preferable.

If $f,g$ are total functions on $\om$, then one can see that $f\leq_{\gW}g$ if and only if $f$ is Turing reducible to $g$.
For this reason, $\leq_{\gW}$ is called Turing reducibility in \cite{Kih23}.
\end{remark}

Note that the rule of the above game does not mention $F$ except for Player I's first move.
Hence, if we skip Player I's first move, we can judge if a given play follows the rule without specifying $F$.

\begin{definition}\label{def:relative-computation-F}
For a partial multifunction $G\pcolon{\bf X}\tto{\bf X}$, we define $G^{\clo}\pcolon{\bf X}\tto{\bf X}$ as follows:
\begin{itemize}
\item $(x_0,\tau)\in{\dom}(G^{\clo})$ if and only if $\tau$ is Player II's strategy, and for Player I's any strategy $\sigma$ with first move $x_0$, if Player II declares victory at some round along the play $\sigma\otimes\tau$.
\item Then, $u\in G^{\clo}(x_0,\tau)$ if and only if Player II declares victory with $u$ at some round along the play $\sigma\otimes\tau$ for some $\sigma$ with first move $x_0$.
\end{itemize}
\end{definition}

Here, the statement ``Player II declares victory'' does not necessarily mean ``Player II wins''.
Indeed, the above definition is made before $F$ is specified, so the statement ``Player II wins'' does not make any sense.
Again, one can remove $x_0$ from an input for $G^{\clo}$ by considering a continuous strategy. 
The following is obvious by definition.

\begin{obs}[see \cite{NePa18}]
$F\leq_{\sf W}G^{\clo}\iff F\leq_{\gW}G$.
\end{obs}

\begin{remark}
We may think of $G^{\clo}$ as a {\em universal machine} for $G$-relative computation.
The relative computation process is displayed as a game, a program is Player II's strategy, and an input is Player I's first move.
This means that $G^{\clo}(x_0,\sigma)$ is the result of the computation of a program $\sigma$ with an oracle $G$ and an input $x_0$.
Using the notation of traditional computability theory, it would be appropriate to write $\Phi_{\sigma}^G(x_0):=G^{\clo}(x_0,\sigma)$.
\end{remark}

Note that Neumann-Pauly \cite{NePa18} used generalized register machines to define this notion $G^\dia$.
Neumann-Pauly \cite{NePa18} showed that $F\leq_{\sf W}G^\dia$ if and only if $F\leq_{\gW}G$.
By transitivity of $\leq_{\gW}$, we have $F^{\clo}\star F^{\clo}\equiv_{\sf W}F^{\clo}$.
Moreover, Westrick \cite{westrick2020note} showed that $\dia$ gives the {\em least fixed point} of the operation $F\mapsto F\star F$ whenever $F\geq_{\sf W}\mathbf{1}$, where $\mathbf{1}$ is the Weihrauch degree of the identity function on ${\bf X}$.
More precisely, if $\mathbf{1}\leq_{\sf W} F$ then $F\leq_{\sf W}F\star F$ implies $F^{\clo}\leq_{\sf W}F$.
For more details on the closure operator $\dia$, see Westrick \cite{westrick2020note}.

\subsubsection{Tree representation of reducibility}
If the reader understand the definitions in Section \ref{sec:def-game-Weihrauch}, they may skip this section.
However, one may find the definition using games cumbersome, so let us also provide a definition using a tree.
If a strategy of one Player ${\sf A}$ is fixed, the possible moves of the other Player ${\sf B}$ form a tree.
Here, the moves of Player ${\sf A}$ following this fixed strategy are labeled on each node.

\begin{definition}[via tree]\label{def:strategy-tree}
%
A labeled well-founded tree is a pair of a well-founded tree $T\subseteq{\bf X}^{<\om}$ and a function $\nu\colon T\to{\bf X}$.
A {\em $G$-strategy tree} is a labeled well-founded tree $(T,\nu)$ such that, either $\sigma\in T$ is a leaf, or else, $\nu(\sigma)\in{\rm dom}(G)$, and the set of all immediate successors of $\sigma$ is given by $G(\nu(\sigma))$; that is, $G(\nu(\sigma))=\{x\in{\bf X}^{<\om}:\sigma x\in T\}$.
%
\end{definition}

Alternatively, one can inductively define the collection $\mathcal{S}$ of {\em $G$-strategy trees} as follows:
\begin{enumerate}
\item Each $x\in{\bf X}$ is in $\mathcal{S}$; where $x$ represents a singleton-tree $\{\ep\}$ labeled by $\nu(\ep)=x$.
\item Each $(z,\tau)\in\sum_{x\in\dom(G)}\mathcal{S}^{G(x)}$ is in $\mathcal{S}$, where $(z,\tau)$ represents a labeled tree $(T,\nu)$ such that the root $\ep\in T$ is labeled by $z$, and for $y\in G(z)$ and $\tau(y)=(T_y,\nu_y)$, if $\sigma\in T_y$ then $y\fr\sigma\in T$ and $\nu(y\fr\sigma)=\nu_y(\sigma)$.
\end{enumerate}

\begin{definition}\label{def:legal-tree}
A strategy is a pair $(\hat{t},\hat{\nu})$ of partial continuous functions $\hat{\nu}\pcolon{\bf X}^{<\om}\to{\bf X}$ and $\hat{t}\pcolon{\bf X}^{<\om}\to 2$.
A strategy $(\hat{t},\hat{\nu})$ is {\em $G$-legal} if there is a (unique) $G$-strategy tree $(T,\nu)$ such that $\nu$ is a restriction of $\hat{\nu}$, $T\subseteq{\rm dom}(\hat{t})$ and for any $\sigma\in T$, $\hat{t}(\sigma)=1$ if and only if $\sigma$ is a leaf.
\end{definition}

Note that $(\hat{t},\hat{\nu})$ corresponds to Player II's strategy in game-theoretic terms, and $(T,\nu)$ corresponds to the tree of all possible plays by Player I with respect to Player II's fixed strategy $(\hat{t},\hat{\nu})$.
Below, we write $f_x(\alpha)=f(x\fr\alpha)$.

\begin{obs}
Let $F,G\pcolon{\bf X}\tto{\bf X}$ be partial multifunctions.
Then $F\leq_{\gW}G$ iff there is a computable strategy $(\hat{t},\hat{\nu})$ such that, for any $x\in{\rm dom}(F)$, $(\hat{t}_x,\hat{\nu}_x)$ is $G$-legal, and each root of the corresponding $G$-strategy tree $(T_x,\nu_x)$ is labeled by some $y\in F(x)$.
\end{obs}


\begin{remark}
The idea of describing $\gW$-reducibility via the inductive definition of trees is close to the definition of what Lee-van Oosten \cite{Lee,LvO13} called a sight.
They introduced this notion of reducibility in a different context from Weihrauch reducibility, but it is interesting to note that they provided this definition earlier than Hirschfeldt-Jockusch \cite{HiJo16} and Neumann-Pauly \cite{NePa18}.

However, Lee-van Oosten's original definition is given under a more extended setting and is therefore much more complicated than what is described here.
For this reason, its relevance to Weihrauch reducibility was not clear at all.
In fact, until pointed out by the author \cite{Kih23}, it had not been discovered that their definition could be linked to the context of Weihrauch reducibility.
\end{remark}

\subsubsection{Recursion trick}\label{sec:recursion-trick}

To prove various separation results on Weihrauch reducibility, {\em Kleene's recursion theorem} has been frequently used as a simple, but very strong, proof machinery, cf.~\cite{HiJo16,pauly-kihara2-mfcs,KiPa19}.
We employ this machinery throughout this article, so we separate this argument, and call it {\em recursion trick}.
In Kihara-Pauly \cite{pauly-kihara2-mfcs}, recursion trick is described in the context of two player games, where Player I is called {\sf Pro} (the proponent) and II is called {\sf Opp} (the opponent) because our purpose is to show Weihrauch separation $\not\leq_{\sf W}$.
The formal description of recursion trick for short reduction games is described in Kihara-Pauly \cite{KiPa19}.

We first describe the idea of recursion trick in a slightly longer reduction game, i.e., for $F\leq_{\sf W}G\star H$.
Note that, in this kind of reduction, any play of the reduction game ends at the third round of Player II:
\[
\begin{array}{rcccccc}
{\rm I}\colon	& x_0	&		& x_1	&		& x_2	&	\\
{\rm II}\colon	&		& y_0	&		& y_1	&		& y_2
\end{array}
\]

If Player II's strategy is given, we may consider that the value of $y_i$ depends on Player I's play; that is, $y_0=y_0(x_0)$, $y_1=y_1(x_0,x_1)$, and $y_2=y_2(x_0,x_1,x_2)$.
If we describe an algorithm constructing $x_0'$ from $(y_0,y_1,y_2)$ (with a parameter $x_0$), by Kleene's recursion theorem, the parameter $x_0$ can be interpreted as {\em self-reference}; that is, we may assume that $x_0'=x_0$.
Hence, one can remove $x_0$ from the list of parameters in $y_0$, $y_1$ and $y_2$.
In summary, to show that $F\not\leq_{\sf W}G\star H$, given $y_0,y_1,y_2$, it suffices to construct $x_0$ such that Player I wins along the following play
\[
\begin{array}{rcccccc}
{\rm I}\colon	& x_0	&		& x_1	&		& x_2	&	\\
{\rm II}\colon	&		& y_0	&		& y_1(x_1)	&		& y_2(x_1,x_2)
\end{array}
\]
for some $x_1$ and $x_2$, that is, Player II violates the rule before I violates the rule, or Player I obeys the rule and $y_2\not\in F(x_0)$.
The essence of recursion trick is that {\em it allows us to construct Player I's first move $x_0$ later than II's moves $y_0,y_1,y_2$}.
However, note that I's move $x_0$ has to obey the rule even if II violates the rule.

In general reduction games, II's strategy yields $y_n(x_0,x_1,\dots,x_n)$.
Again we describe an algorithm $(y_n)_{n\in\om}\mapsto x_0'$, and by Kleene's recursion theorem, consider $x_0$ as self-reference.
As before, this allows us to construct Player I's first move $x_0$ later than II's moves $(y_n)_{n\in\om}$, but I's move $x_0$ has to obey the rule even if II violates the rule.

\subsubsection*{Formal description of recursion trick:}
For ${\bf X}=\om^\om$, we rigorously describe what recursion trick is.
Fix an oracle $Z$. 
Given Player II's $Z$-computable strategy $\tau$, we get a uniform sequence of $Z$-computable functions $y^\tau_n(x_0,x_1,\dots,x_n)$ which determine the $n$th move of Player II.
Now, assume that Player I's first move is the $e$th partial $Z$-computable function $\varphi_e^Z$ on $\om$.
Then, the reduction game proceeds as follows:
\[
\begin{array}{rccccccc}
{\rm I}\colon	& \varphi_e^Z	&		& x_1	&		& x_2	&	& \dots \\
{\rm II}\colon	&		& y_0^\tau(\varphi_e^Z)	&		& y_1^\tau(\varphi_e^Z,x_1)	&		& y_2^\tau(\varphi_e^Z,x_1,x_2) &\dots
\end{array}
\]

Here, each move in the above play may be a {\em partial} function on $\om$, while each move in the original reduction game is a {\em total} function on $\om$ (i.e., an element of ${\bf X}=\om^\om$).

Assume that we have constructed a $Z$-computable algorithm $e\mapsto x_0(e)$ constructing Player I's first move $x_0(e)$ using information of II's play $\langle\tau(\varphi_e^Z,x_1,\dots,x_n)\rangle_{n\in\om}$.
Clearly, there is a $Z$-computable function $u$ such that $x_0(e)\simeq\varphi_{u(e)}^Z$.
Moreover, assume that
\begin{itemize}
\item[($\star$)]
The move $x_0(e)$ is designed to defeat Player II's strategy $\langle\tau(\varphi_e^Z,x_1,\dots,x_n)\rangle_{n\in\om}$; that is, for any $e$, there exist subsequent moves $(x_1,x_2,\dots)$ such that Player I wins along the following play:
\[
\begin{array}{rccccccc}
{\rm I}\colon	& x_0(e)\simeq\varphi_{u(e)}^Z	&		& x_1	&		& x_2	&	& \dots \\
{\rm II}\colon	&		& y_0^\tau(\varphi_e^Z)	&		& y_1^\tau(\varphi_e^Z,x_1)	&		& y_2^\tau(\varphi_e^Z,x_1,x_2) &\dots
\end{array}
\]
\end{itemize}

Here, we do not need to find such $(x_1,x_2,\dots)$ effectively.
By Kleene's recursion theorem, there exists $r$ such that $x_0(r)\simeq\varphi_{u(r)}^Z\simeq\varphi_r^Z$.
Note that, in the above play with $e=r$, Player II just follows the strategy $\tau$.
By the assumption ($\star$), I's play $(x_0(r),x_1,x_2,\dots)$ defeats Player II who follows $\tau$, which means that $\tau$ is not a winning strategy.

In summary, if for any computable strategy $\tau$ for Player II (in the reduction game $\mathfrak{G}(F,G)$) one can construct a computable algorithm $e\mapsto x_0(e)$ satisfying ($\star$), then we obtain $F\not\leq_{\sf W}G^{\clo}$.
Moreover, if this is possible relative to any oracle $Z$, then we also obtain the continuous Weihrauch separation $F\not\leq_{\sf W}^cG^{\clo}$.
We often describe such a construction $e\mapsto x_0(e)$ as an algorithm describing $x_0(\tilde{\tau})$ from given $\langle\tilde{\tau}(x_1,\dots,x_n)\rangle$, where $\tilde{\tau}(x_1,\dots,x_n)\simeq\tau(\varphi_e^Z,x_1,\dots,x_n)$ for some $e$.

%
%

\subsection{Realizability}

In this section, we construct a realizability predicate from a given Weihrauch degree, and moreover, if a given Weihrauch degree has a good closure property, then we show that the induced realizability predicate validates all axioms of intuitionistic Zermelo-Fraenkel set theory ${\bf IZF}$.

\subsubsection{Partial combinatory algebra}\label{sec:realizability-pca}

In order to give a computability-theoretic model of constructive mathematics, it is convenient to abstract the property of computability as a kind of algebra.

\begin{definition}[see also \cite{vOBook}]
A partial magma is a pair $(M,\ast)$ of a set $M$ and a partial binary operation $\ast$ on $M$.
We often write $xy$ instead of $x\ast y$, and as usual, we consider $\ast$ as a left-associative operation, that is, $xyz$ stands for $(xy)z$.
A partial magma is {\em combinatory complete} if, for any term $t(x_1,x_2,\dots,x_n)$, there is $a_t\in M$ such that $a_tx_1x_2\dots x_{n-1}\downarrow$ and $a_tx_1x_2\dots x_n\simeq t(x_1,x_2,\dots,x_n)$.
For terms $t(x,y)=x$, and $u(x,y,z)=xz(yz)$, the corresponding elements $a_t,a_u\in M$ are usually written as ${\sf k}$ and ${\sf s}$.
A combinatory complete partial magma is called a {\em partial combinatory algebra} (abbreviated as {\em pca}).
\end{definition}

\begin{example}[The algebra of computable functions]\label{exa:first-algebra-definition}
Kleene's first algebra ${\sf K}_1=(\om,\ast)$ is defined by $e\ast n=\varphi_e(n)$, where $\varphi_e$ is the $e$th partial computable function.
Then ${\sf K}_1$ forms a partial combinatory algebra.
The combinatory completeness follows from the so-called smn-theorem.
\end{example}

\begin{example}[The algebra of continuous functions]\label{exa:second-algebra-definition}
Kleene's second algebra ${\sf K}_2=(\om^\om,\ast)$ is defined by $x\ast y=\varphi_e^z(y)$, where $x=e\fr z$ and $\varphi_e$ is the $e$th partial computable function relative to $z\in\om^\om$.
Then ${\sf K}_2$ forms a partial combinatory algebra.
\end{example}

However, this notion is not sufficient for our purposes.
The most basic setting in modern computability theory is not ``computable objects $+$ computable morphisms,'' nor ``topological objects $+$ topological morphisms,'' but ``topological objects $+$ computable morphism \cite{Bau00}.''

\begin{definition}[see also {\cite[Sections 2.6.9 and 4.5]{vOBook}}]
A {\em relative pca} is a triple $\mathbb{P}=({\tt P},\tpbf{P},\ast)$ such that ${\tt P}\subseteq\tpbf{P}$, both $(\tpbf{P},\ast)$ and $({\tt P},\ast\upto {\tt P})$ are pcas, and share combinators ${\sf s}$ and ${\sf k}$.
\end{definition}

In this article, the boldface algebra $\tpbf{P}$ is always the set $\om^\om$ of all infinite sequences.

\begin{example}[The algebra of type-two computable functions]
Kleene-Vesley's algebra is ${\sf K}_2^{\sf eff}=((\om^\om)^{\sf eff},\om^\om,\ast)$, where $\ast$ is the same as ${\sf K}_2$ in Example \ref{exa:second-algebra-definition}, and $(\om^\om)^{\sf eff}$ is the set of all computable sequences.
Then ${\sf K}_2^{\sf eff}$ forms a relative pca.

Note that the functions of the forms $x\mapsto e\ast x$ for some $e\in(\om^\om)^{\sf eff}$ are exactly partial computable functions on $\om^\om$.
\end{example}

\begin{example}
In descriptive set theory, the idea of a relative pca is ubiquitous, which usually occurs as a pair of {\em lightface} and {\em boldface} pointclasses.
A large number of nontrivial, deep, examples of relative pcas have been (implicitly) studied in descriptive set theory \cite{MosBook}:
By the {\em good parametrization lemma} in descriptive set theory (see Moschovakis \cite[Lemma 3H.1]{MosBook}), any $\Sigma^\ast$-pointclass $\Gamma$ --- in particular, any Spector pointclass (see Moschovakis \cite[Lemma 4C]{MosBook}), so many infinitary computation models including infinite time Turing machines (ITTMs) --- yields a (relative) pca, since the good parametrization lemma is, roughly speaking, a generalized smn-theorem, so such a pointclass admits currying.
Here, the {\em partial $\Gamma$-computable function application} form a lightface pca, and the {\em partial $\tpbf{\Gamma}$-measurable function application} form a boldface pca.
For the details, see \cite{Kih22}.
\end{example}

\begin{example}
Let us give some computability/descriptive-set-theoretic examples of relative pcas.

\begin{enumerate}
\item If $\Gamma=\Sigma^0_1$, the induced lightface pca is equivalent to {\em Kleene's first algebra} (associated with Kleene's number realizability), and the boldface pca is {\em Kleene's second algebra} (associated with Kleene's functional realizability).
The relative pca induced from $\Gamma=\Sigma^0_1$ is Kleene-Vesley's algebra.

\item The pointclass $\Sigma^0_1(\emptyset^{(n)})$ is employed by Akama et al.~\cite{Aka} to show that the arithmetical hierarchy of the law of excluded middle does not collapse over Heyting arithmetic ${\bf HA}$.
Here, note that $\Sigma^0_n$ does not yield a pca on $\om^\om$ since $\Sigma^0_n$-computable functions (on $\om^\om$) and $\tpbf{\Sigma}^0_n$-measurable functions are not closed under composition.
In descriptive set theory, Borel pointclass whose associated measurable functions are closed under composition are studied under the name of Borel amenability, cf.~\cite{MRos09}; for instance, the $n$th level Borel functions (a.k.a.~$\Sigma_{n,n}$-functions) are closed under composition, and the induced reducibility notion is studied, e.g.~by \cite{KiSe19}.

\item The pointclass $\Pi^1_1$ is the best-known example of a Spector pointclass, and the induced lightface pca obviously yields hyperarithmetical realizability.
For the boldface pca, the associated total realizable functions are exactly the Borel measurable functions.
If one considers Kleene realizability with such a pca (i.e., one obtained by a Spector pointclass), then it trivially validates $\Sigma^0_n\lemr$ for any $n\in\om$.

\item Bauer \cite{Bau} also studied the pca (and the induced realizability topos) obtained from infinite time Turing machines (ITTMs).
Here, recall that ITTMs form a Spector pointclass.
\end{enumerate}

One may also consider descriptive set theoretic versions of Scott's graph model using lightface/boldface pointclasses.
\end{example}

\subsubsection{Realizability interpretation}\label{sec:realizability-interpretation-def}

One of the most important interpretations of constructive mathematics is the realizability interpretation; see e.g.~\cite{Tr98}.
The most basic Kleene's realizability interpretation is an interpretation of constructive mathematics using partial computable functions on $\om$.
Or, it can be said to be an interpretation using ${\sf K}_1$ (Definition \ref{exa:first-algebra-definition}).
This idea is generalized to an interpretation using arbitrary relative pca.

\begin{definition}\label{def:realizability-relative-to-multifunction}
Let $\mathbb{P}=({\tt P},\tpbf{P},\ast)$ be a relative pca.
Given $a\in\tpbf{P}$ and a formula $\varphi$, the realizability relation $a\Vdash_\mathbb{P}\varphi$ is inductively defined as follows:
\begin{align*}
a\Vdash_\mathbb{P} A&\iff A\mbox{ is true \quad(for atomic $A$).}\\
(a,b)\Vdash_\mathbb{P}\varphi\land\psi&\iff a\Vdash_\mathbb{P}\varphi\mbox{ and }b\Vdash_\mathbb{P}\psi.\\
(i,a)\Vdash_\mathbb{P}\varphi\lor\psi&\iff\mbox{if $i=0$ then $a\Vdash_\mathbb{P}\varphi$ else $a\Vdash_\mathbb{P}\psi$.}\\
e\Vdash_\mathbb{P}\varphi\to\psi&\iff\forall a\in\tpbf{P}\;[a\Vdash_\mathbb{P}\varphi\ \longrightarrow\ ea\Vdash_\mathbb{P}\psi].\\
({\tt x},a)\Vdash_\mathbb{P}\exists x\in X.\varphi(x)&\iff a\Vdash_\mathbb{P}\varphi(x)\mbox{, where ${\tt x}$ is a code of $x\in X$}.\\
e\Vdash_\mathbb{P}\forall x\in X.\varphi(x)&\iff\forall {\tt x}\in\tpbf{P}\;[\mbox{${\tt x}$ is a code of $x\in X$}\;\longrightarrow\;e{\tt x}\Vdash_\mathbb{P}\varphi(x)].
\end{align*}

Then $\varphi$ is $\mathbb{P}$-realizable if $a\Vdash_{\mathbb{P}}\varphi$ for some $a\in{\tt P}$.
Here, we assume that each type $X$ is interpreted as a coded set (an assembly, a represented space, etc.)
\end{definition}

\begin{example}
This notion can clearly be relativized to an oracle.
For instance, if an oracle $\alpha\colon\om\to\om$ is given, instead of the pca ${\sf K}_1=(\om,\ast)$ of computable functions, one can consider the pca $(\om,\ast_\alpha)$ of $\alpha$-computable functions, where $\ast_\alpha$ is the $\alpha$-computable function application.
In general, any relative pca $\mathbb{P}$ can be relativized to a partial oracle $\alpha\pcolon\tpbf{P}\to\tpbf{P}$; see \cite{vO06}.
\end{example}

However, most of the $\forall\exists$-theorems in mathematics are not single-valued (that is, a witness for the existence is not unique).
Therefore, what we really want to consider is the realizability interpretation relative to a partial {\em multi-valued} oracle $\alpha\pcolon\tpbf{P}\tto\tpbf{P}$.

This goes beyond the framework of relative pca:
It is not possible to combine the pair of a relative pca and a multi-valued oracle into a single relative pca, but that is fine.
Rather than collapsing a computation system and an oracle, it is sometimes better to consider them separately for clarity.

We declare that the pair $(\mathbb{P},F)$ of a relative pca $\mathbb{P}$ and a partial multi-valued function $F\pcolon \tpbf{P} \tto \tpbf{P}$ is the fundamental notion.
When the oracle part $F$ satisfies certain conditions, we call $(\mathbb{P},F)$ a {\em site} (see also Definition \ref{def:site}).

We use the notation for relative computation as close to that of ordinary computability theory, as: $\Phi_e^F(x_0)=F^{\clo}(x_0,e)$.
The realizability interpretation is relativized to a multi-valued oracle $F$ as follows.

\begin{definition}\label{def:realizability-relative-to-multifunction2}
Let $\mathbb{P}$ be a relative pca, and $F\pcolon\tpbf{P}\tto\tpbf{P}$ be a partial multifunction.
The realizability relation $a\Vdash_F\varphi$ is inductively defined as follows:
\begin{align*}
e\Vdash_F\varphi\to\psi&\iff\forall a\;[a\Vdash_F\varphi\ \longrightarrow\ {\Phi_e^F(a)\downarrow}
\mbox{ and }(\forall b\in\Phi_e^F(a))\;b\Vdash_F\psi].\\
e\Vdash_F\forall x\in X.\varphi(x)&\iff\forall {\tt x}\;[\mbox{${\tt x}$ is a code of $x\in X$}\\
&\qquad\qquad\longrightarrow\;{\Phi_e^F({\tt x})\downarrow}\mbox{ and }\forall b\in\Phi_e^F({\tt x}).\;b\Vdash_F\varphi({\tt x})].
\end{align*}

For other logical connectives, the definitions are the same as in Definition \ref{def:realizability-relative-to-multifunction}.
\end{definition}

\begin{obs}\label{obs:same-Weihrauch-same-realizability}
$F\equiv_{\gW}G$ then $F$-realizability and $G$-realizability are equivalent.
\end{obs}

Let us connect the relativization to a multivalued oracle with a known realizability interpretation.
If we define ${\bf j}=F^{\clo}$, Definition \ref{def:realizability-relative-to-multifunction2} can be rephrased as follows:
\begin{align*}
e\Vdash_{\bf j}\varphi\to\psi&\iff\forall a\;[a\Vdash_{\bf j}\varphi\ \longrightarrow\ {{\bf j}(e,a)\downarrow}
\mbox{ and }(\forall b\in{\bf j}(e,a))\;b\Vdash_{\bf j}\psi].\\
e\Vdash_{\bf j}\forall x\in X.\varphi(x)&\iff\forall {\tt x}\;[\mbox{${\tt x}$ is a code of $x\in X$}\\
&\qquad\qquad\longrightarrow\;{{\bf j}(e,{\tt x})\downarrow}\mbox{ and }\forall b\in{\bf j}(e,{\tt x}).\;b\Vdash_{\bf j}\varphi({\tt x})].
\end{align*}

We call this the {\em realizability interpretation over the site $(\mathbb{P},{\bf j})$.}
Some readers may recognize a similar realizability notion.
That is Lifschitz' realizability.

\begin{example}[Lifschitz realizability]
The key idea in Lifschitz \cite{Lif} (and van Oosten \cite{vO}) of validating ${\sf CT}_0!+\neg{\sf CT}_0$ is the use of {\em multifunction applications}, rather than single-valued applications.
From the computability-theoretic perspective, Lifschitz' key idea is regarding {\em $\Pi^0_1$ sets} (or {\em $\Pi^0_1$ classes}) as basic concepts rather than computable functions.
More precisely, over the Kleene first algebra $(\om,\ast)$, Lifschitz considered the partial multifunction $\jump_{\rm L}\pcolon\om\tto\om$ defined by
\[\jump_{\rm L}(\langle e,b\rangle)=\{n\in\om:n<b\;\land\;e\ast n\uparrow\},\]
where $\jump_{\rm L}(\langle e,b\rangle)\downarrow$ if and only if the set is nonempty.
Obviously, Lifschitz's multifunction $\jump_{\rm L}$ gives an effective enumeration of all bounded $\Pi^0_1$ subsets of $\om$.
Van Oosten \cite{vO} extended this notion to the Kleene second algebra $(\om^\om,\ast)$ by considering the following partial multifunction $\jump_{\rm vO}\pcolon\om^\om\tto\om^\om$:
\[\jump_{\rm vO}(\langle g,h\rangle)=\{x\in\om^\om:(\forall n\in\om)\;x(n)<h(n)\;\land\;g\ast x\uparrow\},\]

It is again obvious that van Oosten's multifunction $\jump_{\rm vO}$ gives a representation of all compact subsets of $\om^\om$, and $\jump_{\rm vO}(\langle g,h\rangle)$ is an effectively compact $\Pi^0_1$ class relative to $\langle g,h\rangle$.

Compare their realizability interpretations \cite{Lif,vO} with ours.
One can see that Lifschitz' realizability \cite{Lif} is equivalent to the realizability interpretation over the site $({\sf K}_1,\jump_{\rm L})$, and van Oosten's realizability \cite{vO} is equivalent to the realizability interpretation over the site $({\sf K}_2^{\sf eff},\jump_{\rm vO})$.
\end{example}

\begin{remark}
Lifschitz' realizability and van Oosten's realizability can also be given as the realizability interpretations relative to the weakest oracles that solves ${\sf LLPO}$ and ${\sf WKL}$, respectively.
As a logical principle, ${\sf LLPO}$ is de Morgan's law for $\Sigma_1$ formulas:
\[\neg(\varphi\land\psi)\longrightarrow\neg\varphi\lor\neg\psi\qquad(\varphi,\psi\colon\mbox{$\Sigma_1$ formulas})\]

As a problem, ${\sf LLPO}$ is a partial multifunction on $\om$ such that
\begin{enumerate}
\item $e\in{\rm dom}({\sf LLPO})$ iff either $\varphi_e(0)\uparrow$ or $\varphi_e(1)\uparrow$.
\item $i\in{\sf LLPO}(e)$ iff $i<2$ and $\varphi_e(i)\uparrow$.
\end{enumerate}

For computability-theorists, it is obvious that ${\sf LLPO}^{\clo}\equiv_{\sf W}\jump_{\rm L}$ (see e.g.~\cite{PaTs16,Yos2}).
Thus, by Observation \ref{obs:same-Weihrauch-same-realizability}, Lifschitz realizability is equivalent to ${\sf LLPO}$-realizability; that is, realizability over the site $({\sf K}_1,{\sf LLPO}^{\clo})$.

As a mathematical principle, weak K\"onig's lemma ${\sf WKL}$ states that every infinite binary tree has an infinite path.
As a problem, ${\sf WKL}$ is a partial multifunction on $\om^\om$ such that
\begin{enumerate}
\item $T\in{\rm dom}({\sf WKL})$ iff $T\subseteq 2^{<\om}$ and $T$ is an infinite tree.
\item $p\in{\sf WKL}(T)$ iff $p$ is an infinite path through $T$; that is, $p\upto n\in T$ for any $n\in\om$.
\end{enumerate}

Again, it is easy to see that ${\sf WKL}^{\clo}\equiv_{\sf W}{\sf WKL}\equiv_{\sf W}\jump_{\rm vO}$ (see e.g.~\cite[Proposition 4.9]{HiJo16}).
Thus, by Observation \ref{obs:same-Weihrauch-same-realizability}, van Oosten's version of Lifschitz realizability is equivalent to ${\sf WKL}$-realizability; that is, realizability over the site $({\sf K}_2^{\sf eff},{\sf WKL}^{\clo})$.
This has also been independently pointed out in \cite{Yos2}.
\end{remark}

\begin{remark}
After finding Definition \ref{def:realizability-relative-to-multifunction2}, one might mistakenly assume that it is a straightforward generalization of Lifschitz's realizability \cite{Lif}, but it is not so simple at all.
It is essential to recognize Lifschitz's realizability as a realizability interpretation relative to a multi-valued oracle with an appropriate reducibility notion for multifunctions.
However, the discovery of the correct reducibility notion (that is, the correct definition of $\Phi_e^F$ for a partial multifunction $F$) had to wait until Hirschfeldt-Jockusch \cite{HiJo16} and Neumann-Pauly \cite{NePa18}.
This was not accessible in Lifschitz's time (until recently, in fact).
\end{remark}

\begin{remark}
Yoshimura \cite{Yos2} came closest to Definition \ref{def:realizability-relative-to-multifunction2}, but he did not have the definition of relative computation $\Phi_e^F$.
He had arrived at the definition of Lifschitz-like realizability interpretation $\Vdash_\jump$ over $(\mathbb{P},\mathbf{j})$ (although he did not use the term ``site'') independently of us.

We must be careful to note that there is a big difference between oracle realizability $\Vdash_F$ and Lifschitz-like realizability $\Vdash_{\jump}$ (if $\jump$ is not of the form $F^{\clo}$).
The former uses the relative computation allowing an arbitrary finite number of accesses to the oracle $F$, while the latter only accesses the oracle $\jump$ once.
In research in this field, overcoming this gap was not at all easy.
For this reason, Yoshimura \cite{Yos2} formulated an interpretation of finite-type {\em modal} Heyting arithmetic, where the modal operator $\diamondsuit$ is interpreted by a single query to a (multi-valued) oracle.
\end{remark}

\subsubsection{Set theory}

McCarty \cite{McPhD85} introduced a realizability interpretation of the intuitionistic Zermero-Fraenkel set theory ${\bf IZF}$ which validates Church's thesis ${\sf CT}_0$.
Chen-Rathjen \cite{ChRa12} combined McCarty-realizability with Lifschitz-realizability \cite{Lif} to give a realizability interpretation of ${\bf IZF}$ which validates ${\sf CT}_0!+\neg{\sf CT}_0+{\sf LLPO}+\neg{\sf LPO}+\neg{\sf AC}_{\om,2}$.

To define a set-theoretic realizability interpretation, we follow the argument in Chen-Rathjen \cite{ChRa12}:
First, we consider a set-theoretic universe with urlements $\N$.
For a relative pca $\mathbb{P}=({\tt P},\tpbf{P})$, as in the usual set-theoretic forcing argument, we consider a {\em $\mathbb{P}$-name}, which is any set $x$ satisfying the following condition:
\[x\subseteq\{\sspair{p}{u}:\mbox{$p\in\tpbf{P}$ and ($u\in\N$ or $u$ is a $\mathbb{P}$-name)}\}.\]

Roughly speaking, $(p,u)\in x$ indicates that $p$ is a witness for ``$u\in x$.''
The $\mathbb{P}$-names are used as our universe.
To be precise, we get the cumulative hierarchy of $\mathbb{P}$-names as follows:
\[V_0^\mathbb{P}=\emptyset,\quad V_\alpha^\mathbb{P}=\bigcup_{\beta<\alpha}\mathcal{P}(\tpbf{P}\times(V_\beta^\mathbb{P}\cup\N)),\mbox{ and }V^\mathbb{P}_{\rm set}=\bigcup_{\alpha\in{\rm Ord}}V_\alpha^\mathbb{P}.\]

Note that the urelements $\N$ are disjoint from $V^\mathbb{P}_{\rm set}$.
We define $V^\mathbb{P}=V^\mathbb{P}_{\rm set}\cup\N$.
McCarty \cite{McPhD85} used this notion to give a Kleene-like realizability interpretation for ${\bf IZF}$.
Then, as mentioned above, Chen-Rathjen \cite{ChRa12} incorporated Lifschitz's multifunction $\jump_{\rm L}$ into McCarty realizability.
We now generalize Chen-Rathjen's realizability predicate to {\em any} partial multifunction $\jump$.

\begin{definition}\label{def:Chen-Rathjen-realizability}
Fix a partial multifunction $\jump\pcolon\tpbf{P}\tto\tpbf{P}$.
For $e\in\tpbf{P}$ and a sentence $\varphi$ of ${\bf IZF}$ from parameters from $V^\mathbb{P}$, we define a relation $e\reap\varphi$.
For relation symbols:
\begin{align*}
e\reap R(\bar{a})&\iff\N\models R(\bar{a})\\ 
e\reap\bfN(a)&\iff a\in\N\;\;\&\;\;e=\ul{a}\\
e\reap\set(a)&\iff a\in V^\mathbb{P}_{\rm set}
\end{align*}
where $R$ is a primitive recursive relation.
For set-theoretic symbols:
\begin{align*}
e\reap a\in b\iff\;&\foralle d\in\jump(e)\exists c\;[\sspair{\pi_0d}{c}\in b\;\land\;\pi_1d\reap a=c]\\
e\reap a=b\iff\;&(a,b\in\N\,\land\, a=b)\lor(\set(a)\,\land\,\set(b)\,\land\\
&\foralle d\in\jump(e)\forall p,c\;[\sspair{p}{c}\in a\;\to\;\pi_0dp\reap c\in b]\;\land\\
&\foralle d\in\jump(e)\forall p,c\;[\sspair{p}{c}\in b\;\to\;\pi_1dp\reap c\in a].
\end{align*}

Here, we write $\foralle d\in\jump(e)$ if $e\in{\rm dom}(\jump)$ and $\forall d\in\jump(e)$.
For logical connectives:
\begin{align*}
e\reap A\land B&\iff \pi_0e\reap A\;\land\;\pi_1e\reap B\\
e\reap A\lor B&\iff\foralle d\in\jump(e)\;[(\pi_0d=0\land\pi_1d\reap A)\lor(\pi_0d=1\land\pi_1d\reap B)]\\
e\reap \neg A&\iff(\forall a\in\tpbf{P})\;a\not\reap A\\
e\reap A\to B&\iff(\forall a\in\tpbf{P})\;[a\rea A\;\rightarrow\;ea\reap B].
\end{align*}

For quantifiers:
\begin{align*}
e\reap \forall xA&\iff(\foralle d\in\jump(e))(\forall c\in V^\mathbb{P})\;e\reap A[c/x]\\
e\reap \exists xA&\iff(\foralle d\in\jump(e))(\exists c\in V^\mathbb{P})\;e\reap A[c/x].
\end{align*}

Then we say that a formula $\varphi$ is {\em $\jump$-realizable over $\mathbb{P}$} if there is $e\in {\tt P}$ such that $e\reap \varphi$.
\end{definition}

Unfortunately, a $\jump$-realizability predicate does not necessarily validate axioms of ${\bf IZF}$.
Next, we will discuss what condition for a multifunction $\jump$ is needed in order to ensure all axioms of ${\bf IZF}$.

\subsubsection{$j$-operator}
The difficulty with Lifschitz-like realizability is that it only makes one access to an oracle $\jump$.
Therefore, for this to behave correctly as oracle realizability, the condition that one access and any finite number of accesses to $\jump$ are equivalent must be satisfied.
This leads to the following notion:
\begin{definition}\label{def:j-operator}
Let $\mathbb{P}=({\tt P},\tpbf{P},\ast)$ be a relative pca.
We say that a partial multifunction $\jump\pcolon\tpbf{P}\tto\tpbf{P}$ is a {\em $j$-operator} on $\mathbb{P}$ if
\begin{enumerate}
\item There is ${\tt u}\in {\tt P}$ such that for any $a,x\in\tpbf{P}$, $a\jump(x)=\jump({\tt u}ax)$.
\item There is $\eta\in {\tt P}$ such that for any $x\in\tpbf{P}$, $x=\jump(\eta x)$.
\item There is $\mu\in {\tt P}$ such that for any $x\in\tpbf{P}$, $\jump\jump(x)=\jump(\mu x)$.
\end{enumerate}

Here, $\jump\jump$ is an abbreviation for the composition $\jump\circ\jump$, where recall the definition of the composition of multifunctions from Section \ref{subsec:weihrauch-reducibility}.
Also, if $F$ is a multifunction on $\tpbf{P}$ and $a,x\in\tpbf{P}$, then define $aF(x)=\{ay:y\in F(x)\}$.
\end{definition}

\begin{definition}\label{def:site}
A pair $(\mathbb{P},{\bf j})$ of a relative pca $\mathbb{P}$ and a $j$-operator ${\bf j}$ is called a {\em site}.
\end{definition}

\begin{definition}\label{def:trackability}
Let $(\mathbb{P},{\bf j})$ be a site.
A partial multifunction $f\pcolon\tpbf{P}\tto\tpbf{P}$ is {\em $(\mathbb{P},\jump)$-trackable} if there is $a\in\tpbf{P}$ such that, for any $x\in{\dom}(f)$, we have $\jump(ax)\subseteq f(x)$.
If $a\in{\tt P}$, we say that $f$ is {\em $(\mathbb{P},\jump)$-computable}.
\end{definition}

\begin{remark}
This notion (for operators satisfying (1) and (2)) is implicitly studied in the work on the jump of a represented space, e.g.~by de Brecht \cite{dB14}.
From the viewpoint of the jump of a represented space, one may think of $\mathbf{j}$ as an endofunctor on the category ${\bf Rep}$ of represented spaces and computable functions.
Indeed, any $j$-operator $\mathbf{j}$ yields a {\em monad} on the category ${\bf Rep}$:
The condition (3) gives us a monad multiplication, and from (2) we get a unit.
Thus, the $\mathbf{j}$-computable functions on represented spaces are exactly the {\em Kleisli morphisms} for this monad.
\end{remark}

\begin{remark}
In the context of Lifschitz realizability, the property (1) in our definition of a $j$-operator corresponds to \cite[Lemma 4]{Lif}, \cite[Lemma 5.7]{vO}, and \cite[Lemma 4.4]{ChRa12}, the property (2) corresponds to \cite[Lemma 2]{Lif}, \cite[Lemma 5.4]{vO}, and \cite[Lemma 4.2]{ChRa12}, and (3) corresponds to \cite[Lemma 3]{Lif}, \cite[Lemma 5.6]{vO}, and \cite[Lemma 4.5]{ChRa12}.
\end{remark}

\begin{remark}
The conditions (1), (2), and (3) are also mentioned independently in \cite{Yos2}.
There, (1) is called self-universality.
\end{remark}

One can obtain a lot of examples of $j$-operators from partial multi-valued oracles.

\begin{lemma}\label{lem:idempotent}
Let $\mathbb{P}=(P,\tpbf{P},\ast)$ be a relative pca, and let $F\pcolon\tpbf{P}\tto\tpbf{P}$ be a partial multifunction.
Then, $F^{\clo}$ is a $j$-operator on $\mathbb{P}$.
\end{lemma}

\begin{proof}
By definition, it is easy to check that $F^{\clo}$ satisfies the conditions (1) and (2).
The idempotence (3) follows from transitivity of $\leq_{\gW}$ \cite{HiJo16}.
\end{proof}

In fact, by Westrick's theorem \cite{westrick2020note}, $F^{\clo}$ is the weakest $j$-operator generated by $F$.
Moreover, the $(\mathbb{P},F^{\clo})$-computable partial multifunctions coincide with the partial multifunctions $\leq_{\gW}F$.

\begin{obs}\label{obs:basic-of-j-operator}
Let $\jump$ be a $j$-operator on $\mathbb{P}$ and $f\pcolon\tpbf{P}\tto\tpbf{P}$ be a partial multifunction.
\begin{enumerate}
\item $f\leq_{\sf W}\jump$ iff $f$ is $(\mathbb{P},\jump)$-computable.
\item $f\leq_{\sf W}\jump$ iff $f^{\clo}\leq_{\sf W}\jump$ iff $f\leq_{\gW}\jump$.
\end{enumerate}
\end{obs}

\begin{proof}
(1) This follows from the self-universality of a $j$-operator; i.e., the condition (1) in Definition \ref{def:j-operator}.
(2) See Westrick \cite{westrick2020note}, and also \cite{Kih22}.
\end{proof}

For a $j$-operator $\jump$ on a relative pca $\mathbb{P}=({\tt P},\tpbf{P},\ast)$, one can introduce a new partial application $\ast_\mathbf{j}$ on $\tpbf{P}$ defined by $a\ast_\mathbf{j}b\simeq a^!\ast b$ if $\jump(a)=\{a^!\}$; otherwise $a\ast_\mathbf{j}b$ is undefined.
Hereafter, we always write $a^!$ for the unique element of $\jump(a)$ whenever $\jump(a)$ is a singleton.
Then, consider the following ``deterministic part'':
\[{\tt P}_\jump=\{a^!:a\in {\tt P}\mbox{ and }\jump(a)\mbox{ is a singleton}\}.\]

\begin{lemma}\label{lem:pca-renewal}
$\mathbb{P}_\jump=({\tt P}_\jump,\tpbf{P},\ast_\jump)$ is a relative pca.
Moreover, (4) there is $\ep\in {\tt P}_\jump$ such that for any $a\in\tpbf{P}$, if $a^!$ is defined, then $\ep\ast_\jump a=a^!$.
\end{lemma}

\begin{proof}
To see that ${\tt P}_\jump$ is closed under $\ast_\jump$, for $a,b\in {\tt P}$, we have $a\ast_\jump b=a^!b$.
Then, $a^!b=(\lambda x.xb)a^!=({\tt u}(\lambda x.xb)a)^!$, where ${\tt u}$ is from (1).
As ${\tt u}(\lambda x.xb)a\in {\tt P}$, we have $a^!b\in {\tt P}_\jump$.

Given a combinator ${\sf k}$ in ${\tt P}$, define ${\sf k}_\jump=\eta(\lambda x.\eta({\sf k}x))$, where $\eta$ is from (2).
Then, ${\sf k}_\jump\ast_\jump a\ast_\jump b=({\sf k}_\jump^!a)^!b=(\iota({\sf k}a))^!b={\sf k}ab=a$.
Given a combinator ${\sf s}$ in ${\tt P}$, define ${\sf s}_\jump=\eta(\lambda x.\eta(\lambda y.\eta({\sf s}xy)))$.
Then, ${\sf s}_\jump\ast_\jump a\ast_\jump b\ast_\jump c=(({\sf s}_\jump^!a)^!b)^!c=((\eta(\lambda y.\eta({\sf s}ay))^!b)^!c=(\eta({\sf s}ab))^!c={\sf s}abc=ac(bc)$.

Thus, $\mathbb{P}_\jump$ is a relative pca, and so we get a new $\lambda_\jump$ by combinatory completeness.
For the second assertion, define $\ep=\lambda_\jump x.{\tt u}({\sf k}x)\ast_\jump 0$.
Then, $\ep\ast_\jump a={\tt u}({\sf k}a)\ast_\jump 0=({\tt u}({\sf k}a))^!0={\sf k}a^!0=a^!$.
\end{proof}

\begin{remark}
In the context of Lifschitz realizability, the condition (4) in Lemma \ref{lem:pca-renewal} corresponds to \cite[Lemma 1]{Lif}, \cite[Lemma 5.3]{vO}, and \cite[Lemma 4.3]{ChRa12}.
\end{remark}

We next show that the update of a given pca preserves the class of realizable functions.

\begin{lemma}\label{lem:trackable-jump-uniquew}
A partial multifunction $f\pcolon\tpbf{P}\tto\tpbf{P}$ is $(\mathbb{P}_\jump,\jump)$-trackable iff it is $(\mathbb{P},\jump)$-trackable.
Similarly, $f$ is $(\mathbb{P}_\jump,\jump)$-computable iff it is $(\mathbb{P},\jump)$-computable.
\end{lemma}

\begin{proof}
We only consider the latter assertion.
By definition, every $(\mathbb{P},\jump)$-computable partial multifunction is $(\mathbb{P}_\jump,\jump)$-computable.
Conversely, if $F$ is $(\mathbb{P}_\jump,\jump)$-computable, then there is $a\in {\tt P}_\jump$ such that $\jump(ax)\subseteq F(x)$.
By the definition of ${\tt P}_\jump$, there is $b\in {\tt P}$ such that $a=b^!$.
Then $b^!x=(\lambda y.yx)b^!=({\tt u}(\lambda y.yx)b)^!$, where ${\tt u}$ is from (1).
Therefore, $\jump(ax)=\jump(b^!x)=\jump\jump({\tt u}(\lambda y.yx)b)=\jump(\mu({\tt u}(\lambda y.yx)b))$ where $\mu$ is from (3).
Since $\mu({\tt u}(\lambda y.yx)b)\in {\tt P}$, this shows that $F$ is $(\mathbb{P},\jump)$-computable.
\end{proof}

Consequently, Lemma \ref{lem:idempotent} still holds even if we replace $\mathbb{P}$ with the new pca $\mathbb{P}_\jump$.
Hereafter, we always assume the underlying pca $\mathbb{P}$ is $\mathbb{P}_\jump$; so in particular, the properties (1), (2), (3) and (4) hold.

Lifschitz \cite{Lif} showed that the above four properties (1), (2), (3) and (4) ensure that all axioms of the Heyting arithmetic ${\bf HA}$ are $\jump_{\rm L}$-realizable over ${\sf K}_1$.
Moreover, van Oosten \cite{vO} showed that all axioms of Troelstra's elementary analysis ${\bf EL}$ are $\jump_{\rm vO}$-realizable over ${\sf K}_2^{\sf eff}$.
Then Chen-Rathjen \cite{ChRa12} showed that all axioms of ${\bf IZF}$ are $\jump_{\rm L}$-realizable over ${\sf K}_1$.
To summarize the above, one can now show that the properties (1), (2), (3), and (4) ensure the following general result.

\begin{theorem}\label{thm:IZF-realizability}
Let $(\mathbb{P},\jump)$ be a site.
Then, all axioms of ${\bf IZF}$ are $\jump$-realizable over $\mathbb{P}_\jump$.
\qed
\end{theorem}

\subsubsection{Topos-theoretic perspective}

A topos-theoretic explanation may convince some readers of the fact that {\bf IZF} is $\jump$-realizable (Theorem \ref{thm:IZF-realizability}) without referring to \cite{Lif,vO,ChRa12}.
If the reader is already convinced by the concrete discussion in the previous section (with the references \cite{Lif,vO,ChRa12}), they can skip this section.

As noted by the author \cite{Kih22}, a $j$-operator on $\mathbb{P}$ in the sense of Definition \ref{def:j-operator} yields a Lawvere-Tierney topology on the relative realizability topos ${\sf RT}(\mathbb{P})$.
A {\em Lawvere-Tierney topology} on a topos is an (internally) $\land$-preserving closure operator on the truth-value object $\Omega$.
It is known that a Lawvere-Tierney topology on a topos $\mathcal{E}$ corresponds to a subtopos of $\mathcal{E}$; see e.g.~\cite[Chapter V]{SGL}.
In particular, a site $(\mathbb{P},{\bf j})$ always presents a subtopos of ${\sf RT}(\mathbb{P})$.

Let us explain more concretely.
Let $\Omega$ be the power set of $\tpbf{P}$.
A partial multifunction $F\pcolon\tpbf{P}\tto\tpbf{P}$ yields $j_F\colon\Omega\to\Omega$ as follows:
\[
j_F(p)=\{\langle e,n\rangle:n\in{\rm dom}(\Phi_e^F)\mbox{ and }\Phi_e^F(n)\subseteq p\}
\]

This gives a Lawvere-Tierney topology on ${\sf RT}(\mathbb{P})$; see \cite{Kih22}.
Note that $j_F$ can also be considered as a closure operation on the realizability tripos over $\mathbb{P}$.
That is, think of $\Omega^X$ as the Heyting pre-algebra, where $\varphi\leq_X\psi$ iff there is $a\in{\tt P}$ such that $n\in\varphi(x)$ implies $an\in\psi(x)$ for any $x\in X$.
This is called a realizability tripos \cite{vOBook}.
Then $F$ yields a new preorder $\leq^F_X$ on $\Omega^X$ as follows:
For predicates $\varphi,\psi$ on $X$, define $\varphi\leq^F_X\psi$ as $\varphi(x)\leq_Xj_F(\psi(x))$.
Equivalently,
\[
\exists e\in{\tt P}\forall x\in X\forall n\in\varphi(x)\ [{\Phi^F_e(n)\downarrow}\mbox{ and }\forall b\in\Phi^F_e(n).\ b\in\psi(x)]
\]

Then $(\Omega^X,\leq^F_X)$ yields a new tripos.
The interpretations of logical connectives are given as follows:
\begin{gather*}
\top_F=\top,
\qquad
\bot_F=\bot,\\
\varphi\land_F\psi = \varphi\land\psi,
\qquad
\varphi\lor_F\psi = \varphi\lor\psi,
\qquad
\varphi\to_F\psi = \varphi\to j_F(\psi),
\\
\exists^F = \exists,
\qquad
\forall^F = \forall\circ j_F
\end{gather*}

This is a variant of the G\"odel-Gentzen translation, called the (Kuroda-style) $j_F$-translation (see e.g.~\cite{Lee,vdB18,RaSw20}).
This gives exactly the same definition as Definition \ref{def:realizability-relative-to-multifunction2}, so one can conclude that $F$-realizability is the semantics of the subtopos of ${\sf RT}(\mathbb{P})$ consisting of the $j_F$-sheaves (since an inclusion of triposes coincides with an inclusion of toposes; see also van Oosten \cite[Theorem 2.5.11]{vOBook}).

\begin{remark}
Our term ``site'' introduced in Definition \ref{def:site} originates from topos theory.
In topos theory, a site is a presentation of a Grothendieck topos, specifically, the pair $(\mathcal{C},J)$ of a small category $\mathcal{C}$ and a Grothendieck topology $J$.
In other words, a site $(\mathcal{C},J)$ is a presentation of the topos of $J$-sheaves over $\mathcal{C}$, which is a subtopos of the presheaf topos ${\sf Psh}(\mathcal{C})$ over $\mathcal{C}$.
In realizability theory, we consider a relative pca $\mathbb{P}$ to play the role of a base category $\mathcal{C}$, and an oracle $\jump$ to play the role of a Grothendieck topology (see also \cite{Kih22}), which is why we borrowed the term ``site.''
\end{remark}

\subsection{Weihrauch analysis of realizability interpretation}
\subsubsection{The internal Baire space}\label{sec:internal-Baire}

%

Our main focus is on interpreting mathematical statements in the language of second-order arithmetic.
Here, $\N$ and $\N^\N$ are basic sorts in the language, so we want to find out how they are interpreted.
One may use Definition \ref{def:Chen-Rathjen-realizability} directly, but the interpretation in a realizability topos (or its subtopos) is clearer, so we explain that first.

Let $(\mathbb{P},\jump)$ be a site.
We explicitly describe the internal logic of the topos of $\jump$-sheaves.
Of course, the use of a topos for the interpretation of second-order arithmetic is like using a sledgehammer to crack a nut.
It is sufficient to consider its full subcategory consisting of assemblies.

\begin{definition}
A {\em $\mathbb{P}$-assembly} (or a multi-represented space) is a pair of a set $X$ and a function $\mathbf{E}_X\colon X\to\Omega$, where $\Omega$ is the power set of $\tpbf{P}$, and $\mathbf{E}_X(x)\not=\emptyset$.
Each $p\in\mathbf{E}_X(x)$ is called a code of $x$.
For $\mathbb{P}$-assemblies $X$ and $Y$, a function $f\colon X\to Y$ is $(\mathbb{P},\jump)$-trackable if there is $a\in\tpbf{P}$ such that, for any code $p$ of $x\in X$, any element of $\jump(ap)$ is a code of $f(x)\in Y$.
If $a\in{\tt P}$, we say that $f$ is $(\mathbb{P},\jump)$-computable.
\end{definition}

If we set $F(p)={\bf E}_Y(f(x))$ for $p\in\mathbf{E}_X(x)$, this terminology is consistent with Definition \ref{def:trackability}.
The Kleisli category over the site $(\mathbb{P},\jump)$ consists of $\mathbb{P}$-assemblies as objects and $(\mathbb{P},\jump)$-computable functions as morphisms.
We often omit $\mathbb{P}$ and simply refer to $(\mathbb{P},\jump)$-trackable (-computable) as a $\jump$-trackable (-computable) function.

\begin{definition}
The exponential assembly $Y^X$ is the set of $\jump$-trackable functions from $X$ to $Y$, where a code of $f\in Y^X$ is a $\jump$-tracker of $f$.
\end{definition}

\begin{example}\label{exa:internal-Baire-space}
The natural numbers object ${\sf N}$ is given as follows:
The underlying set is $\om$, and a code of $n\in{\sf N}$ is the Church numeral $\num{n}$; i.e., ${\bf E}_{\sf N}(n)=\{\num{n}\}$.

Then the {\em internal Baire space} ${\sf N}^{\sf N}$ is given as the exponential assembly.
To be explicit, the underlying set is the set of all $\jump$-trackable functions on $\om$, and ${\tt f}\in\tpbf{P}$ is a code of $f\in{\sf N}^{\sf N}$ iff $\jump({\tt f}\num{n})=\{\num{f(n)}\}$ for any $n\in\om$.
\end{example}

\begin{example}\label{exa:external-Baire-space}
If $\mathbb{P}={\sf K}_2^{\sf eff}$ then $\tpbf{P}=\om^\om$.
One may consider $\om^\om$ as an assembly given by the identical coding ${\bf E}_{\om^\om}(f)=f$ for any $f\in\om^\om$.
We call this the {\em external Baire space}.

In the case where $\tpbf{j}={\rm id}$, we have $\om^\om\simeq{\sf N}^{\sf N}$, but this is not always the case in general.
\end{example}

\begin{remark}
Of course, one may obtain the internal Baire space directly using the realizability interpretation (Definition \ref{def:Chen-Rathjen-realizability}) without going through the category of assemblies:
First $e\reap x\in\N^\N$ means that $e\reap\forall n\in\N\exists! m\in\N\;\langle n,m\rangle\in x$.
Then, it is easy to check that there is a unique $\tilde{x}\in\N^\N$ such that $\langle n,\tilde{x}(n)\rangle$ is realizable.
By Chen-Rathjen \cite[Lemma 4.7]{ChRa12}, given such an $e$, one can effectively produce a $p_e$ such that $p_e\ast_\jump n=\{\tilde{x}(n)\}$, i.e., $(p_e^!n)^!=\tilde{x}(n)$.
Conversely, from such $p_e$, one can effectively recover the realizer $e$.
That is, we consider $p_x$ as a name of $\tilde{x}\in\N^\N$ if $(p_x^!n)^!=\tilde{x}(n)$.
In essence, this gives a $\mathbb{P}_\jump$-tracker for $\tilde{x}$.
Therefore, by Lemma \ref{lem:trackable-jump-uniquew}, one can see that this is equivalent to considering the above exponential assembly ${\sf N}^{\sf N}$.
\end{remark}


\subsubsection{Internal versus external Baire spaces}

We connect the realizability interpretation and Weihrauch reducibility.
When interpreting a formula in second-order arithmetic, we must be careful that the interpretation of type $\N^\N$ is the internal Baire space ${\sf N}^{\sf N}$, whereas our Weihrauch reducibility is based on the external Baire space $\om^\om$.
If we can identify ${\sf N}^{\sf N}$ with $\om^\om$ as in ${\sf K}_2^{\sf eff}$, then there is no need to worry about this.

\begin{definition}
Let $\jump$ be a $j$-operator on ${\sf K}_2^{\sf eff}$.
We say that $\jump$ satisfies the {\em relative unique Church thesis} ({\sf rCT}!) if ${\rm id}\colon\om^\om\to{\sf N}^{\sf N}$ is an isomorphism in the Kleisli category over $({\sf K}_2^{\sf eff},\jump)$.
\end{definition}

\begin{remark}
As in Lemma \ref{lem:pca-renewal}, the computability of ${\rm id}\colon\om^\om\to{\sf N}^{\sf N}$ always holds, so the nontrivial part is the computability of ${\rm id}\colon{\sf N}^{\sf N}\to\om^\om$.
That is, $\jump$ satisfies ${\sf rCT}!$ iff there is a $\jump$-computable function which, given a $\jump$-tracker of $f\colon{\sf N}\to{\sf N}$, returns $f\in\om^\om$.
\end{remark}

\begin{example}
$\jump={\rm id}$ satisfies ${\sf rCT}!$.
This is because ${\rm id}\colon\om^\om\to{\sf N}^{\sf N}$ is an isomorphism in the category of ${\sf K}_2^{\sf eff}$-assemblies.
\end{example}

\begin{definition}
Let $\jump$ be a $j$-operator on ${\sf K}_2^{\sf eff}$.
Then define $\jump^!\pcolon\om^\om\to\om^\om$ as follows:
\[
\jump^!(\alpha)(n)=m\iff \jump(\alpha_n)=\{\num{m}\}
\]
where for $\alpha\in\om^\om$, $\alpha_n\in\om^\om$ is given by $\alpha_n(m)=\alpha(\langle n,m\rangle)$, and $\alpha\in{\rm dom}(\jump^!)$ iff $\jump^!(\alpha)(n)$ is defined for all $n\in\om$.
\end{definition}

\begin{remark}
Using the terminology of the theory of Weihrauch degrees, $\jump^!$ is the parallelization of the deterministic first-order part of $\jump$.
\end{remark}

\begin{obs}\label{obs:monotone-pdf-partofjump}
If $\jump_0\leq_{\sf W}\jump_1$ then $\jump_0^!\leq_{\sf W}\jump_1^!$.
\end{obs}

\begin{lemma}\label{lem:charac-rCTunique}
$\jump$ satisfies ${\sf rCT}!$ if and only if $\jump^!\leq_{\sf W}\jump$.
\end{lemma}

\begin{proof}
($\Rightarrow$)
Let $\alpha=(\alpha_n)_{n\in\om}\in{\rm dom}(\jump^!)$ be given.
One can easily compute $e\in\om^\om$ such that $\alpha_n=e \alpha \num{n}$ holds in ${\sf K}_2^{\sf eff}$ for any $n\in\om$.
As $\alpha\in{\rm dom}(\jump^!)$, we have $\jump(e\alpha\num{n})=\jump(\alpha_n)=\{\num{\jump^!(\alpha)(n)}\}$.
By the definition of ${\sf N}^{\sf N}$ in Example \ref{exa:internal-Baire-space}, this means that $e\alpha$ is a code of $\jump^!(\alpha)\in{\sf N}^{\sf N}$.
Now, by our assumption, ${\rm id}\colon{\sf N}^{\sf N}\to\om^\om$ is $\jump$-computable, so there is $p\in\om^\om$ such that any $z\in\jump(p(e\alpha))$ is a code of $\jump^!(\alpha)\in\om^\om$.
By the definition of the external Baire space $\om^\om$ in Example \ref{exa:external-Baire-space}, this exactly means $\jump(p(e\alpha))=\{\jump^!(\alpha)\}$.
Therefore, we obtain $\jump^!\leq_{\sf W}\jump$.

($\Leftarrow$)
As in Lemma \ref{lem:pca-renewal}, one can check that ${\rm id}\colon\om^\om\to{\sf N}^{\sf N}$ is always computable.
To show that ${\rm id}\colon{\sf N}^{\sf N}\to\om^\om$ is $\jump$-computable, suppose that we are given a code $z$ for $x\in{\sf N}^{\sf N}$.
By the definition of the internal Baire space, for any $n\in\om$, we have $\jump(z\num{n})=\{\num{x(n)}\}$.
Take $\alpha$ such that $\alpha_n=z\num{n}$.
Then we have $\jump(\alpha_n)=\{\num{x(n)}\}$, so we obtain that $\jump^!(\alpha)=x$.
By the assumption, we have $\jump^!\leq_{\sf W}\jump$, so by Observation \ref{obs:basic-of-j-operator}, there exists some $w\in\om^\om$ such that $\jump(w\alpha)=\{\jump^!(\alpha)\}=\{x\}$ holds.
Now, if we consider $v\in\om^\om$ such that $vz=\alpha$, then $\lambda z.w(vz)$ tracks ${\rm id}\colon{\sf N}^{\sf N}\to\om^\om$.
\end{proof}

Recall from Lemma \ref{lem:idempotent} that $F^{\clo}$ is a $j$-operator for any $F\pcolon\om^\om\tto\om^\om$.
By an abuse of notation, we also use $F^!$ as an abbreviation for $(F^{\clo})^!$.

\begin{lemma}\label{lem:below-WKL-rCT}
If $\jump\leq_{\sf W}{\sf WKL}$, then $\jump$ satisfies ${\sf rCT}!$.
\end{lemma}

\begin{proof}
By Observation \ref{obs:monotone-pdf-partofjump} and Lemma \ref{lem:charac-rCTunique}, it suffices to show that ${\sf WKL}^!\leq_{\sf W}{\sf WKL}$, but this is obvious since ${\sf WKL}^!\equiv_{\sf W}{\rm id}$ (that is, ${\sf WKL}^!$ is computable), which follows from an easy observation that the deterministic part of ${\sf WKL}$ (so its parallelization) is trivial (see e.g.~\cite[Corollary 8.8]{BG11}).
Note that this observation corresponds to the fact that ${\sf WKL}^{\clo}$-realizability satisfies the unique Church thesis ${\sf CT}_0!$.
\end{proof}

\begin{example}[Non-example]
The discrete limit operator $\lim_\om\pcolon\om^\om\to\om$ does not satisfy ${\sf rCT}!$:
This is a partial single-valued function defined by $\lim_\om(\alpha)=\lim_{n\to\infty}\alpha(n)$ if the limit exists in the discrete space $\om$.
Its deterministic first-order part is $\lim_\om$ itself, but its parallelization is the limit operator $\lim_{\om^\om}\pcolon(\om^\om)^\om\to\om^\om$, where $\lim_{\om^\om}(\alpha)=\lim_{n\to\infty}\alpha(n)$ if the limit exists in the Baire space $\om^\om$.
Then we have $\lim_\om<_{\sf W}\lim_\om^!\equiv_{\sf W}\lim_{\om^\om}$.
\end{example}

\begin{remark}
For any $\jump$, $\jump^!\pcolon\om^\om\to\om^\om$ is ``parallelizable'' partial single-valued function.
The continuous Weihrauch degrees of parallelizable total single-valued functions\ are completely classified by the author \cite{Ki19} (under meta-theory ${\sf ZF}+{\sf DC}+{\sf AD}$, where ${\sf DC}$ is the axiom of dependent choice, and ${\sf AD}$ is the axiom of determinacy): The structure of parallelized continuous Weihrauch degrees of single-valued functions is well-ordered (with the order type $\Theta$), and the smallest nontrivial one already entails the Turing jump operator (a universal Baire class $1$ function).
In particular, under meta-theory ${\sf ZFC}$, by Martin's Borel determinacy theorem, there are $\om_1$ many Borel single-valued functions (which are exactly universal Baire class $\alpha$ functions).
%
%
%
%
%
%
\end{remark}

\subsubsection{$\forall\exists$-theorems and Weihrauch degrees}

We say that a partial multifunction $F\pcolon\tpbf{P}\tto\tpbf{P}$ is {\em $\neg\neg$-dense} if $\jump(x)\not=\emptyset$ for any $x\in{\rm dom}(F)$.
Note that if $F$ is $\neg\neg$-dense then so is $\jump=F^{\clo}$.
In this section, we assume that $\jump$ is a $\neg\neg$-dense $j$-operator on ${\sf K}_2^{\sf eff}$.

\begin{remark}
In this case, an almost negative statement $\psi$ is $\jump$-realizable iff it is true, and one may compute a term realizing $\psi$ (see e.g.~\cite[Proposition 4.4.5]{TrvD88}).
Here, an almost negative formula $\varphi$ is a formula which does not contain $\lor$, and $\exists$ appears only immediately in front of an atomic subformula.
Since the realizability of a $\Delta^0_0$-formula (a formula in second-order arithmetic containing no unbounded quantifiers) also corresponds to the truth, one can treat a $\Delta^0_0$-formula as an atomic formula without any harm.
In particular, $\Sigma^0_1$- and $\Pi^0_1$-formulas can be considered almost negative formulas.
\end{remark}

Now, consider an $\forall\exists$-statement
\[\varphi\equiv\forall x\in\N^\N\;[\eta(x)\to \exists y\in\N^\N\;\theta(x,y)]\]
for some almost negative formulas $\eta$ and $\theta$.
We call such a formula $\varphi$ a {\em tidy $\forall\exists$-principle}.

\begin{example}
${\sf WKL}$ is a tidy $\forall\exists$-principle: The premise $\eta(T)$ is the conjunction of the statement that $T$ is a binary tree, and the formula $\forall\ell\exists\sigma\in 2^\ell.\ \sigma\in T$.
Hence, $\eta$ is $\Pi^0_1$, so almost negative.
The conclusion is $\theta(p)\equiv\forall n.\;p\upto n\in T$.
This is clearly almost negative.

Similarly, ${\sf IVT}$, ${\sf BE}$, ${\sf RDIV}$ and ${\sf LLPO}$ are tidy $\forall\exists$-principles:
The premises of ${\sf BE}$, ${\sf RDIV}$ and ${\sf LLPO}$ are $\Pi^0_1$, and the premise of ${\sf IVT}$ is $\Sigma^0_1$.
The conclusions are $\Pi^0_1$.
\end{example}

\begin{example}
$\Sigma^0_1$-${\sf DNE}_\R$, $\Sigma^0_1$-${\sf DML}_\R$ and $\Pi^0_1$-${\sf LEM}_\R$ introduced in Section \ref{sec:subsec-lem-coll} are tidy $\forall\exists$-principles.
For example, the premise of $\Pi^0_1$-${\sf LEM}_\R$ is $\Pi^0_1$.
For the conclusion $\theta(\langle x,y\rangle,i)\equiv (i=0\to x=y\mbox{ and }i\not=0\to x\not=y)$, clearly, $x=y$ is $\Pi^0_1$ and $x\not=y$ is $\Sigma^0_1$, so $\theta$ is almost negative.
\end{example}

By the above Remark, one can easily see that a tidy $\forall\exists$-principle $\varphi$ is $\jump$-realizable iff there is an element $e\in(\om^\om)^{\sf eff}$ such that any $a\in\jump(e)$ witnessing that if $p$ codes $x\in{\sf N}^{\sf N}$ such that $\eta(x)$ is true then any $q\in\jump(ap)$ codes $y\in{\sf N}^{\sf N}$ such that $\theta(x,y)$ is true.

The partial multifunction corresponding to the formula $\varphi$ is given as follows.
\begin{align*}
{\rm dom}(F_\varphi)&=\{x\in\om^\om:\eta(x)\mbox{ is true}\},\\
F_\varphi(x)&=\{y\in\om^\om:\theta(x,y)\mbox{ is true}\}.
\end{align*}

Of course, we must pay attention to the difference between $\om^\om$ and ${\sf N}^{\sf N}$.
However, as long as we are considering $j$-operators that satisfy ${\sf rCT}!$, we can ignore this difference.

\begin{lemma}\label{lem:main-lemma}
Let $\jump$ be a $\neg\neg$-dense $j$-operator on ${\sf K}_2^{\sf eff}$, and assume that $\jump$ satisfies ${\sf rCT}!$.
Then a tidy $\forall\exists$-principle $\varphi$ is $\jump$-realizable if and only if $F_\varphi\leq_{\sf W}\jump$.
\end{lemma}

\begin{proof}
Since $\jump$ satisfies ${\sf rCT}!$, ${\sf id}\colon{\sf N}^{\sf N}\to\om^\om$ is $\jump$-computable.
Thus, by the idempotence and $\neg\neg$-density of $\jump$, $\varphi$ is $\jump$-realizable iff there is $e\in(\om^\om)^{\sf eff}$ such that for any $x\in\om^\om$ if $\eta(x)$ is true then $\theta(x,y)$ is true for any $y\in\jump(ex)$.
That is, $x\in{\rm dom}(F_\varphi)$ implies $y\in F_\varphi(x)$ for any $y\in\jump(ex)$.
This means $F_\varphi\leq_{\sf W}\jump$ by Observation \ref{obs:basic-of-j-operator}.
\end{proof}


\section{Key principles}

We introduce some non-constructive principles in the context of Weihrauch degrees, and show some implications and equivalences among them.

\subsection{Weak K\"onig's lemma}\label{sec:weak-variants-wkl}

In this section, we consider weak variants of weak K\"onig's lemma.
A {\em binary tree} is a subset of $2^{<\om}$ closed under taking initial segments.
A binary tree $T$ is {\em infinite} if for any $\ell$ there is a node $\sigma\in T$ of length $\ell$.
We consider the following properties for trees.
\begin{itemize}
\item A {\em $2$-tree} is a binary tree $T$ such that, for any $\ell\in\N$, $T$ has exactly two strings of length $\ell$.
\item A {\em rational $2$-tree} is a $2$-tree which does {\em not} have infinitely many branching nodes.
\item An {\em all-or-unique tree} or simply an {\em aou-tree} is a binary tree $T$ such that, for any $\ell\in\N$, $T$ has either all binary strings or only one string of length $\ell$.
\item A {\em convex tree} is a binary tree $T$ such that, for any $\ell\in\om$, if two strings $\sigma$ and $\tau$ of length $\ell$ are contained in $T$, then all length $\ell$ strings between $\sigma$ and $\tau$ are also contained in $T$.
\item A binary tree $T$ is {\em constructively non-clopen} if, for any finitely many binary strings $\sigma_0,\dots,\sigma_\ell$, there is a binary string $\tau$ such that the following does not hold: $\tau\in T$ if and only if $\tau$ extends $\sigma_i$ for some $i\leq\ell$.
\item A {\em clopen tree} is a binary tree $T$ which is not constructively non-clopen.
If $\ell=0$, then we say that $T$ is a {\em basic clopen tree}.
\end{itemize}

The notion of a clopen tree is a bit weird, but we use this notion as it is useful for understanding the relationship among non-constructive principles.
For a collection $\mathcal{T}$ of trees, we mean by {\em weak K\"onig's lemma for $\mathcal{T}$} the statement that every infinite tree in $\mathcal{T}$ has an infinite path.
Then, we denote by ${\sf WKL}_{\leq 2}$, ${\sf WKL}_{=2}$, ${\sf WKL}_{\rm aou}$, ${\sf WKL}_{\rm conv}$, and ${\sf WKL}_{\rm clop}$ weak K\"onig's lemma for $2$-trees, rational $2$-trees, aou-trees, convex trees, and convex clopen trees, respectively.

\begin{remark}
It is known that, over elementary analysis ${\sf EL}_0$, the binary expansion principle ${\sf BE}$ is equivalent to weak K\"onig's lemma for $2$-trees, ${\sf WKL}_{\leq 2}$, and the intermediate value theorem ${\sf IVT}$ is equivalent to weak K\"onig's lemma for convex trees, ${\sf WKL}_{\rm conv}$; see~\cite{BIKN}.
For the relationship among these principles, see Figure \ref{figure:principles-wkl} and Section \ref{section:implications} for the proofs.
\end{remark}

\begin{figure}[t]
\[{\small
\xymatrix{
	&		{\sf WKL} \ar[d]	 & & & {\sf WKL} \ar[d]	\\
	&	{\sf IVT}\ar[d]\ar[dl]	& & & {\sf WKL}_{\rm conv}\ar[d]\ar[dl]\\
	 {\sf IVT}_{\rm lin} \ar[d]\ar[dr] & {\sf BE} \ar[d] & & {\sf WKL}_{\rm clop} \ar[d]\ar[dr]  &  {\sf WKL}_{\leq 2} \ar[d] \\
	 {\sf RDIV} \ar[dr] & {\sf BE}_\mathbb{Q}\ar[d] & & {\sf WKL}_{\rm aou} \ar[dr] & \wklQ\ar[d]\\
 & {\sf LLPO} & & & {\sf LLPO}
}}
\]
\caption{Nonconstructive principles as variants of weak K\"onig's lemma}\label{figure:principles-wkl}
\end{figure}

\begin{obs}\label{obs:WKLs-are-tidy}
${\sf WKL}_{\leq 2}$, ${\sf WKL}_{=2}$, ${\sf WKL}_{\rm aou}$, ${\sf WKL}_{\rm conv}$ and ${\sf WKL}_{\rm clop}$ are tidy $\forall\exists$-principles.
\end{obs}

\begin{proof}
Clearly, ${\sf WKL}_{\leq 2}$, ${\sf WKL}_{\rm aou}$ and ${\sf WKL}_{\rm conv}$ are tidy $\forall\exists$-principles.

Similarly, ${\sf WKL}_{=2}$ is a tidy $\forall\exists$-principle:
The formula describing that $T$ does not have infinite many branching nodes is
\[
\neg\forall\ell\exists\sigma\ (|\sigma|\geq\ell\mbox{ and }\sigma,\sigma\fr 0,\sigma\fr 1\in T).
\]

Clearly, the existential quantifier $\exists$ appears only in the front of a $\Delta^0_0$ subformula.
Hence, the premise of ${\sf WKL}_{=2}$ is almost negative.

${\sf WKL}_{\rm clop}$ is also a tidy $\forall\exists$-principle:
For a formula describing being non-constructively non-clopen, the existential quantifier $\exists\tau$ appears only in the front of a $\Delta^0_0$ subformula.
Hence, the premise of ${\sf WKL}_{\rm clop}$ is almost negative.
\end{proof}

In the context of Weihrauch degrees, these notions are formalized as multifunctions.
First, {\em weak K\"onig's lemma for $2$-trees}, ${\sf WKL}_{\leq 2}\pcolon 2^{2^{<\om}}\tto 2^\om$, is defined as follows:
\begin{align*}
{\dom}({\sf WKL}_{\leq 2})&=\{T\subseteq 2^{<\om}:\mbox{$T$ is a $2$-tree}\},\\
{\sf WKL}_{\leq 2}(T)=[T]&:=\{p\in 2^\om:\mbox{$p$ is an infinite path through $T$}\}.
\end{align*}

Classically, a $2$-tree has either {\em one or two} infinite paths, and a rational $2$-tree has {\em exactly two} infinite paths.
Then, define the multifunction $\wklQ$ as ${\sf WKL}_{\leq 2}$ restricted to the rational $2$-trees.

\begin{remark}
In the study of Weihrauch degrees, this kind of notion is treated as a choice principle.
Consider the {\em choice principle for finite closed sets in $X$ with at most $n$ elements ${\sf C}_{X,\#\leq n}$ (with exactly $n$ elements ${\sf C}_{X,\#=n}$)}, which states that, given (a code of) a closed subset $P$ of a space $X$ if $P$ is nonempty, but has at most $n$ elements (exactly $n$ elements), then $P$ has an element; see \cite{LRPa15}.
The principles ${\sf WKL}_{\leq 2}$ and ${\sf C}_{X,\#\leq 2}$ are not constructively equivalent, where $X=2^\om$.
However, if we consider ${\sf WKL}_{\leq 2}$ as a multifunction, then they are Weihrauch equivalent.
Similarly, $\wklQ$ and ${\sf C}_{X,\#=2}$ are Weihrauch equivalent.
\end{remark}

Next, the {\em weak K\"onig's lemma for convex trees} is a multifunction ${\sf WKL}_{\rm conv}\pcolon 2^{2^{<\om}}\tto 2^\om$ defined as follows:
\begin{align*}
{\dom}({\sf WKL}_{\rm conv})&=\{T\subseteq 2^{<\om}:\mbox{$T$ is a convex tree}\},\\
{\sf {\sf WKL}_{\rm conv}}(T)=[T]&:=\{p\in 2^\om:\mbox{$p$ is an infinite path through $T$}\}.
\end{align*}

In the context of Weihrauch degrees, a clopen tree is exactly a binary tree $T$ such that $[T]$ is clopen, and a basic clopen tree is a tree of the form $\{\tau\in 2^{<\om}:\tau\succ\sigma\}$ for some binary string $\sigma$.
It is clear that every basic clopen tree is a convex clopen tree.
Then, define the multifunction ${\sf WKL}_{\rm clop}$ as ${\sf WKL}_{\rm conv}$ restricted to the convex clopen trees.

\begin{remark}
As mentioned above, the intermediate value theorem ${\sf IVT}$ is known to be constructively equivalent to  weak K\"onig's lemma for convex trees, which is also Weihrauch equivalent to the convex choice ${\sf XC}_X$ for $X=2^\om$, where ${\sf XC}_X$ states that, given (a code of) a closed subset $P$ of $X$, if $P$ is convex, then $P$ has an element; see \cite{LRPa15,KiPa19}.
\end{remark}

Finally, as a multifunction, {\em weak K\"onig's lemma for aou-trees}, ${\sf WKL}_{\rm aou}\pcolon 2^{2^{<\om}}\tto 2^\om$, is defined as follows:
\begin{align*}
{\dom}({\sf WKL}_{\rm aou})&=\{T\subseteq 2^{<\om}:\mbox{$T$ is an aou-tree}\},\\
{\sf WKL}_{\rm aou}(T)=[T]&:=\{p\in 2^\om:\mbox{$p$ is an infinite path through $T$}\}.
\end{align*}

\begin{remark}
As a multifunction, the robust division principle ${\sf RDIV}$ is Weihrauch equivalent to the {\em all-or-unique choice} ${\sf AoUC}$ (or equivalently, the {\em totalization of unique choice}), which takes, as an input,  (a code of) a closed subset $P$ of Cantor space $2^\om$ such that either $P=X$ or $P$ is a singleton, and then any element of $P$ is a possible output; see \cite{{pauly-kihara2-mfcs}}.
This is, of course, Weihrauch equivalent to weak K\"onig's lemma for aou-trees.
\end{remark}

%
%
%
%
%
%

\subsubsection{Implications}\label{section:implications}

As seen above, the principles introduced in Section \ref{sec:between-llpo-wkl} are written as some variants of weak K\"onig's lemma; see Figure \ref{figure:principles-wkl}.
We show some nontrivial implications in Figure \ref{figure:principles-wkl}.

%

\begin{prop}\label{prop:clop-to-aou}
${\sf WKL}_{\rm clop}\to{\sf WKL}_{\rm aou}$.
\end{prop}

\begin{proof}
Let $T$ be an aou-tree.
If $T$ is found to be unique $\sigma$ at height $s$, then let $H(T)=\{\tau:\tau\succ\sigma\}$, and for any infinite path $p$, define $K(p)=(p(0),p(1),\dots,p(s-2),q(s-1),q(s),q(s+1),\dots)$, where $q$ is the unique infinite path through $T$.
If it is not witnessed, i.e., $T=2^{<\om}$, then $H(T)$ becomes $2^{<\om}$, and $K(p)=p$ for any infinite path $p$.

More formally, we define a basic clopen tree $T^\ast$ and a function $k\colon\om\times 2\to 2$ as follows:
First put the empty string into $T^\ast$.
Given $s>0$, if $T$ contains all binary strings of length $s$ then $T^\ast$ contains all strings of length $s$, and define $k(s-1,i)=i$ for each $i<2$.
If $T$ has only one binary string $\sigma$ of length $s$ but it is not true for all $t<s$, then $\sigma$ is also the unique binary string of length $s$ in $T^\ast$, and define $k(s-1,i)=\sigma(s-1)$ for each $i<2$.
Otherwise, for each $\tau\in T^\ast$ of length $s-1$, put $\tau 0$ and $\tau 1$ into $T$, and define $k(s-1,i)=\sigma(s-1)$ for each $i<2$.
Then, $T^\ast$ is clearly a basic clopen tree.
Therefore, by ${\sf WKL}_{\rm clop}$, the tree $T^\ast$ has an infinite path $p$.
Then, define $p^\ast(n)=k(n,p(n))$.
Then, $p^\ast$ is of the form $\langle p(0),p(1),\dots,p(t-1),\sigma_{t}(t),\sigma_{t+1}(t+1),\dots\rangle$, where $\sigma_t$ is a string of length $t+1$ in $T$.
It is easy to check that $p^\ast$ is an infinite path through $T$.
\end{proof}

\begin{prop}\label{prop:clop-to-two}
${\sf WKL}_{\rm clop}\to{\sf WKL}_{=2}$.
\end{prop}

\begin{proof}
Let $T$ be a rational $2$-tree.
We define a basic clopen tree $T^\ast$ as follows:
For each $s>0$, $T$ has exactly two nodes $\ell_s,r_s$ of length $s$.
For $s=0$, let $\ell_s=r_s$ be the empty string.
If both $\ell_s$ and $r_s$ extend $\ell_{s-1}$ and $r_{s-1}$, respectively, then for each $\tau \in T$ of length $s-1$ put $\tau 0$ and $\tau 1$ into $T$.
Otherwise, $\ell_s$ and $r_s$ is of the form $\ell_s=b_s0$ and $r_s=b_s1$ for some $b_s$.
Then, put both $\ell_s=b_s0$ and $r_s=b_s1$ into $T^\ast$.
Clearly, $T^\ast$ is the clopen tree of the form $T^\ast=\{\tau\in 2^\om:\tau\succ b\}$, where $b$ is the last branching node of $T$.

By ${\sf WKL}_{\sf clop}$, the tree $T^\ast$ has an infinite path $p$.
For each $s$, check if $p\upto s+1\in T$ or not.
If it is true, define $p^\ast(s)=p(s)$.
Otherwise, if $s$ is the least such number, then $p\upto s\in T$, and there is no branching node above $p\upto s$.
Thus, for any $t\geq s$, there is the unique string $\sigma_t\in T$ of length $t$ extending $p\upto s$.
Hence, for such $\sigma_t$, define $p^\ast(t)=\sigma_t(t)$.
Then, it is easy to see that $p^\ast$ is an infinite path through $T$.
\end{proof}

\begin{prop}\label{prop:dne-to-clop}
$\Sigma^0_2\dner\to{\sf WKL}_{\rm clop}$.
\end{prop}

\begin{proof}
Let $T$ be a basic clopen tree.
Then, we have the double negation of the existence of $\sigma$ such that $\tau\in T$ for all $\tau\succeq\sigma$.
By $\Sigma^0_2\dner$, we get such a $\sigma$.
Then, $\sigma\fr\langle 0,0,\dots\rangle$ is an infinite path through $T$.
\end{proof}

\subsubsection{Auxiliary principles}\label{sec:Auxiliary}

As seen in Section \ref{sec:weak-variants-wkl}, several mathematical principles can be described as weak variants of ${\sf WKL}$.
For instance, ${\sf IVT}$, ${\sf BE}$ and ${\sf RDIV}$ are equivalent to ${\sf WKL}_{\rm conv}$, ${\sf WKL}_{\leq 2}$ and ${\sf WKL}_{\rm aou}$, respectively.
Focusing on weak variants of ${\sf WKL}$, as we have already seen, there are principles such as ${\sf WKL}_{=2}$ and ${\sf WKL}_{\rm clop}$.
They are well positioned as benchmarks for measuring the strength of principles.
Let us look at how these additional principles relate to ${\sf IVT}, {\sf BE}$ and ${\sf RDIV}$.
The two principles introduced below may look weird, and we do not think it has any constructive meaning, but they are introduced solely as auxiliary principles to link the benchmark principles with mathematical principles such as ${\sf IVT}$, ${\sf BE}$ and ${\sf RDIV}$; see also Figures \ref{figure:principles-over-izf} and \ref{figure:principles-wkl}.


\subsubsection*{Binary expansion for regular Cauchy rationals ${\sf BE}_\mathbb{Q}$}

We first consider the binary expansion principle ${\sf BE}_\mathbb{Q}$ for regular Cauchy rationals (more precisely, for non-irrationals), which states that, for any regular Cauchy real, if it happens to be a rational, then it has a binary expansion:
\[(\forall x\in[0,1])\;\left[\neg(\forall a,b\in\mathbb{Z},\;ax\not=b)\;\longrightarrow\;\exists f\colon\N\to\{0,1\},\;x=\sum_{n=1}^\infty 2^{-f(n)}\right].\]

\subsubsection*{Intermediate value theorem for piecewise linear maps ${\sf IVT}_{\rm lin}$}

Next, we consider a (weird) variant of the intermediate value theorem.
A rational piecewise linear map on $\mathbb{R}$ is determined by finitely many rational points $\bar{p}=(p_0,\dots,p_\ell)$ in the plane.
Then, any continuous function on $\mathbb{R}$ can be represented as the limit of a sequence $(f_s)_{s\in\om}$ of rational piecewise linear maps coded by $(\bar{p}^s)_{s\in\om}$, where any $p^{s+1}_i$ belongs to the $2^{-s-1}$-neighborhood of the graph of $f_s$.
That is, the sequence $(\bar{p}^s)_{s\in\om}$ codes an approximation procedure of a continuous function $f$.
Then, consider the case that a code $(\bar{p}^s)_{s\in\om}$ of an approximation stabilizes, in the sense that $\bar{p}^s=\bar{p}^t$ for sufficiently large $s,t$.
Then, of course, it converges to a rational piecewise linear map.

In other words, we represent a continuous function as usual, but we consider the case that such a continuous function {\em happens to be a rational piecewise linear map}.
Formally, if $f$ is coded by $(\bar{p}^s)_{s\in\om}$ which does not change infinitely often (that is, $\neg\forall s\exists t>s\;\bar{p}^t\not=\bar{p}^s$) then we call $f$ {\em a rational piecewise linear map as a discrete limit}.
The principle ${\sf IVT}_{\rm lin}$ states that if $f\colon[0,1]\to\mathbb{R}$ is a rational piecewise linear map as a discrete limit such that $f(0)<0<f(1)$ then there is $x\in[0,1]$ such that $f(x)=0$.

\begin{prop}\label{prop:beQ-eq-wkl2}
${\sf BE}_\mathbb{Q}\longleftrightarrow{\sf WKL}_{=2}$.
\end{prop}

\begin{proof}
For the forward direction, from a given rational $2$-tree $T$, one can easily construct a regular Cauchy real which happens to be a dyadic rational whose binary expansions are exactly infinite paths through $T$.
For the backward direction, from a regular Cauchy real $\alpha$, it is also easy to construct a $2$-tree $T$ whose infinite paths are binary expansions of $\alpha$ as usual.
Although the binary expansion of a rational $\alpha$ is periodic, i.e., of the form $\sigma\fr\tau\fr\tau\fr\tau\fr\dots$, this is insufficient for ensuring $T$ to be rational, so we need to construct another rational $2$-tree $T^\ast$ from $T$.

Use $\ell_s^\ast$ to denote the leftmost node of length $s$ in $T^\ast$, and $r_s^\ast$ to denote the rightmost node.
We inductively ensure that for any $s$, every node of length $s+1$ in $T$ is a successor of a node of length $s$ in $T^\ast$.
Fix $s$.
Let $b_{s+1}$ be the last branching node of length $\leq s+1$ in $T$.
Consider all decompositions $b_{s+1}=\sigma\fr\tau\fr\tau\fr\dots\fr\tau\fr(\tau\upto j)$, and then take the shortest $\sigma\fr\tau$.
Put $t=2|b_{s+1}|+1$.
If either $b_t=b_{s+1}$ or $\sigma\fr\tau$ is updated from the previous one, then define $\ell_s^\ast=\ell_s$ and $r_s^\ast=r_s$.
Otherwise, $b_{s+1}\not=b_t$, so either $\ell_{s+1}\preceq b_t$ or $r_{s+1}\preceq b_t$ holds.
If the former holds then, by induction hypothesis, $\ell_{s+1}$ is a successor of $u_s^\ast$, where $u\in\{\ell,r\}$.
Put $\overline{u}=r$ if $u=\ell$; otherwise $\overline{u}=\ell$.
Then, define $u^\ast_{s+1}=\ell_{s+1}$ and $\overline{u}^\ast_{s+1}=\overline{u}^\ast_{s}0$.
This ensures that $\ell^\ast_{s+1}\preceq b_t\preceq \ell_t,r_t$.
If $r_{s+1}\preceq b_t$ holds, replace $\ell_{s+1}$ with $r_{s+1}$ and vice versa.
This procedure adds no new branching node.
By periodicity of a binary expansion, $\sigma\fr\tau$ cannot be updated infinitely often.
Moreover, there cannot be infinitely many different $b_{s+1}$ such that $b_t=b_{s+1}$ happens.
This is because $b_t=b_{s+1}$ implies that $\ell_t=b_{s+1}01^{|b_{s+1}|}$ and $r_t=b_{s+1}10^{|b_{s+1}|}$, which witnesses that there is no seed $\sigma\fr\tau$ of periodicity of $\alpha$ in $b_{s+1}$.
Hence, if there are infinitely many such $b_{s+1}$, then $\alpha$ cannot be periodic.
Thus, $T^\ast$ is a rational $2$-tree.

Note that if $T$ is already rational then $T^\ast=T$.
Otherwise, $T$ has infinitely many branching nodes, and thus, we check arbitrary long decompositions, so the last branching node of $T^\ast$ is the first branching node in $T$ after finding the correct seed $\sigma\fr\tau$ of periodicity of $\alpha$.
In particular, any infinite path through $T^\ast$ extends $\sigma\fr\tau$.
Now, given an infinite path $p$ through $T^\ast$, for each $s$, check if $p\upto s\in T$.
If $s$ is the least number such that $p\upto s+1\not\in T$ then $T$ is not rational, and thus $p$ extends the first branching node $b\in T$ after finding the correct seed $\sigma\fr\tau$.
In particular, we must have $\sigma\fr\tau\preceq b\preceq p\upto s\in T$.
Now, consider the all decompositions $p\upto s=\mu\fr\nu\fr\nu\fr\dots\fr\nu\fr(\nu\upto j)$.
There are only finitely many such decompositions, and if a decomposition is incorrect, it is witnessed after seeing $T$ up to some finite height (which is effectively calculated from $s$).
Thus, after checking such a height, we see that all surviving decompositions are equivalent.
Therefore, a seed $\mu\fr\nu$ of such a decomposition is actually equivalent to the correct seed $\sigma\fr\tau$; hence $\mu\fr\nu\fr\nu\fr\nu\fr\dots$ is an infinite path through $T$.
\end{proof}

\begin{prop}\label{prop:lin-eq-clop}
${\sf IVT}_{\rm lin}\longleftrightarrow{\sf WKL}_{\rm clop}$.
\end{prop}

\begin{proof}
For the forward direction, one can assign an interval $I_\sigma$ to each binary string $\sigma$ in a standard manner: let $\tilde{\sigma}=\sigma[2/1]$ be the result of replacing every occurrence of $1$ in $\sigma$ with $2$, and then the endpoints of $I_\sigma$ are defined by $0.\tilde{\sigma}1$ and $0.\tilde{\sigma}2$ under the ternary expansions. 
Given a basic clopen tree $T=\{\tau:\tau\succeq\sigma\}$, it is easy to construct a piecewise linear map whose zeros are exactly $I_\sigma$.

For the backward direction, let $(f_s)$ be a descrete approximation  of a rational piecewise linear map $f$ with $f(0)<0<f(1)$.
Note that each $f_s$ is determined by finitely many points in the plane.
However, for each $s$, we only need to pay attention on at most four rational points $(x_i,f_s(x_i))_{i<4}$, where $x_0\leq x_1<x_2\leq x_3$, and $-2^{-s}>f_s(x_0)\leq f_s(x_1)\leq 0\leq f(x_2)\leq f(x_3)>2^{-s}$.
If $f(x_1)=f(x_2)$ then consider $I_s=[x_1,x_2]$.
Otherwise, by linearity, one can compute a unique rational $x_1<a_s<x_3$ such that $f(a_s)=0$.
Define $I_s$ as a sufficiently small rational neighborhood of $a_s$, and one can ensure that $(I_s)$ is a decreasing sequence.

Then, we define the tree $T^\ast$ of binary expansions of elements of $\bigcap_nI_n$.
If the rational interval $I_s$ is updated because of $f(x_1)=f(x_2)$, then one can easily compute finitely many binary strings $\sigma_0,\dots,\sigma_\ell$ of length $s$ such that any binary expansion of a real in $I_s$ extends some $\sigma_i$.
Then the nodes of length $s$ in $T^\ast$ are exactly $\sigma_0,\dots,\sigma_\ell$.
If we currently guess $f(a_s)=0$, then we proceed the argument in the proof of Proposition \ref{prop:beQ-eq-wkl2} to construct a basic clopen tree which determines a binary expansion of $a_s$.
%

Given an infinite path $p$ through $T^\ast$, for each $s$, check if $p\upto s$ has an extension in $I_s$.
If not, let $s$ be the least number such that $p\upto s+1\not\in T$.
Note that if $I_t$ is updated because of $f(x_1)=f(x_2)$ for some $t>s$ then we take some nodes $\sigma_0,\dots,\sigma_\ell$ of length $t$ and $p\upto s\preceq p\upto t=\sigma_i$ for some $i$, but $\sigma_i$ extends to some element in $I_t\subseteq I_s$.
Thus, this never happens.
If we guess $f(a_s)=0$, then as in the proof of Proposition \ref{prop:beQ-eq-wkl2} we get a binary expansion of $a_s$, so a regular Cauchy representation of $a_s$.
Otherwise, the current guess $I_s$ is already a correct solution, so just take any extension of $p\upto s$ which belongs to $I_s$.
\end{proof}

\subsection{Law of excluded middle as a Weihrauch problem}

Hereafter, we consider non-constructive principles $\Sigma^0_1\dmlr$, $\Sigma^0_2\dner$, and $\Sigma^0_2\dmlr$ as partial multifunctions.
First, recall that de Morgan's law $\Sigma^0_1\mbox{-}{\sf DML}_\mathbb{R}$ is formulated as ${\sf LLPO}$.

As a multifunction, $\Sigma^0_2\mbox{-}{\sf DNE}_\R$ is equivalent to the discrete limit operation $\limN$ (or equivalently, learnability with finite mind changes) in the Weihrauch context, where the discrete limit function $\limN$ is formalized as follows:
Given a sequence $(a_i)_{i\in\N}$ of natural numbers, if $a_i=a_j$ holds for sufficiently large $i,j$, then $\limN$ returns the value of such an $a_i$.
That is, $\limN$ is exactly the limit operation on the discrete space $\om$.

\begin{prop}\label{prop:dner-limN}
$\Sigma^0_2\mbox{-}{\sf DNE}_\R\equiv_{\sf W}{\limN}$.
\end{prop}

\begin{proof}
Assume that a sentence $\neg\neg\exists n\forall m R(n,m,\alpha)$ is given, where $R$ is decidable. 
Search for the least $n$ such that $\forall mR(n,m,\alpha)$ is true.
Begin with $n=0$, and if $\forall mR(n,m,\alpha)$ fails, then one can eventually find $m$ such that $R(n,m,\alpha)$ by Markov's principle.
If such an $m$ is found, proceed this algorithm with $n+1$.
If $\neg\neg\exists n\forall m R(n,m,\alpha)$ is true, then this algorithm eventually halts with the correct $n$ such that $\forall mR(n,m,\alpha)$ is true.
\end{proof}

First note that de Morgan's law $\Sigma^0_2\dmlr$ for $\Sigma^0_2$ formulas is essentially equivalent to the pigeonhole principle on $\N$, or Ramsey's theorem ${\sf RT}^1_2$ for singletons and two colors.
Here, ${\sf RT}^1_2$ (as a multifunction) states that, given a function $f\colon\N\to 2$ (which is called a {\em $2$-coloring}), returns $h\colon\N\to 2$ such that $f$ is constant on $H=\{n:h(n)=1\}$, where $H$ is called a {\em homogeneous set}. 

\begin{prop}\label{prop:dmlr-rt12}
$\Sigma^0_2\dmlr\equiv_{\sf W}{\sf RT}^1_2$.
\end{prop}

\begin{proof}
For ${\sf RT}^1_2\leqW\Sigma^0_2\dmlr$, given a coloring $f\colon\N\to 2$, consider (the characteristic function of) the set $C_i=\{n:f(n)=i\}$.
Either $C_0$ or $C_1$ is unbounded, so $\Sigma^0_2\dmlr$ chooses an index $j$ such that $C_j$ is unbounded (where boundedness is a $\Sigma^0_2$ property).
Then, $C_j$ must be an infinite homogeneous set.

For $\Sigma^0_2\dmlr\leqW{\sf RT}^1_2$, assume that $\Sigma^0_2$ formulas $\psi_i\equiv\exists a\forall b\varphi_i(a,b)$ are given.
If $\psi_0\land\psi_1$ is false, then by Markov's principle, for any $a$, either $\exists b\neg\varphi_0(a,b)$ or $\exists b\neg\varphi_1(a,b)$.
Hence, given $a$, search for the least $b$ such that either $\neg\varphi_0(a,b)$ or $\neg\varphi_1(a,b)$ holds.
For such $b$, if $\varphi_0(a,b)$ fails, then put $f(a)=0$; otherwise put $f(a)=1$.
By ${\sf RT}^1_2$, we get an infinite homogeneous set $H$ for $f$.
Then given any $a\in H$, compute the value $f(a)=j$.
This $j$ is a solution to $\Sigma^0_2\dmlr$ since for any $a\in H$ we have $\exists b\neg\varphi_j(a,b)$.
\end{proof}


\section{Main Proof}

\subsection{Weihrauch separations}\label{sec:Weihrauch-separation}

In order to prove Theorem \ref{thm:main-theorem}, we will lift separation results on Weihrauch degrees to unprovability results over ${\bf IZF}$.
In this section, we formalize and prove statements on Weihrauch degrees which are needed to prove Theorem \ref{thm:main-theorem}.
We will show the following $\leq_{\gW}^c$-separation results (see also Figure \ref{figure:principles-over-izf}):

\begin{theorem}\label{thm:main-sub-GW}~
\begin{enumerate}
\item ${\sf RDIV}\not\leqcGW{\sf BE}$.
\item ${\sf BE}_\mathbb{Q}\not\leqcGW{\sf RDIV}$.
\item ${\sf IVT}_{\rm lin}\not\leqcGW {\sf RDIV}\times{\sf BE}$.
\item ${\sf BE}\not\leqcGW\Sigma^0_2\dmlr\times\Sigma^0_2\dner$.
\item ${\sf IVT}\not\leqcGW\Sigma^0_2\dmlr\times\Sigma^0_2\dner\times{\sf BE}$.
\end{enumerate}
\end{theorem}

Part of this is used for ${\bf IZF}$-separation; however, not every $\leq_{\gW}^c$-separation result yields an ${\bf IZF}$-separation result.
For instance, $\Sigma^0_2\dmlr$ and $\Sigma^0_2\dner$ are stronger than ${\sf LPO}$, which do not satisfy ${\sf rCT}!$.
Since ${\sf LPO}$ is single-valued, the corresponding realizability validates the axiom of countable choice ${\sf AC}_\om$, and since it realizes ${\sf LLPO}$, combining these two realizes ${\sf WKL}$.
In particular, this cannot separate principles weaker than ${\sf WKL}$ as discussed here.
Therefore, this provides a concrete counterexample that $\leq_{\gW}^c$-separation does not necessarily imply ${\bf IZF}$-separation.

\subsubsection{The strength of choice for finite sets}

We first examine the strength of the game reinforcement of weak K\"onig's lemma for $2$-trees, $({\sf WKL}_{\leq 2})^{\clo}$.
It is not hard to see that $({\sf WKL}_{\leq_2})^{\clo}$ can be reduced to the finite parallelization ${\sf WKL}_{\leq 2}^\ast$ of ${\sf WKL}_{\leq 2}$; see Le Roux-Pauly \cite{LRPa15} (where ${\sf WKL}_{\leq n}\star{\sf WKL}_{\leq m}\leq_{\sf W}{\sf WKL}_{\leq nm}$ is claimed) or Pauly-Tsuiki \cite{PaTs16}: ${\sf C}^{\clo}_{2^\om,\#\leq 2}\equiv_{\sf W}\bigsqcup_n{\sf C}_{2^\om,\#\leq n}$.

\begin{fact}\label{wkl2-game-strength}
$({\sf WKL}_{\leq 2})^{\clo}\equiv_{\sf W}({\sf WKL}_{\leq 2})^\ast$.
\end{fact}

We first show that $({\sf WKL}_{\leq 2})^{\clo}$ is not strong enough for solving weak K\"onig's lemma ${\sf WKL}_{\rm aou}$ for all-or-unique trees:

\begin{prop}\label{prop:sep-main5}
${\sf WKL}_{\rm aou}\not\leqcW({\sf WKL}_{\leq 2})^{\clo}$.
\end{prop}

\begin{proof}
By Fact \ref{wkl2-game-strength}, it suffices to show that ${\sf WKL}_{\rm aou}\not\leqW({\sf WKL}_{\leq 2})^\ast$.
We now consider the following play:
\[
\begin{array}{rcccc}
{\rm I}\colon	& T		&		& x=\langle x_0,\dots,x_n\rangle		&	\\
{\rm II}\colon	&		& \langle S_0,\dots,S_n\rangle	&		& \alpha_{x}
\end{array}
\]

Player II tries to construct a sequence $S=\langle S_0,\dots,S_n\rangle$ of $2$-trees, and a path $\alpha_{x}$ through I's aou-tree $T$, where the second move $\alpha_{x}$ depends on Player I's second move $x$.
As usual, each partial continuous function $f$ is coded as an element $p$ of $(\om\cup\{\bot\})^\om$, so one can read a finite portion $p[s]=\langle p(0),p(1),\dots,p(s-1)\rangle$ of $p$ by stage $s$.
More precisely, $p(s)=\langle n,m\rangle$ indicates that we obtain the information $f(n)\downarrow=m$ at stage $s+1$, but if $p(s)=\bot$ we get no information.

By {\em recursion trick} (Section \ref{sec:recursion-trick}), it is sufficient to describe an algorithm constructing an aou-tree $T$ from given $S,\alpha$ such that Player I wins along the above play for some $x$.
Now, the algorithm is given as follows:
\begin{enumerate}
\item At each stage $s$, either $T$ contains all binary strings of length $s$ or $T$ only has a single node of length $s$.
If we do not act at stage $s$ (while our algorithm works), we always assume that $T$ contains all binary strings of length $s+1$. 
\item Wait for $n$ (the length of $S$ in Player II's first move) being determined.
Then wait for $\alpha_x\upto n+2$ being defined for any correct $x$; that is, any $x$ with $x_i\in [S_i]$.
By compactness, if it happens, it is witnessed by some finite stage.
\item As we consider $n+1$ many $2$-trees $S_0,\dots,S_n$, there are at most $2^{n+1}$ many correct $x$'s, so there are at most $2^{n+1}$ many $\alpha_x$'s.
Clearly, there are $2^{n+2}$ nodes of length $n+2$, so choose a binary string $\sigma$ of length $n+2$ which is different from any $\alpha_x\upto n+2$.
Then, we declare that $[T]=\{\sigma\fr 0^\om\}$, and our algorithm halts.
\end{enumerate}

Note that if the procedure arrives at (3) then $\alpha_x$ cannot be a path through $T$.
Otherwise, the algorithm waits at (2) forever, which means that $\alpha_x$ is not total for some correct $x$, and by (1) we have $T=2^{<\om}$, so Player I obeys the rule.

To apply recursion trick, even if Player II violates the rule, the first move $T$ of Player I needs to obey the rule.
To ensure this, one may assume that we proceed the step (3) in our algorithm only if Player II's partial trees look like $2$-trees.
Moreover, our algorithm performs some action on $T$ only once; hence, in any case, $T$ must be an aou-tree.
Hence, Player I wins.
\end{proof}


\subsubsection{The strength of all-of-unique choice}

We next examine the strength of  the game reinforcement of weak K\"onig's lemma for all-or-unique trees, $({\sf WKL}_{\rm aou})^{\clo}$.
Kihara-Pauly \cite{pauly-kihara2-mfcs} showed that the hierarchy of compositional products of the finite parallelization ${\sf AoUC}^\ast$ of the all-or-unique choice ${\sf AoUC}$ collapses after the second level (where recall that ${\sf AoUC}$ is Weihrauch equivalent to ${\sf WKL}_{\rm aou}$).
In particular, we have the following:

\begin{fact}\label{fact:KP-AoUC-closed}
$({\sf WKL}_{\rm aou})^{\clo}\equiv_{\sf W}({\sf WKL}_{\rm aou})^\ast\star({\sf WKL}_{\rm aou})^\ast$
\end{fact}


We first verify that $({\sf WKL}_{\rm aou})^{\clo}$ is not strong enough for solving weak K\"onig's lemma ${\sf WKL}_{=2}$ for rational $2$-trees:

\begin{prop}\label{prop:sep-main4}
$\wklQ\not\leqcW({\sf WKL}_{\rm aou})^{\clo}$.
\end{prop}

\begin{proof}
By Kihara-Pauly \cite{pauly-kihara2-mfcs}, it suffices to show that ${\sf WKL}_{\leq 2}\not\leqW({\sf WKL}_{\rm aou})^\ast\star({\sf WKL}_{\rm aou})^\ast$.
We now consider the following play:
\[
\begin{array}{rcccccc}
{\rm I}\colon	& T		&		& x=\langle x_0,\dots,x_n\rangle	&		& y=\langle y_0,\dots,y_{\ell(x)}\rangle	&	\\
{\rm II}\colon	&		& \langle S_0,\dots,S_n\rangle	&		& \langle R_x^0,\dots,R_x^{\ell(x)}\rangle	&		& \alpha_{x,y}
\end{array}
\]

Player II tries to construct two sequences $S=\langle S_0,\dots,S_n\rangle$ and $R(x)=\langle R_x^0,\dots,R_x^{\ell(x)}\rangle$ of aou-trees, and a path $\alpha_{x,y}$ through I's rational $2$-tree $T$, where the second move $R(x)$ depends on Player I's second move $x$, and the third move $\alpha_{x,y}$ depends on Player I's second and third moves $(x,y)$.
As in Proposition \ref{prop:sep-main5}, each partial continuous function $f$ is coded as an element $p$ of $(\om\cup\{\bot\})^\om$, so one can read a finite portion $p[s]=\langle p(0),p(1),\dots,p(s-1)\rangle$ of $p$ by stage $s$.

By {\em recursion trick} (Section \ref{sec:recursion-trick}), it is sufficient to describe an algorithm constructing a rational $2$-tree $T$ from given $S,R,\alpha$ such that Player I wins along the above play for some $x$ and $y$.
Now, the algorithm is given as follows:
\begin{enumerate}
\item At each stage $s$, $T$ has two nodes $\ell_s,r_s$ of length $s$, where $\ell_s$ is left to $r_s$.
Let $v_s=\ell_s\land r_s$ be the current branching node of $T$.
If we do not act at stage $s$, we always put $\ell_{s+1}=\ell_s1$ and $r_{s+1}=r_s0$. 
\item Wait for $\alpha_{x,y}(|v_s|)$ being defined for any correct $x$ and $y$; that is, any $(x,y)$ with $x_i\in [S_i]$ and $y_j\in [R^j_x]$.
By compactness, if it happens, it is witnessed by some finite stage.
\item Suppose that it happens at stage $s$.
If $\alpha_{x,y}(|v_s|)=0$ for some correct $x$ and $y$ then $\alpha_{x,y}$ is incomparable with $r_s$, so we put $\ell_{s+1}=r_s0$ and $r_{s+1}=r_s1$.
Otherwise, $\alpha_{x,y}$ for any correct $x$ and $y$ is incomparable with $\ell_s$, so we put $\ell_{s+1}=\ell_s0$ and $r_{s+1}=\ell_s1$.
In any case, this action ensures that $\alpha_{x,y}$ for some correct $x$ and $y$ is not a path through $T$.
\item In order for Player II to win, II needs to make $(x,y)$ incorrect; that is, II needs to remove either $x_i$ from $S_i$ for some $i$ or $y_j$ from $R^j_x$ for some $j$.
\item If Player II decided to remove $x_i$ from $S_i$, then this action forces $S_i$ to have at most one path as $S_i$ is aou-tree.
If it happens, go back to (2).
Note that we can arrive at (5) at most $n+1$ times, since we only have $n+1$ trees $S_0,\dots,S_n$.
\item If Player II decided to remove $y_j$ from $R^j_x$, then this action also forces $R^j_x$ to have at most one path as $R^j_x$ is aou-tree.
By continuity of $R$, there is a clopen neighborhood $C$ of $x$ such that $R_z^j$ has a single path for any $z\in C$.
We may assume that $\ell$ takes a constant value $c$ on $C$ by making such a neighborhood $C$ sufficiently small.
Then we now go to the following (2$^\prime$) with this $C$.
\item[(2$^\prime$)] Wait for $\alpha_{x,y}(v_s)$ being defined for any correct $x\in C$ and $y$; that is, any $(x,y)$ with $x_i\in [S_i]\cap C$ and $y_j\in [R^j_x]$.
By compactness, if it happens, it is witnessed by some finite stage.
\item[(3$^\prime$)] Suppose that it happens at stage $s$.
If $\alpha_{x,y}(|v_s|)=0$ for some correct $x\in C$ and $y$ then $\alpha_{x,y}$ is incomparable with $r_s$, so we put $\ell_{s+1}=r_s0$ and $r_{s+1}=r_s1$.
Otherwise, $\alpha_{x,y}$ for any correct $x$ and $y$ is incomparable with $\ell_s$, so we put $\ell_{s+1}=\ell_s0$ and $r_{s+1}=\ell_s1$.
In any case, this action ensures that $\alpha_{x,y}$ for some correct $x$ and $y$ is not a path through $T$.
\item As before, in order for Player II to win, II needs to make $(x,y)$ incorrect; that is, II needs to remove either $x_i$ from $S_i\cap C$ for some $i$ or $y_j$ from $R^j_x$ for some $j$.
In other words, Player II chooses either (5) or (6$^\prime$) below at the next step.
Here, after going to (5), the parameter $C$ is initialized.
\item[(6$^\prime$)]
The action of this step is the same as (6), but a clopen neighborhood $C'$ of $x$ is chosen as a subset of $C$.
The action of (6$^\prime$) forces some $R^j_z$ to have at most one path for any $z\in C'$.
Then go back to (2$^\prime$).
Note that, as the value $c$ of $\ell$ on $C$ is determined at the step (6), we only have $c+1$ many trees $T^0_z,\dots,T^c_z$ for any $z\in C$, so we can arrive at (6$^\prime$) at most $c$ many times unless going to (5).
\end{enumerate}

Consequently, the procedure arrives (5) or (6$^\prime$) at most finitely often.
This means that Player I eventually defeats Player II; that is, either $\alpha_{x,y}$ is not a path through $T$ for some $x,y$ or else Player II violates the rule by making something which is not an aou-tree.

To apply recursion trick, even if Player II violates the rule, the first move $T$ of Player I needs to obey the rule.
To ensure this, one may assume that we proceed our algorithm at stage $s$ only if Player II's partial trees look like aou-trees at all levels below $s$.
Then if Player II violates the rule then our algorithm performs only finitely many actions, so $T$ must be a rational $2$-tree.
\end{proof}

\subsubsection{The strength of their product}

We now consider the product ${\sf WKL}_{\leq 2}\times{\sf WKL}_{\rm aou}$ of choice for finite sets and all-or-unique choice.
To examine the strength of $({\sf WKL}_{\leq 2}\times{\sf WKL}_{\rm aou})^{\clo}$, we consider the following compactness argument:

In general, to show either $F\leq_{\sf W}G^{\clo}$ or $F\not\leq_{\sf W}G^{\clo}$, we need to consider a reduction game of arbitrary length.
However, if everything involved in this game is compact, then one can pre-determine a bound of the length of a given game.
To be more explicit, let $G(x)\subseteq 2^\om$ be compact uniformly in $x$; that is, there is a multifunction $\hat{G}$ whose graph $\{(x,y):x\in{\rm dom}(\hat{G})\mbox{ and }\hat{G}(x)=y\}$ is compact, and $G$ is a restriction of $\hat{G}$ to some domain. 
For instance, one can take ${\sf WKL}_{\leq 2}$, ${\sf WKL}_{\rm aou}$, ${\sf WKL}_{\rm conv}$, and ${\sf WKL}$.
These examples are actually effectively compact.
Then, consider the reduction game $\mathfrak{G}(F,G)$:
\[
\begin{array}{rccccccc}
{\rm I}\colon	& x_0	&		& x_1	&		& x_2	&	& \dots \\
{\rm II}\colon	&		& \langle j_0,y_0\rangle	&		& \langle j_1,y_1\rangle	& 		& \langle j_2,y_2\rangle	& \dots
\end{array}
\]

Assume that Player II has a winning strategy $\tau$, which yields continuous functions $\tilde{j}$ and $\tilde{y}$ such that $j_n=\tilde{j}(x_0,x_1,\dots,x_n)$ and $y_n=\tilde{y}(x_0,x_1,\dots,x_n)$.
Now, consider the set $P_\tau(x_0)$ of Player I's all plays $x=(x_1,\dots)$ (where the first move $x_0$ is fixed) against II's strategy $\tau$, where $x$ obeys the rule (and $x_n\in 2^\om$ is arbitrary if $j_m=1$ for some $m<n$; that is, Player II has already declared victory).
More precisely, consider
\[P_\tau(x_0)=\{x\in (2^\om)^\om:x_{n+1}\in g(\tilde{y}(x_0,\dots,x_n))\mbox{ or }(\exists m<n)\;\tilde{j}(x_0,\dots,x_m)=1\}.\]

By continuity of $\tilde{j}$ and $\tilde{y}$, and compactness of $g(y)$, the set $P_\tau(x_0)$ is compact.
Hence, again by continuity of $\tilde{j}$ and compactness of $P_\tau(x_0)$, one can easily verify that there is a bound $n$ of the length of any play; that is, for any $x=(x_i)\in P_\tau(x_0)$ there is $m$ such that $\tilde{j}(x_0,\dots,x_m)=1$.
Moreover, $x_0\mapsto n$ is continuous.
If the strategy $\tau$ is computable, and $g(y)$ is effectively compact uniformly in $y$, then $P_\tau(x_0)$ is also effectively compact, and moreover, $x_0\mapsto n$ is computable.
This concludes that 
\[F\leq_{\sf W}G^{\clo}\implies F\leq_{\sf W}\bigsqcup_nG^{(n)},\]
where $G^{(n)}$ is the $n$th iteration of the compositional product; that is, $G^{(1)}=G$ and $G^{(n+1)}=G\star G^{(n)}$.
Furthermore, if ${\dom}(F)$ is effectively compact, one can also ensure that $F\leq_{\sf W}\bigsqcup_nG^{(n)}$ implies $F\leq_{\sf W}G^{(n)}$ for some $n\in\om$.
By the above argument, for instance, one can ensure the following:

\begin{obs}\label{obs:compactness-game-bound}
$({\sf WKL}_{\leq 2}\times{\sf WKL}_{\rm aou})^{\clo}\equiv_{\sf W}\bigsqcup_n({\sf WKL}_{\leq 2}\times{\sf WKL}_{\rm aou})^{(n)}$.
\end{obs}

Then we show that $({\sf WKL}_{\leq 2}\times{\sf WKL}_{\rm aou})^{\clo}$ is not strong enough for solving ${\sf WKL}_{\rm clop}$, weak K\"onig's lemma for clopen convex trees.

\begin{prop}\label{prop:sep-main6}
${\sf WKL}_{\rm clop}\not\leqcW({\sf WKL}_{\rm aou}\times{\sf WKL}_{\leq 2})^{\clo}$.
\end{prop}

\begin{proof}
By Observation \ref{obs:compactness-game-bound}, it suffices to show that ${\sf WKL}_{\rm clop}\not\leqW\bigsqcup_n({\sf WKL}_{\rm aou}\times{\sf WKL}_{\leq 2})^{(n)}$.
We now consider the following play:
\[
{\small
\begin{array}{rccccccc}
{\rm I}\colon	& T		&		& x(1)=\langle x^1_A,x^1_B\rangle		&	 & \dots & 
x(n)=\langle x^n_A,x^n_B\rangle &  \\
{\rm II}\colon	&		& n,\langle A,B\rangle	&		& \langle A_{x(1)},B_{x(1)}\rangle & \dots &
& \alpha_{x(1)\dots x(n)}
\end{array}
}
\]

Player II tries to construct a sequence of pairs of aou-trees $A_\bullet$ and $2$-trees $B_{\bullet}$, and a path $\alpha_{x(1)\dots x(n)}$ through a basic clopen tree $T$, where the $i$th move $\langle A_{x(1)\dots x(i-1)},B_{x(1)\dots x(i-1)}\rangle$ depends on Player I's moves $x(1),\dots,x(i-1)$.
As in Proposition \ref{prop:sep-main5}, each partial continuous function $f$ is coded as an element $p$ of $(\om\cup\{\bot\})^\om$, so one can read a finite portion $p[s]=\langle p(0),p(1),\dots,p(s-1)\rangle$ of $p$ by stage $s$.

By {\em recursion trick} (Section \ref{sec:recursion-trick}), it is sufficient to describe an algorithm constructing a basic clopen tree $T$ from given $n,\langle A_\bullet,B_\bullet\rangle,\alpha$ such that Player I wins along the above play for some $x(1),\dots,x(n)$.
Now, the algorithm is given as follows:

\begin{enumerate}
\item At each stage $s$, there is a binary string $\sigma$ such that $T$ contains all binary strings of length $s$ extending $\sigma$.
If we do not act at stage $s$, we always assume that $T$ contains all binary strings of length $s+1$ extending $\sigma$.
\item Wait for $n$ (in Player II's first move) being determined.
Then wait for $\alpha_{x(1)\dots x(n)}\upto n+2$ being defined for any correct $x(1),\dots,x(n)$, which means that $x^{i+1}_A\in A_{x(1)\dots x(i)}$ and $x^{i+1}_B\in B_{x(1)\dots x(i)}$.
By compactness, if it happens, it is witnessed by some finite stage.
\item For any $i$, there are two possibilities for $x_B^i$ for a given $B_{x(i-1)}$ since $B_\bullet$ is a $2$-tree.
So, we consider a non-deterministic play of Player I, which takes both possibilities for $x_B^i$, but for $x_A^i$ chooses just one element.
In other words, we consider (at most) $2^n$ many $x=x(1)\dots x(n)$'s, so there are at most $2^{n+1}$ many $\alpha_x$'s.
Clearly, there are $2^{n+2}$ nodes of length $n+2$, so choose a binary string $\sigma$ of length $n+2$ which is different from any $\alpha_x\upto n+2$.
Then, we declare that $T$ only contains binary strings of length $s+1$ extending $\sigma$ at stage $s+1$.
This ensures that $\alpha_x$ is not a path through $T$.
\item In order for Player II to win, II needs to remove $x^i_A$ from $A_{x(1)\dots x(i-1)}$ for some $i$.
This action forces $A_{x(1)\dots x(i-1)}$ to have at most one path as $A_\bullet$ is an aou-tree.
Player I re-chooses $x^i_A$ as a unique element in $A_{x(1)\dots x(i-1)}$, and go back to (2) with $k(n+2)$ instead of $n+2$, where $k$ is the number of times which we have reached the step (4) so far.
\end{enumerate}

When we arrive at (4), Player II need to declare that some of $A_\bullet$ has at most one path.
Thus, it is easy to ensure that the algorithm can arrive at (4) at most finitely often.
Hence, $x^i_A$ is eventually stabilized for any $i$, and by our action at (3), this ensures that $\alpha_x$ cannot be a path through $T$.
The string $\sigma$ is also eventually stabilized, and thus $T$ becomes a clopen tree.
Thus, Player I obeys the rule.

To apply recursion trick, even if Player II violates the rule, the first move $T$ of Player I needs to obey the rule.
To ensure this, one may assume that we proceed our algorithm at stage $s$ only if Player II's partial trees look like pairs of aou-trees and $2$-trees at all levels below $s$.
Then if Player II violates the rule then our algorithm performs only finitely many actions, so $T$ must be a basic clopen tree.
In any case, Player I obeys the rule, and so Player I defeats Player II.
\end{proof}


\subsubsection{The product of de Morgan's law and the double negation elimination}

We now examine the strength of $({\sf RT}^1_2\times{\limN})^{\clo}$, the game closure of the product of Ramsey's theorem for singletons and two-colors (the infinite pigeonhole principle) and the discrete limit operation.
We say that a partial finite-valued function $F\pcolon\om^\om\rightrightarrows\om^\om$ has a {\em semicontinuous bound} if there is a partial continuous function $g\pcolon\om^\om\times\om\to\om^\om$ such that $F(x)\subseteq\{g(x,n):n<b(x)\}$ for any $x\in{\dom}(F)$, where $b\pcolon\om^\om\to\om$ is lower semicontinuous.

\begin{lemma}\label{lem:RT-semi-cont-bound}
$({\sf RT}^1_2\times{\limN})^{\clo}$ has a semicontinuous bound.
\end{lemma}

\begin{proof}
Let $(x,\tau)$ be an instance of $({\sf RT}^1_2\times{\limN})^{\clo}$.
Consider a play of a reduction game according to II's strategy $\tau$:
\[
{\small
\begin{array}{rcccccccc}
{\rm I}\colon	& x		&		& \sigma(0)=(j_0,k_0)	&		& \sigma(1)=(j_1,k_1)	 & &	\dots \\
{\rm II}\colon	&		& (c,\alpha)		&		& (c_{\sigma(0)},\alpha_{\sigma(0)})	&		 &  (c_{\sigma(0)\sigma(1)},\alpha_{\sigma(0)\sigma(1)})&  \dots
\end{array}
}
\]

Here, Player II's moves are automatically generated from the strategy $\tau$.
For each round, Player I needs to answer a query $c_\bullet$ to ${\sf RT}^1_2$ and a query $\alpha_\bullet$ to ${\limN}$ made by Player II.
As an answer to the ${\sf RT}^1_2$-query $c_\bullet$ is either $0$ or $1$, and to the ${\limN}$-query $\alpha_\bullet$ is a natural number, one may assume that Player I's moves are restricted to $2\times\om$.
We define the {\em tree $T$ of Player I's possible plays} as $T=(2\times\om)^{<\om}$.
Roughly speaking, what we want to do is to describe a computable approximation of a $G$-strategy tree introduced in Definition \ref{def:strategy-tree}, where $G={\sf RT}^1_2\times{\limN}$.

We view each node $\sigma\in T$ as a sequence such that $j(n)\in 2$ and $k(n)\in\om$ alternatively appear (i.e., $\sigma(2n)\in\{0,1\}$ and $\sigma(2n+1)\in\om$) rather than a sequence of pairs from $2\times\om$.
Then, a node of $T$ of even length is called a {\em ${\sf RT}^1_2$-node}, and a node of odd length is called a {\em ${\limN}$-node}.
The strategy $\tau$ automatically assigns a query $c_\sigma$ to each ${\sf RT}^1_2$-node $\sigma\in T_s$, where $c_\sigma\colon\N\to 2$ is a two-coloring.
Then, each ${\sf RT}^1_2$-node $\sigma$ has at most two immediate successors $\sigma0,\sigma 1$, and moving to the node $\sigma i$ indicates that Player I chose $i$ as an answer to $c_\sigma$.
Similarly, the strategy $\tau$ automatically assigns a query $\alpha_\sigma$ to each ${\limN}$-node $\sigma\in T_s$, where $\alpha_\sigma\in\om^\om$ is a (convergent) sequence of natural numbers.

For a ${\limN}$-node $\sigma$, as mentioned above, Player I's next move is restricted to $\om$ (whenever I obeys the rule), but by the uniqueness of the limit, there is an exactly one correct answer.
Since ${\limN}$ is the discrete limit operation, if Player II obeys the rule, then $\alpha_\sigma$ must stabilize, that is, $\alpha_\sigma(s)=\alpha_\sigma(t)$ for any sufficiently large $s,t$, and this value is the correct answer to $\alpha_\sigma$.
Thus, if $\alpha_\sigma(s)=k$, we say that $k$ is the {\em current true outcome} of the ${\limN}$-node $\sigma$ at stage $s$, and only the edge $[\sigma,\sigma k]$ is {\em passable} at this stage.

Similarly, for a ${\sf RT}^1_2$-node $\sigma$, we declare that the edge $[\sigma,\sigma i]$ is {\em passable} if $i$ looks like a correct answer (which means that $c_\sigma(n)=i$ for infinitely many $n$).
More precisely:
First, we say that $\sigma$ is {\em accessible at stage $s$} if $|\sigma|<s$ and all edges in the path $[\langle\rangle,\sigma]$ are passable at stage $s$ (where $\langle\rangle$ is the root of $T$).
Each node $\sigma$ has its own clock $t_\sigma$; that is, let $t_\sigma(s)$ be the number of stages $t\leq s$ such that $[\langle\rangle,\sigma]$ is accessible at stage $t$.
In short, this clock $t_\sigma$ operates only when $\sigma$ is accessible.
Then, the edge $[\sigma,\sigma i]$ is {\em passable at stage $s$} if $c_\sigma(t_\sigma(s))=i$.
Clearly, at each stage $s$, there is a unique accessible node of length $s$.

If Player I obeys the rule, then, as $(x,\tau)$ is an instance of $({\sf RT}^1_2\times{\limN})^{\clo}$, Player II (according to the strategy $\tau$) declares victory at some round.
This means that if we restrict the tree $T$ to the nodes $\sigma$ at which Player I obeys the rule (i.e., Player I gives a correct answer to each query made by Player II), then the restriction forms a well-founded binary tree; that is, a finite binary tree.
This is exactly the unique $G$-strategy tree $R$ given by Player II's fixed strategy $\tau$ (see Definition \ref{def:legal-tree}).
Then, any $\sigma\not\in R$ is never accessible after some stage.
Indeed, if $\sigma\in R$ and $\sigma i\not\in R$, then the edge $[\sigma,\sigma i]$ is never passable after some stage, and there are only finite such nodes as $R$ is finite.

Let $\theta_s$ be the unique accessible node of length $s$ at stage $s$.
The above argument shows that after some stage $s_0$, either $\theta_t\subseteq R$ holds or $\theta_t$ extends a leaf $\theta_t\upto\ell$ of $R$ for any $t\geq s_0$.
The latter means that, after Player I's moves $\theta_t\upto\ell$, Player II declares victory.
For $s\geq s_0$, let $\tilde{\theta}_s$ be the longest initial segment of $\theta_s$ which is contained in $R$.
Then, $\{\tilde{\theta}_s:s\geq s_0\}\subseteq R$ is finite.
At stage $s$, for each $\sigma=\theta_s\upto t$ with $t\leq s$, check if Player II (according to the strategy $\tau$) declares victory with some value $u_\sigma$ at the next round.
If so, enumerate $\tilde{u}_s:=u_\sigma$ into an auxiliary set $B$.
Note that if $s\geq s_0$ then $\sigma=\tilde{\theta_s}\in R$, so $u_\sigma\in({\sf RT}^1_2\times{\limN})^{\clo}(x,\tau)$.
Hence, $\{\tilde{u}_s:s\geq s_0\}$ is finite since $\tilde{u}_s$ only depends on $\tilde{\theta}_s$.
Thus, $B$ is finite since clearly $\{\tilde{u}_s:s<s_0\}$ is finite, too.
In this way, we get an $(x,\tau)$-computable enumeration of a finite set $B$.
By our construction, we also note that if $\sigma\in R$ then there are infinitely many $t$ such that $\sigma\preceq\theta_t$.
Thus, $B$ contains all possible outputs $u_\sigma$ of $({\sf RT}^1_2\times{\limN})^{\clo}(x,\tau)$.
Consequently, (an enumeration of) $B$ gives a semicontinuous bound for ${\sf RT}^1_2\times{\limN}$.
\end{proof}

We apply the above lemma to show that $({\sf RT}^1_2\times{\limN})^{\clo}$ is not strong enough for solving weak K\"onig's lemma for $2$-trees:

\begin{prop}\label{prop:sep-main23}
If $F\pcolon\om^\om\rightrightarrows\om$ has a semi-continuous bound, then ${\sf WKL}_{\leq 2}\not\leqcW F$.
In particular, ${\sf WKL}_{\leq 2}\not\leqcW({\sf RT}^1_2\times{\limN})^{\clo}$.
\end{prop}

\begin{proof}
If ${\sf WKL}_{\leq 2}\leqcW F$, where $F$ has a semi-continuous bound $(\tilde{g},\tilde{b})$, then there are continuous functions $H,K$ such that if $\tilde{T}$ is a $2$-tree then,  there is $n<b:=\tilde{b}(H(\tilde{T}))$ such that $\tilde{g}(H(\tilde{T}),n)$ is defined, and $g(n):=K(\tilde{T},\tilde{g}(H(\tilde{T}),n))$ is a path through $\tilde{T}$.
If we describe an algorithm constructing $T$ from $(g,b)$ (with a parameter $\tilde{T}$), by Kleene's recursion theorem, $\tilde{T}$ can be interpreted as self-reference; that is, we may assume that $T=\tilde{T}$.
This is the simplest version of {\em recursion trick} (Section \ref{sec:recursion-trick}).

Thus, it is sufficient to describe an algorithm constructing a $2$-tree $T$ from given $(g,b)$:
\begin{enumerate}
\item First we declare that each $n\in\N$ is active. 
At each stage $s$, $T$ has two nodes $\ell_s,r_s$ of length $s$, where $\ell_s$ is left to $r_s$.
Let $v_s=\ell_s\land r_s$ be the current branching node of $T$.
If we do not act at stage $s$, we always put $\ell_{s+1}=\ell_s1$ and $r_{s+1}=r_s0$. 
\item Then, ask if $g(n)(|v_s|)$ is already defined by stage $s$ for some active $n<b_s$, where $b_s$ is the stage $s$ approximation of $b$ (where recall that $\tilde{b}$ is lower semicontinuous).
\item If such $n$ exists, then the least such $n$ receives attention.
For this $n$, if $g(n)(|v_s|)=0$, then $g(n)$ is incomparable with the rightmost node $r_s$, so we put $\ell_{s+1}=r_s0$ and $r_{s+1}=r_s1$.
Similarly, if $g(n)(|v_s|)=1$, then $g(n)$ is incomparable with the leftmost node $\ell_s$, so we put $\ell_{s+1}=\ell_s0$ and $r_{s+1}=\ell_s1$.
This procedure ensures that $g(n)$ is incomparable with $\ell_{s+1}$ and $r_{s+1}$.
Then, we declare that $n$ is not active anymore.
\end{enumerate}

Assume that Player II obeys the rule; then $H(\tilde{T})\in{\rm dom}(F)$, so $b=\tilde{b}(H(\tilde{T}))$ is defined, i.e., $(b_s)$ converges to a finite value $b$.
Then, since there are only finitely many $n$ such that $g(n)$ is total and $n<b$, such $n$ receives attention at some stage $s$.
By our strategy described above, for any such $n$, if $p$ is a path through $T$ then we must have $g(n)\not=p$.
This contradicts our assumption that $g(n)$ is a path through $T$ for some $n<b$.

To apply recursion trick, even if Player II violates the rule (i.e., $(b_s)$ does not converge), the first move $T$ of Player I needs to obey the rule.
In this case, our algorithm may arrive at (3) infinitely often, but this just implies that $T$ converges to a tree which has a unique path.
In any case, $T$ is a $2$-tree.
Hence, we have ${\sf WKL}_{\leq 2}\not\leq_{\sf W}^cF$.
\end{proof}


\subsubsection{More products}

We finally examine the strength of the game closure of ${\sf RT}^1_2\times{\limN}\times{\sf WKL}_{\leq 2}$.
We show that it is not strong enough for solving weak K\"onig's lemma ${\sf WKL}_{\rm conv}$ for convex trees.

\begin{prop}\label{prop:sep-main78}
${\sf WKL}_{\rm conv}\not\leqcW({\sf RT}^1_2\times{\limN}\times{\sf WKL}_{\leq 2})^{\clo}$.
\end{prop}

\begin{proof}
We now consider the following play:
\[
{\small
\begin{array}{rccccccc}
{\rm I}\colon	& T		&		& \sigma(0)=(j_0,k_0,x_0)	&		& \sigma(1)=(j_1,k_1,x_1)	 & \dots \\
{\rm II}\colon	&		& (c,\alpha,B)		&		& (c_{\sigma(0)},\alpha_{\sigma(0)},B_{\sigma(0)})	&		 & \dots
\end{array}
}
\]

Again, by {\em recursion trick} (Section \ref{sec:recursion-trick}), it is sufficient to describe an algorithm constructing a basic convex tree $T$ from Player II's given strategy $c_\bullet,\alpha_\bullet,B_\bullet,\dots$ such that Player I wins along the above play for some $j_0,k_0,x_0,j_1,k_1,x_1,\dots$
The difficulty here is again lack of compactness.

Recall the tree of I's possible plays from Proposition \ref{lem:RT-semi-cont-bound} (an approximation of a strategy tree in Definition \ref{def:strategy-tree}).
In the present game, each node of the tree essentially codes Player I's play $\langle j_0,k_0,x_0,j_1,k_1,x_1,\dots\rangle$; however, now our new tree $\mathcal{O}$ is slightly different from the previous one.
The first entry is $j=j_0$, which is an answer to the query $c$, and the second entry is $k=k_0$, which codes an answer to the query $\alpha$ (assigned to the node $\langle j\rangle$).

However, although the $2$-tree $B$ (assigned to the node $\langle jk\rangle$) has at most two paths, we have a lot of candidates for an answer $x_0\in 2^\om$.
We assign infinitely many possible outcomes $(\infty,0,1,\dots)$ to the node $\langle jk\rangle$.
The outcome $\infty$ indicates that $B$ has at most one path, and the outcome $n$ indicates that $B$ has exactly two paths which branch at height $n$.
In the latter case, Player I has two possible moves: choosing the left path $x$ or choosing the right path $x'$.
Therefore, instead of outcome $n$, we actually handle outcomes $n{\tt L}$ and $n{\tt R}$.
The outcome $n{\tt L}$ ($n{\tt R}$, respectively) suggests that Player I chooses the left path $x$ (the right path $x'$, respectively) at the branch of height $n$.
Then, the next entry of Player I's move may be $j_1$ and $j'_1$ of answers to two different queries $c_{jkx}$ and $c_{jkx'}$.
The tree $\mathcal{O}$ of outcomes is constructed by continuing this procedure.

We call a node of length $3n$ an {\em ${\sf RT}^1_2$-node}, a node of length $3n+1$ an {\em ${\limN}$-node}, and a node of length $3n+2$ a {\em ${\sf WKL}_{\leq 2}$-node}.
Note that an ${\sf RT}^1_2$-node has at most finitely many immediate successors, while a ${\limN}$-node and a ${\sf WKL}_{\leq 2}$-node has infinitely many immediate successors.
For each query on a ${\sf RT}^1_2$-node, the outcomes are defined as in the proof of Proposition \ref{lem:RT-semi-cont-bound}.
For each query $\alpha$ on a ${\limN}$-node $\gamma$, we assign infinitely many possible finite outcomes $n\in\N$ to this query, where the outcome $n$ indicates that $\alpha$ stabilizes after $n$; that is, $n$ is the least number such that $\alpha(k)=\alpha(n)$ for any $k\geq n$ (which is slightly different from Proposition \ref{lem:RT-semi-cont-bound}).
To be precise, the outcome $n$ is currently true if $\alpha$ stabilizes between stages $n$ and $t$; that is, $n$ is the least number such that $\alpha(k)=\alpha(n)$ whenever $n\leq k\leq s$.
In this case, $[\gamma,\gamma n]$ is passable at stage $s$.
Then Player I's move corresponds to $\gamma n$ is $\alpha(n)$.
Then, Player I's corresponding move to the outcome $n$ is $\alpha(n)$.

The outcomes of each query $B$ on a ${\sf WKL}_{\leq 2}$-node $\gamma$ are the same as described as above:
The outcome $\infty$ indicates that $B$ has at most one path, and the outcome $n$ indicates that $B$ has exactly two paths which branch at height $n$.
As in Proposition \ref{lem:RT-semi-cont-bound}, each node $\gamma\in\mathcal{O}$ has its own clock $t_\gamma$.
Each ${\sf WKL}_{\leq 2}$-query $B$ (which is assigned to a ${\sf WKL}_{\leq 2}$-node $\gamma$) has a unique {\em current true outcome} at stage $t=t_\gamma(s)$ in the following sense:
The $2$-tree $B$ has exactly two strings $\ell_t$ and $r_t$ of length $t$.
If $\ell_t$ extends $\ell_{t-1}$ and $r_t$ extends $r_{t-1}$, then the finite outcome $n$ is currently true, where $n$ is the height where $\ell_{t}$ and $r_t$ branch; otherwise $\infty$ is currently true.
If $n$ is currently true, then $[\gamma,\gamma n{\tt L}]$ and $[\gamma,\gamma n{\tt R}]$ are {\em passable} at $s$.
If $\infty$ is currently true, then $[\gamma,\gamma\infty]$ is passable at $s$.
The term ``accessible'' is defined in the same way as in Proposition \ref{lem:RT-semi-cont-bound}.

Let us describe an algorithm $\Psi$ that gives Player I's next move $\Psi(\gamma a)$ based on $\gamma a$, an immediate successor of $\gamma$.
The description of this algorithm itself depends only on the node $\gamma a$ and does not depend on the stage $s$ (of course, the algorithm itself proceeds according to the stage progression).
First, we look at a query $B$ assigned to $\gamma$.
If $B$ is not a $2$-tree (i.e., if the number of nodes at some height is not $2$), then either $\gamma$ is not in a true position or Player II has violated the rules, so Player I does not need to decide the next move, and the algorithm can be abandoned.
Therefore, we proceed with the algorithm assuming that $B$ is a $2$-tree.

Player I's move $\Psi(\gamma n{\tt L})$ assigned to $\gamma n{\tt L}$ is given by the following algorithm.
Wait until a branching point $v$ of height $n$ of $B$ is found, and then search for the unique path of $B$ that extends the left side $v0$ of that branch.
If $v$ is the last branching point in $B$, then the path extending $v_0$ in $B$ is unique, so the path is successfully computed.
If a new branch point in $B$ that properly extends $v$ is found, the algorithm terminates the computation.
In this case, $n{\tt L}$ is no longer a true outcome, so this algorithm is unnecessary.
An algorithm that assigns Player I's move $\Psi(\gamma n{\tt R})$ to $\gamma n{\tt R}$ also searches for the unique path in $B$ that extends the right side $v_1$ of the branching point $v$ of height $n$ using the same method as above.

Player I's move $\Psi(\gamma \infty)$ assigned to $\gamma\infty$ is given by the following algorithm.
At each stage, output the longest branching point $v \in B$.
Since $B$ is a $2$-tree, i.e., it has exactly two nodes at each height, as long as the longest branching point is not updated, the two nodes extend $v0$ and $v1$, respectively.
If the unique infinite path thought $B$ extends $v i$, then there is no path that extends $v(1-i)$, so $B$ will no longer have a node that extends $v(1-i)$ at some height.
However, since $B$ is a $2$-tree, it must have two nodes that extend $vi$ at that height.
Thus, $B$ must have a branching point longer than $v$.
Repeating this process, we see that $B$ has branching points at infinitely many heights.
Note that if the infinite path in $B$ is unique, it must pass through all branching points.
Therefore, this algorithm generates an infinite path through $B$.
Furthermore, since the branching points are updated infinitely many times, if $\gamma$ is accessible at infinitely many stages, then $[\gamma,\gamma\infty]$ is passable, so $\gamma\infty$ is accessible, at infinitely many stages.

Let $\mathcal{A}_s$ be the subtree of $\mathcal{O}$ consisting of all accessible nodes at each stage $s$.
Each ${\sf RT}^1_2$-node has a unique passable edge at each stage.
The same applies to $\lim_\om$-nodes.
For a ${\sf WKL}_{\leq 2}$-node $\gamma$, either $[\gamma,\gamma\infty]$ is passable, or both $[\gamma,\gamma n{\tt L}]$ and $[\gamma,\gamma n{\tt R}]$ are passable.
Therefore, $\mathcal{A}_s$ is a binary tree.


Now, the tree $\mathcal{A}_s$ of outcomes is converted to a strategy tree through the above translation algorithm $\Psi$.
That is, each node of $\mathcal{A}_s$ generates Player I's play by applying $\Psi$ to each edge.
If $\gamma \in \mathcal{A}_s$ is a true node (i.e., it continues to select true outcomes), then Player I follows the rules.
Now, let $\gamma$ be a true path through $\mathcal{A}_s$.
As Player I obeys the rule along $\gamma$, if Player II wins, then Player II declares victory along $\gamma$.
This means that the tree $\mathcal{A}$ of true nodes is well-founded; hence, finite as $\mathcal{A}_s$ is binary.

Inductively, we can guarantee that $\mathcal{A}_s\subseteq\mathcal{A}$ for infinitely many stages $s$.
That is, at such stages, the current true outcome at stage $s$ for each $\gamma\in\mathcal{A}_s$ is actually the true outcome.
If $\gamma$ is a ${\sf RT}^1_2$-node or a $\lim_\om$-node, the current true outcome at a sufficiently large stage is always the true outcome.
If $\gamma$ is a ${\sf WKL}_{\leq 2}$-node, let $B$ be a query assigned to $\gamma$.
If the height of the last branch of $B$ is $n$, then the current true outcome at a sufficiently large stage is $n$, which is the true outcome.
If $B$ has only one path, then, as seen above, $B$ branches at infinitely many heights, so the true outcome $\infty$ becomes a current true outcome at infinitely many stages.
In nodes in $\mathcal{O}$ that extend $\gamma\infty$, since each clock is restricted to $t_{\gamma\infty}$, the same argument guarantees that the subsequent successors also have current true outcomes which are actually true at infinitely many stages.
Thus, we obtain that $\mathcal{A}_s \subseteq \mathcal{A}$ at infinitely many stages $s$.

Now, the algorithm constructing a convex tree $T$ is given as follows:
\begin{enumerate}
\item At each stage $s$, there is a binary string $\sigma$ such that $T$ contains all binary strings of length $s$ extending $\sigma$.
If we do not act at stage $s$, we always assume that $T$ contains all binary strings of length $s+1$ extending $\sigma$.
\item Wait for some stage when Player II declares victory with $(u_i)_{i\in I}$ along all plays $(\gamma_i)_{i\in I}$ on paths through $\mathcal{A}_s$, and $u_i(|\sigma|+|I|)$ is defined for any $i\in I$.
\item Then choose $\tau$ extending $\sigma$ such that $\tau(|\sigma|+|I|)\not=u_i(|\sigma|+|I|)$ for any $i\in I$.
We declare that, at stage $s+1$, all binary strings in $T$ of length $s+1$ extend $\tau$.
This action ensures that $u_i$ is not a path through $T$ for any $i\in I$.
\end{enumerate}

If Player II's strategy is winning, then by the above observation that $\mathcal{A}_s \subseteq \mathcal{A}$ occurs infinitely many often, there must be a stage at which action (3) is performed.
If action (3) is performed with a ``correct $\mathcal{A}_s\subseteq\mathcal{A}$,'' this clearly leads to Player II's defeat (Player I's win).
%
However, one may think that an ``incorrect $\mathcal{A}_s$'' taking action (3) may prevent a correct $\mathcal{A}_s$ from taking action (3).
We show that this is not the case.

Let $\gamma\in \mathcal{A}_s$ be the first node that does not return a true outcome.
If it is a ${\sf RT}^1_2$- or $\limN$-node, the problem will not occur after some stage because the outcome is correct at almost all stages.
The same applies to ${\sf WKL}_{\leq 2}$-nodes where $n\in\om$ is the true outcome.
The problem arises when $\infty$ is the true outcome for ${\sf WKL}_{\leq 2}$-nodes.
In this case, $[\gamma,\gamma n]$ becomes passable for different $n$ at infinitely many stages.
That is, there are infinitely many stages $s$ with $\gamma n\in\mathcal{A}_s$ for some $n\in\om$.

Suppose that action (3) is performed with such an $\mathcal{A}_s$.
Let $\gamma$ be labeled with a query $B$, and suppose that the height of the current last branching node $v$ of $B$ is $n$.
Then $\Psi$ converts $\gamma n$ into a string $x$ that extends $v0$ and a string $x'$ that extends $v1$, which become initial segments of Player I's next moves.
However, the fact that action (3) was performed at stage $s$ means that Player II has declared victory in the game using only the information about $x$ and $x'$ before reaching a new branch in $B$.
Now, $\infty$ is the correct outcome, and $\Psi(\gamma\infty)$ gives the unique path through $B$, which extends either $x$ or $x'$.
Therefore, an output of the game when Player II declares vitory is an extension of some $u_i\upto|\sigma|+|I|$ when action (3) was performed with $\mathcal{A}_s$.
By action (3), it is guaranteed that $u_i$ is not a path through $T$.
Thus, even if action (3) is performed with an incorrect $\mathcal{A}_s$, our requirements are satisfied.
Therefore, in any cases, Player II cannot win.

To apply recursion trick, even if Player II violates the rule, the first move $T$ of Player I needs to obey the rule.
In this case, our algorithm may arrive at (3) infinitely often, but this just implies that $T$ converges to a tree which has a unique path.
In any case, $T$ is a convex tree, which means that Player I obeys the rule, and so Player I defeats Player II.
\end{proof}


Now, Theorem \ref{thm:main-sub-GW} follows from Propositions \ref{prop:sep-main5}, \ref{prop:sep-main4}, \ref{prop:sep-main6}, \ref{prop:sep-main23}, and \ref{prop:sep-main78}.

\subsection{Lifting Weihrauch separations to {\bf IZF}-separations}\label{sec:main-section-proof}

We now translate Weihrauch separation results in Section \ref{sec:Weihrauch-separation} into the context of realizability.
Recall that $\varphi$ is $\jump$-realizable over $\mathbb{P}$ (or $\varphi$ is realizable over the site $(\mathbb{P},\jump)$) if there is $e\in {\tt P}$ such that $e\reap \varphi$.
We say that $\varphi$ is {\em boldface $\jump$-realizable} over $\mathbb{P}$ if there is $e\in\mathbf{P}$ such that $e\reap \varphi$.
Let us recall the following key observations:
\begin{itemize}
\item 
By definition, $\neg\varphi$ is $\jump$-realizable over $\mathbb{P}$ if and only if $\varphi$ is not boldface $\jump$-realizable over $\mathbb{P}$.
\item For any partial multifunction $F$, by Lemma \ref{lem:idempotent} and Theorem \ref{thm:IZF-realizability}, ${\bf IZF}$ is $F^{\clo}$-realizable.
\item For any $j$-operator $\jump\leq_{\sf W}{\sf WKL}$, a tidy $\forall\exists$-principle $\varphi$ is $\jump$-realizable iff $F_\varphi\leq_{\sf W}\jump$ by Lemmas \ref{lem:below-WKL-rCT} and \ref{lem:main-lemma}.
One can also see that $\varphi$ is boldface $\jump$-realizable iff $F_\varphi\leq_{\sf W}^c\jump$.
\end{itemize}

Therefore, we hereby declare the following:
\begin{itemize}
\item All the principles we deal with below are tidy $\forall\exists$-principles.
\item Below we only consider partial multifunction $F\leq_{\sf W}{\sf WKL}$.
In this case, we have $\jump:=F^{\clo}\leq_{\sf W}{\sf WKL}^{\clo}\equiv_{\sf W}{\sf WKL}$ (see e.g.~\cite[Proposition 4.9]{HiJo16}).
\end{itemize}




\begin{proof}[Proof of Theorem \ref{thm:main-theorem}]
We prove the following results which are stronger than Theorem \ref{thm:main-theorem}:
\begin{enumerate}
\item ${\bf IZF}+\Sigma^0_1\dmlr+\neg{\sf RDIV}+\neg{\sf BE}_\mathbb{Q}$ is ${\sf LLPO}^{\clo}$-realizable. \label{item:main1}
\item ${\bf IZF}+\neg\Pi^0_1\lemr+{\sf BE}_\mathbb{Q}+\neg{\sf BE}+\neg{\sf RDIV}$ is $(\wklQ)^{\clo}$-realizable.\label{item:main2}
\item ${\bf IZF}+\neg\Pi^0_1\mbox{-}{\sf LEM}_\mathbb{R}+{\sf RDIV}+\neg{\sf BE}_\mathbb{Q}$ is $({\sf WKL}_{\rm aou})^{\clo}$-realizable. \label{item:main3}
\item ${\bf IZF}+\neg\Pi^0_1\lemr+{\sf IVT}_{\rm lin}+\neg{\sf BE}$ is $({\sf WKL}_{\rm clop})^{\clo}$-realizable. \label{item:main4}
\item ${\bf IZF}+{\sf BE}+\neg{\sf RDIV}$ is $({\sf WKL}_{\leq 2})^{\clo}$-realizable. \label{item:main5}
\item ${\bf IZF}+\neg\Pi^0_1\lemr+{\sf BE}_\mathbb{Q}+{\sf RDIV}+\neg{\sf IVT}_{\rm lin}+\neg{\sf BE}$ is $(\wklQ\times{\sf WKL}_{\rm aou})^{\clo}$-realizable.\label{item:main6}
\item ${\bf IZF}+\neg\Pi^0_1\mbox{-}{\sf LEM}_\mathbb{R}+{\sf BE}+{\sf RDIV}+\neg{\sf IVT}_{\rm lin}$ is $({\sf WKL}_{\leq 2}\times{\sf WKL}_{\rm aou})^{\clo}$-realizable. \label{item:main7}
\item ${\bf IZF}+\neg\Pi^0_1\mbox{-}{\sf LEM}_\mathbb{R}+{\sf BE}+{\sf IVT}_{\rm lin}+\neg{\sf IVT}$ is $({\sf WKL}_{\leq 2}\times{\sf WKL}_{\rm clop})^{\clo}$-realizable. \label{item:main8}
\item ${\bf IZF}+\neg\Pi^0_1\mbox{-}{\sf LEM}_\mathbb{R}+{\sf IVT}+\neg{\sf WKL}$ is $({\sf WKL}_{\sf conv})^{\clo}$-realizable. \label{item:main9}
\end{enumerate}

\noindent
(\ref{item:main1})
We use ${\sf LLPO}^{\clo}$-realizability (i.e., Lifshitz realizability; see below) to realize $\Sigma^0_1\mbox{-}{\sf DML}_\mathbb{R}+\neg{\sf RDIV}+\neg{\sf BE}_\mathbb{Q}$.

Recall that de Morgan's law $\Sigma^0_1\mbox{-}{\sf DML}_\mathbb{R}$ is formulated as ${\sf LLPO}$, so it is clear that ${\sf LLPO}^{\clo}$-realizability validates $\Sigma^0_1\mbox{-}{\sf DML}_\mathbb{R}$.
Moreover, ${\sf LLPO}^{\clo}$ is known to be Weihrauch equivalent to the {\em choice principle ${\sf K}_\N$ for compact sets in $\N$}, which is a multifunction that, given a bound $b\in\N$ and a sequence $a=(a_i)$ of natural numbers, if $V_{a,b}=\{n<b:(\forall i)\;n\not=a_i\}$ is nonempty, return any element from $V_{a,b}$.

\begin{fact}[\cite{PaTs16}]\label{fact:LLPO-closed-KN}
${\sf LLPO}^{\clo}\equiv_{\sf W}{\sf K}_\N$.
\end{fact}

To prove (\ref{item:main1}), we need to show that ${\sf LLPO}^{\clo}$-realizability refutes ${\sf BE}$ (equivalent to ${\sf WKL}_{\leq 2}$) and ${\sf RDIV}$ (equivalent to ${\sf WKL}_{\rm aou}$).
This essentially follows from the following:

\begin{lemma}\label{lem:sep-main1}
$\wklQ\not\leqcW{\sf K}_\N$ and ${\sf WKL}_{\rm aou}\not\leqcW{\sf K}_\N$.
\end{lemma}

The former follows from Proposition \ref{prop:sep-main4} as ${\sf LLPO}\leq_{\sf W}{\sf WKL}_{\rm aou}$.
The latter is a known fact, which also follows from Proposition \ref{prop:sep-main5} as ${\sf LLPO}\leq_{\sf W}{\sf WKL}_{\leq 2}$.

Now, we discuss on realizability.
By Fact \ref{fact:LLPO-closed-KN}, ${\sf LLPO}^{\clo}$-realizability is equivalent to ${\sf K}_\N$-realizability, which may be considered as {\em Lifschitz realizability} (over the Kleene-Vesley algebra), and so the $\star$-closure property of ${\sf K}_\N$ has implicitly been known in the context of realizability (cf.~Lifschitz \cite[Lemma 3]{Lif}).
It is well-known that Lifschitz realizability validates ${\sf LLPO}$.
In our context, this is because ${\sf LLPO}$ is a tidy $\forall\exists$-principle and ${\sf LLPO}\leq_{\sf W}{\sf WKL}$, so ${\sf LLPO}^{\clo}$-realizability validates ${\sf LLPO}$ by Lemmas \ref{lem:below-WKL-rCT} and \ref{lem:main-lemma}.

If ${\sf BE}_\mathbb{Q}$ is boldface ${\sf LLPO}^{\clo}$-realizable, so is ${\sf WKL}_{=2}$ by Proposition \ref{prop:beQ-eq-wkl2}.
As seen in Observation \ref{obs:WKLs-are-tidy}, ${\sf WKL}_{=2}$ is a tidy $\forall\exists$-principle.
Hence, by the boldface version of Lemma \ref{lem:main-lemma}, we get ${\sf WKL}_{=2}\leqcW{\sf K}_\N$, which contradicts Lemma \ref{lem:sep-main1}.


Similarly, if ${\sf RDIV}$ is boldface ${\sf K}_\N$-realizable, then ${\sf WKL}_{\rm aou}$ is also ${\sf K}_\N$-realizable.
As seen in Observation \ref{obs:WKLs-are-tidy}, ${\sf WKL}_{\rm aou}$ is a tidy $\forall\exists$-principle.
Hence, by the boldface version of Lemma \ref{lem:main-lemma}, we get ${\sf WKL}_{\rm aou}\leqcW{\sf K}_\N$, which contradicts Lemma \ref{lem:sep-main1}.
%
This verifies the item (\ref{item:main1}).


\medskip
\noindent
(\ref{item:main2})
We use $(\wklQ)^{\clo}$-realizability to realize $\neg\Pi^0_1\lemr+{\sf BE}_\mathbb{Q}+\neg{\sf BE}+\neg{\sf RDIV}$.
By Observation \ref{obs:WKLs-are-tidy}, $\wklQ$ is a tidy $\forall\exists$-principle and $\wklQ\leq_{\sf W}{\sf WKL}$; hence, it is evident that $(\wklQ)^{\clo}$-realizability validates ${\sf BE}_\mathbb{Q}$ by Proposition \ref{prop:beQ-eq-wkl2}.
Thus, to prove (\ref{item:main2}), it suffices to show the following:

\begin{lemma}\label{lem:sep-main2}
${\sf LPO}\not\leqcW(\wklQ)^{\clo}$, ${\sf WKL}_{\leq 2}\not\leqcW(\wklQ)^{\clo}$ and ${\sf WKL}_{\rm aou}\not\leqcW(\wklQ)^{\clo}$.
\end{lemma}

The first assertion is trivial:
It is known that ${\sf WKL}^{\clo}\equiv_{\sf W}{\sf WKL}$.
This means that $P\leqW{\sf WKL}$ implies $P^{\clo}\leqW{\sf WKL}^{\clo}\equiv_{\sf W}{\sf WKL}$.
It is also well-known that ${\sf LPO}\not\leqW{\sf WKL}$, where recall that ${\sf LPO}$ is equivalent to $\Pi^0_1\mbox{-}{\sf LEM}_\mathbb{R}$.
The second assertion follows from Proposition \ref{prop:sep-main23} since $\wklQ\leq_{\sf W}{\limN}$ by Propositions \ref{prop:clop-to-two} and \ref{prop:dne-to-clop}.
The last assertion clearly follows from Proposition \ref{prop:sep-main5}.
This proves Lemma \ref{lem:sep-main2}; hence, as in (\ref{item:main1}), one can easily verify the item (\ref{item:main2}).


\medskip
\noindent
(\ref{item:main3})
We use $({\sf WKL}_{\rm aou})^{\clo}$-realizability to realize $\neg\Pi^0_1\mbox{-}{\sf LEM}_\mathbb{R}+{\sf RDIV}+\neg{\sf BE}_\mathbb{Q}$.
As ${\sf WKL}_{\rm aou}$ is tidy and equivalent to ${\sf RDIV}$, it is evident that $({\sf WKL}_{\rm aou})^{\clo}$-realizability validates ${\sf RDIV}$.
Thus, to prove (\ref{item:main3}), it suffices to show the following:

\begin{lemma}\label{lem:sep-main3}
${\sf LPO}\not\leqW({\sf WKL}_{\rm aou})^{\clo}$, and $\wklQ\not\leqW({\sf WKL}_{\rm aou})^{\clo}$.
\end{lemma}

The first assertion is trivial as in Lemma \ref{lem:sep-main2}.
The second assertion clearly follows from Proposition \ref{prop:sep-main4}.
This proves Lemma \ref{lem:sep-main3}; hence, as in (\ref{item:main1}), one can easily verify the item (\ref{item:main3}).


\medskip
\noindent
(\ref{item:main4})
We use $({\sf WKL}_{\rm clop})^{\clo}$-realizability to realize $\neg\Pi^0_1\lemr+{\sf IVT}_{\rm lin}+\neg{\sf BE}$.
As seen in Observation \ref{obs:WKLs-are-tidy}, ${\sf WKL}_{\rm clop}$ is a tidy $\forall\exists$-principle.
Hence, it is evident that $({\sf WKL}_{\rm clop})^{\clo}$-realizability validates ${\sf IVT}_{\rm lin}$ by Proposition \ref{prop:lin-eq-clop}.
To prove (\ref{item:main4}), it suffices to show the following:

\begin{lemma}\label{fact:separate-main2}
${\sf LPO}\not\leqcW({\sf WKL}_{\rm clop})^{\clo}$, and ${\sf WKL}_{\leq 2}\not\leqcW({\sf WKL}_{\rm clop})^{\clo}$.
\end{lemma}

The first assertion is trivial as in Lemma \ref{lem:sep-main2}.
The second assertion follows from Proposition \ref{prop:sep-main23} since ${\sf WKL}_{\rm clop}\leq_{\sf W}{\limN}$ by Proposition \ref{prop:dne-to-clop}.
This proves Lemma \ref{fact:separate-main2}; hence, as in (\ref{item:main1}), one can easily verify the item (\ref{item:main4}).


\medskip
\noindent
(\ref{item:main5})
We use $({\sf WKL}_{\leq 2})^{\clo}$-realizability to realize ${\sf BE}+\neg{\sf RDIV}$.
As ${\sf WKL}_{\leq 2}$ is equivalent to ${\sf BE}$, it is evident that $({\sf WKL}_{\leq 2})^{\clo}$-realizability validates ${\sf BE}$.
Moreover, $({\sf WKL}_{\leq 2})^{\clo}$-realizability refutes ${\sf RDIV}$ by Proposition \ref{prop:sep-main5}.
Hence, as in (1), one can easily verify the item (\ref{item:main5}).


\medskip
\noindent
(\ref{item:main6})
We use $(\wklQ\times{\sf WKL}_{\rm aou})^{\clo}$-realizability to realize $\neg\Pi^0_1\lemr+{\sf BE}_\mathbb{Q}+{\sf RDIV}+\neg{\sf IVT}_{\rm lin}+\neg{\sf BE}$.
It is evident that $(\wklQ\times{\sf WKL}_{\rm aou})^{\clo}$-realizability validates ${\sf BE}_\mathbb{Q}+{\sf RDIV}$.
To prove (\ref{item:main6}), it suffices to show the following:

\begin{lemma}\label{fact:separate-main2b}
${\sf LPO}\not\leqcW(\wklQ\times{\sf WKL}_{\rm aou})^{\clo}$, ${\sf WKL}_{\rm clop}\not\leqcW(\wklQ\times{\sf WKL}_{\rm aou})^{\clo}$, and ${\sf WKL}_{\leq 2}\not\leqcW(\wklQ\times{\sf WKL}_{\rm aou})^{\clo}$.
\end{lemma}

The first assertion is trivial as in Lemma \ref{lem:sep-main2}.
The second assertion clearly follows from Proposition \ref{prop:sep-main6}.
The last assertion follows from Proposition \ref{prop:sep-main23} since $\wklQ,{\sf WKL}_{\rm aou}\leq_{\sf W}{\limN}$ by Proposition \ref{prop:dne-to-clop}.
This proves Lemma \ref{fact:separate-main2b}; hence, as in (\ref{item:main1}), one can easily verify the item (\ref{item:main6}).


\medskip
\noindent
(\ref{item:main7})
We use $({\sf WKL}_{\leq 2}\times{\sf WKL}_{\rm aou})^{\clo}$-realizability to realize $\neg\Pi^0_1\mbox{-}{\sf LEM}_\mathbb{R}+{\sf BE}+{\sf RDIV}+\neg{\sf IVT}_{\rm lin}$.
It is evident that $({\sf WKL}_{\leq 2}\times{\sf WKL}_{\rm aou})^{\clo}$-realizability validates ${\sf BE}+{\sf RDIV}$.
To prove (\ref{item:main7}), it suffices to show the following:

\begin{lemma}\label{lem:sep-main-9}
${\sf LPO}\not\leqcW({\sf WKL}_{\leq 2}\times{\sf WKL}_{\rm aou})^{\clo}$, and ${\sf WKL}_{\rm clop}\not\leqcW({\sf WKL}_{\leq 2}\times{\sf WKL}_{\rm aou})^{\clo}$.
\end{lemma}

The first assertion is trivial as in Lemma \ref{lem:sep-main2}.
The second assertion follows from Proposition \ref{prop:sep-main78} since ${\sf WKL}_{\rm clop}\leq_{\sf W}{\limN}$ by Proposition \ref{prop:dne-to-clop}.
This proves Lemma \ref{lem:sep-main-9}; hence, as in (\ref{item:main1}), one can easily verify the item (\ref{item:main7}).


\medskip
\noindent
(\ref{item:main8})
We use $({\sf WKL}_{\leq 2}\times{\sf WKL}_{\rm clop})^{\clo}$-realizability to realize $\neg\Pi^0_1\mbox{-}{\sf LEM}_\mathbb{R}+{\sf BE}+{\sf IVT}_{\rm lin}+\neg{\sf IVT}$.
It is evident that $({\sf WKL}_{\leq 2}\times{\sf WKL}_{\rm clop})^{\clo}$-realizability validates ${\sf BE}+{\sf IVT}_{\sf lin}$ by Proposition \ref{prop:dne-to-clop}.
To prove (\ref{item:main8}), it suffices to show the following:

\begin{lemma}\label{lem:sep-main-8}
${\sf LPO}\not\leqcW({\sf WKL}_{\leq 2}\times{\sf WKL}_{\rm clop})^{\clo}$, and ${\sf WKL}_{\rm conv}\not\leqcW({\sf WKL}_{\leq 2}\times{\sf WKL}_{\rm clop})^{\clo}$.
\end{lemma}

The first assertion is trivial as in Lemma \ref{lem:sep-main2}.
The second assertion clearly follows from Proposition \ref{prop:sep-main6}.
This proves Lemma \ref{lem:sep-main-8}; hence, as in (\ref{item:main1}), one can easily verify the item (\ref{item:main8}).


\medskip
\noindent
(\ref{item:main9}): Straightforward.
\end{proof}

\begin{ack}
The author would like to thank Hajime Ishihara and Takako Nemoto for proposing the separation problem of ${\sf IVT}, {\sf BE}, {\sf LLPO}$ around 2014.
The author also thank Hajime Ishihara for reminding him of this problem at the Oberwolfach workshop in November 2017.
The author thank Takako Nemoto for encouraging the publication of this article, which was available as a preprint in 2020, in November 2024.
The author also would like to thank Matthew de Brecht and Satoshi Nakata for valuable discussions.
In the early stages of this study, the author was partially supported by JSPS KAKENHI Grant 15H03634, 19K03602, and the JSPS Core-to-Core Program (A. Advanced Research Networks).
The author's research was also partially supported by JSPS KAKENHI Grant 22K03401 and 23K28036.
\end{ack}

\bibliographystyle{plain}
\bibliography{references}

\end{document}



\subsubsection{TODO}

In later sections, we will show $\leq_{\gW}^c$-separation results for most principles in Figure \ref{figure:principles-over-izf} (where recall that $\leq_{\gW}^c$ stands for continuous generalized Weihrauch reducibility):

\begin{theorem}\label{thm:main-sub-GW}~
\begin{enumerate}
\item ${\sf RDIV}\not\leqcGW{\sf BE}$.
\item ${\sf BE}_\mathbb{Q}\not\leqcGW{\sf RDIV}$.
\item ${\sf IVT}_{\rm lin}\not\leqcGW {\sf RDIV}\times{\sf BE}$.
\item ${\sf BE}\not\leqcGW\Sigma^0_2\dmlr\times\Sigma^0_2\dner$.
\item ${\sf IVT}\not\leqcGW\Sigma^0_2\dmlr\times\Sigma^0_2\dner\times{\sf BE}$.
\end{enumerate}
\end{theorem}

However, not every such $\leq_{\gW}^c$-separation result yields an ${\bf IZF}$-separation result.


\newpage

\com{/TODO}
If we have the axiom ${\sf AC}_{\om,\om}!$ of unique choice on natural numbers, then $\Pi^0_1\lemr$ implies the existence of the Turing jump of a given real.
This implies the axiom scheme ${\sf ACA}$ of arithmetical comprehension, which clearly entails the law of excluded middle for each level of the arithmetical hierarchy.
By this reason, {\em under unique choice ${\sf AC}_{\om,\om}!$, the arithmetical hierarchy of the law of excluded middle collapses}.

A realizability model always validates the unique choice ${\sf AC}_{\om,\om}!$.
Thus, for instance, the following seems still open:

\begin{question}
For $n>0$, does ${\sf EL}_0$ prove $\Sigma^0_n\lemr\to\Sigma^0_{n+1}\lemr$?
\end{question}

We now go back to the discussion in Section \ref{sec:internal-Baire}:
The discrete limit operation $\lim_\N$ is closed under the compositional product, but not parallelizable, and its parallelization $\lim$ is not closed under the compositional product.
Obviously, its $\star$-closure $\lim^{\clo}$ dominates the whole arithmetical hierarchy.
A similar argument also applies for ${\sf LPO}$.
Roughly speaking, $\lim_\N$ can be thought of as the law of double negation elimination for $\Sigma^0_2$-formulas, $\Sigma^0_2\mbox{-}{\sf DNE}$ (see Proposition \ref{prop:dner-limN}).
The parallelization of a single-valued function is essentially the axiom ${\sf AC}_{\om,\om}!$ of unique choice on natural numbers.
The unique choice property for our realizability notions comes from the condition (4) in Lemma \ref{lem:pca-renewal}.
This is another way of explaining the collapsing phenomenon of the arithmetical hierarchy of the law of excluded middle.

In summary, in ${\sf AC}_{\om,\om}!$-models, ${\sf LPO}$ implies ${\sf ACA}$, so in particular ${\sf WKL}$.
However, it is easy to see that ${\sf WKL}\not\leq_{\gW}{\sf LPO}$.
This means that a $\leq_{\gW}$-separation result does not always imply a separation result in constructive reverse mathematics.
Then, we are also interested in how ${\sf AC}_{\om,\om}!$-models and generalized Weihrauch reducibility $\leq_{\gW}$ are different.
This is what we will deal with in the later sections.
\com{消す？/}

\subsubsection{Auxiliary principles}

To understand the structure of generalized Weihrauch reducibility $\leq_{\gW}$, we also consider the following principles ${\sf BE}_\mathbb{Q}$ and ${\sf IVT}_{\rm lin}$.
These two principles may look weird, but we consider them since they naturally occur in our proofs.
Indeed, all of these principles turn out to be equivalent to very weak variants of weak K\"onig's lemma (and indeed equivalent to weak K\"onig's lemma under countable choice); see also Section \ref{sec:weak-variants-wkl}.

\subsubsection*{Binary expansion for regular Cauchy rationals ${\sf BE}_\mathbb{Q}$}

We next consider the binary expansion principle ${\sf BE}_\mathbb{Q}$ for regular Cauchy rationals (more precisely, for non-irrationals), which states that, for any regular Cauchy real, if it happens to be a rational, then it has a binary expansion:
\[(\forall x\in[0,1])\;\left[\neg(\forall a,b\in\mathbb{Z},\;ax\not=b)\;\longrightarrow\;\exists f\colon\N\to\{0,1\},\;x=\sum_{n=1}^\infty 2^{-f(n)}\right].\]

\subsubsection*{Intermediate value theorem for piecewise linear maps ${\sf IVT}_{\rm lin}$}

Finally, we consider a (weird) variant of the intermediate value theorem.
A rational piecewise linear map on $\mathbb{R}$ is determined by finitely many rational points $\bar{p}=(p_0,\dots,p_\ell)$ in the plane.
Then, any continuous function on $\mathbb{R}$ can be represented as the limit of a sequence $(f_s)_{s\in\om}$ of rational piecewise linear maps coded by $(\bar{p}^s)_{s\in\om}$, where any $p^{s+1}_i$ belongs to the $2^{-s-1}$-neighborhood of the graph of $f_s$.
That is, the sequence $(\bar{p}^s)_{s\in\om}$ codes an approximation procedure of a continuous function $f$.
Then, consider the case that a code $(\bar{p}^s)_{s\in\om}$ of an approximation stabilizes, in the sense that $\bar{p}^s=\bar{p}^t$ for sufficiently large $s,t$.
Then, of course, it converges to a rational piecewise linear map.

In other words, we represent a continuous function as usual, but we consider the case that such a continuous function {\em happens to be a rational piecewise linear map}.
Formally, if $f$ is coded by $(\bar{p}^s)_{s\in\om}$ which does not change infinitely often (that is, $\neg\forall s\exists t>s\;\bar{p}^t\not=\bar{p}^s$) then we call $f$ {\em a rational piecewise linear map as a discrete limit}.
The principle ${\sf IVT}_{\rm lin}$ states that if $f\colon[0,1]\to\mathbb{R}$ is a rational piecewise linear map as a discrete limit such that $f(0)<0<f(1)$ then there is $x\in[0,1]$ such that $f(x)=0$.

\end{document}

\begin{proof}[Proof of Theorem \ref{thm:main-theorem}]
(1)
Consider ${\sf LLPO}^{\clo}$-realizability.
By Fact \ref{fact:LLPO-closed-KN}, this is equivalent to ${\sf K}_\N$-realizability and to Lifshitz realizability.
It is evident that ${\sf K}_\N$-realizability validates ${\sf LLPO}$ (so $\Sigma^0_1\dmlr$); see also Chen-Rathjen \cite{ChRa12}.

If ${\sf BE}$ is boldface ${\sf LLPO}^{\clo}$-realizable, so is ${\sf WKL}_{\leq 2}$.
Then, as in the argument in Section \ref{sec:internal-Baire}, one can see that there is a boldface ${\sf LLPO}^{\clo}$-realizable function which, given $(T,a,b)$, where $T$ is a tree, $a$ witnesses that $T$ is infinite, and $b$ witnesses that $T$ has at most two nodes at each level, returns an infinite path through $T$.
However given such a $T$ one can always recover $a$ and $b$ in an effective manner, so they have no extra information.
This means that ${\sf WKL}_2\leqcW{\sf K}_\N$, which contradicts Fact \ref{fact:separate-KN-AoUC}.

Similarly, if ${\sf RDIV}$ is boldface ${\sf K}_\N$-realizable, then weak K\"onig's lemma for aou-trees is also ${\sf K}_\N$-realizable, so there is a boldface ${\sf K}_\N$-realizable function which, given $(T,a,b)$, where $T$ is a tree, $a$ witnesses that $T$ is infinite, and $b$ witnesses that $T$ has all nodes or a single node at each level, returns an infinite path through $T$.
However given such a $T$ one can always recover $a$ and $b$ in an effective manner, so they have no extra information.
This means that ${\sf RDIV}\leqcW{\sf K}_\N$, which contradicts Fact \ref{fact:separate-KN-AoUC}.

Consequently, ${\bf IZF}+\Sigma^0_1\dmlr+\neg{\sf RDIV}+\neg{\sf BE}$ is ${\sf LLPO}^{\clo}$-realizable.

(2)
Consider $({\sf Lim}_\N)^{\clo}$-realizability, which is equivalent to ${\limN}$-realizability by Fact \ref{fact:limN-game}.
Note that any realizability validates Markov's principle; hence, as mentioned in Section \ref{sec:lem-for-2}, ${\limN}$-realizability validates $\Sigma^0_2\dner$.
If ${\sf BE}$ is boldface ${\sf Lim}_\N$-realizable, then as above we have ${\sf WKL}_{\leq 2}\leqcW{\sf Lim}_\N$, which contradicts Proposition \ref{prop:limN-redutes-BE}.
Hence, ${\bf IZF}+\Sigma^0_2\dner+\neg{\sf BE}$ is $({\sf Lim}_\N)^{\clo}$-realizable.

(3)
Consider $({\sf RT}^1_2)^{\clo}$-realizability.
Note that any realizability validates Markov's principle; hence an in Proposition \ref{prop:dmlr-rt12}, one can check that $\Sigma^0_2\dmlr$ is $({\sf RT}^1_2)^{\clo}$-realizable.
If ${\sf BE}$ is boldface $({\sf RT}^1_2)^{\clo}$-realizable, then as above we have ${\sf WKL}_{\leq 2}\leqcW({\sf RT}^1_2)^{\clo}$, which contradicts Proposition \ref{prop:WKL2-RT12-sepa}.
Hence, ${\bf IZF}+\Sigma^0_2\dmlr+\neg{\sf BE}$ is $({\sf RT}^1_2)^{\clo}$-realizable.

(4)
Consider $({\sf AoUC})^{\clo}$-realizability.
It is evident that $({\sf AoUC})^{\clo}$-realizability validates ${\sf AoUC}$, so ${\sf RDIV}$ as well.
If ${\Sigma^0_1\lemr}$ is boldface $({\sf AoUC})^{\clo}$-realizable, so is ${\sf LPO}$; however it clearly implies ${\sf LPO}\leqcW{\sf AoUC}^{\clo}$, which contradicts Fact \ref{fact:trivial-lpo-aouc}.
If ${\sf BE}$ is boldface $({\sf AoUC})^{\clo}$-realizable, then as above we have ${\sf WKL}_{\leq 2}\leqcW({\sf AoUC})^{\clo}$, which contradicts Proposition \ref{prop:sep-wkl2-aoucg}.
Hence, ${\bf IZF}+\neg\Pi^0_1\lemr+{\sf RDIV}+\neg{\sf BE}$ is $({\sf AoUC})^{\clo}$-realizable.

(5)
Consider $({\sf WKL}_{\leq 2})^{\clo}$-realizability.
It is evident that $({\sf WKL}_{\leq 2})^{\clo}$-realizability validates ${\sf WKL}_{\leq 2}$, so ${\sf BE}$ as well.
If ${\sf RDIV}$ is boldface $({\sf WKL}_{\leq 2})^{\clo}$-realizable, then as above we have ${\sf AoUC}\leqcW({\sf WKL}_{\leq 2})^{\clo}$, which contradicts Proposition \ref{prop:aouc-sep-wkl-2-g}.
Hence, ${\bf IZF}+{\sf BE}+\neg{\sf RDIV}$ is $({\sf WKL}_{\leq 2})^{\clo}$-realizable.

(6)
Consider $({\sf WKL}_{\leq 2}\times{\sf AoUC})^{\clo}$-realizability.
As above, it is evident that $({\sf WKL}_{\leq 2}\times{\sf AoUC})^{\clo}$-realizability validates both ${\sf BE}$ and ${\sf RDIV}$.
If ${\sf IVT}$ is boldface $({\sf WKL}_{\leq 2}\times{\sf AoUC})^{\clo}$-realizable, then as above we have ${\sf IVT}\leqcW({\sf WKL}_{\leq 2}\times{\sf AoUC})^{\clo}$, which contradicts Proposition \ref{prop:sep-ivt-wkl-2-aouc-g}.If ${\Sigma^0_1\lemr}$ is boldface $({\sf AoUC})^{\clo}$-realizable, so is ${\sf LPO}$; however it clearly implies ${\sf LPO}\leqcW({\sf WKL}_{\leq 2}\times{\sf AoUC})^{\clo}$, which contradicts Fact \ref{fact:lpo-wkl2-aouc-ga}.
Hence, ${\bf IZF}+\neg{\Pi^0_1\lemr}+{\sf BE}+{\sf RDIV}$ is $({\sf WKL}_{\leq 2}\times{\sf AoUC})^{\clo}$-realizable.

(7) and (8) also follow from the same argument using results from Sections \ref{sec:lem-for-7} and \ref{sec:lem-for-8}.

(9)--(13): By similar arguments.
\end{proof}

\newpage

\section{Appendix}

\subsection{Realizability Predicate}

\begin{itemize}
\item $e\rea R(\bar{a})\iff\N\models R(\bar{a})$, where $R$ is a primitive recursive relation.
\item $e\rea\bfN(a)\iff a\in\N\;\;\&\;\;e=\ul{a}$.
\item $e\rea\set(a)\iff a\in V^{\rm set}$.
\item $e\reak a\in b\iff\exists c\;[\spair{\pi_0e}{c}\in b\;\land\;\pi_1e\rea a=c]$.
\item $e\rea a\in b\iff\foralle\langle d\rangle\in\jump(e)\;d\reak a\in b$.
\item $d\reak a\subseteq' b\iff(\forall p,c)\;[\spair{p}{c}\in a\;\rightarrow\;dp\rea c\in b]$.
\item If $a,b\in\N$, $e\rea a=b\iff a=b$.
\item If $a,b\in V^{\rm set}$, $e\rea a=b\iff\foralle d\in\jump(e)\;[\pi_0d\reak a\subseteq'b\;\land\;\pi_1d\reak b\subseteq'a]$.
\item $e\rea A\land B\iff \pi_0e\rea A\;\land\;\pi_1e\rea B$.
\item $e\reak A\lor B\iff(\pi_0e=0\land\pi_1e\rea A)\lor(\pi_0e=1\land\pi_1e\rea B)$.
\item $e\rea A\lor B\iff(\foralle d\in\jump(e))\;d\reak A\lor B$.
\item $e\rea \neg A\iff(\forall a)\;a\not\rea A$.
\item $e\rea A\to B\iff(\forall a)\;[a\rea A\;\rightarrow\;ea\rea B]$.
\item $e\reak \forall xA\iff(\forall c\in V)\;e\rea A[c/x]$.
\item $e\rea \forall xA\iff(\foralle d\in\jump(e))\;d\reak\forall xA$.
\item $e\reak \exists xA\iff(\exists c\in V)\;e\rea A[c/x]$.
\item $e\rea \exists xA\iff(\foralle d\in\jump(e))\;d\reak\exists xA$.
\end{itemize}

In our case, we may assume that $\ul{n}=n0^\om$. 

$\foralle x\in A\;P(x)$ means that $\exists x\in A$ and $\forall x\in A\;P(x)$. 

\subsection{Single-valued part}

Given $\jump$, one can get a binary operation $\ast$ on $M$ by $a\ast b\downarrow= c$ iff $\jump(f_a)(b)=\{c\}$.

Proof of Lemma 4.6:

There exists $\one$ such that if $\jump(e)$ is a singleton then $\one e\in\jump(e)$.
%
\[
\infer{\one e\rea\mathbf{N}(a)}{
\infer{\jump(e)=\{\ul{a}\}}{
		\foralle d\in\jump(e)\;d\rea \mathbf{N}(a)
	}
}
\]

For a primitive recursion relation $R$ and the predicate ${\bf Set}$, the claim is trivial.
Similar for $a=b$ for $a,b\in\N$.
\[
	\infer{{\sf j}e\rea a=b}{
		\infer{\foralle c\in\jump({\sf j}e)\;c\reak a=b}{
			\infer{\foralle c\in\jump\jump(e)\;c\reak a=b}{
				\infer{\foralle d\in\jump(e)\foralle c\in\jump(d)\;c\reak a=b}{
					\foralle d\in\jump(e)\;d\rea a=b
				}
			}
		}
	}
\]

For $\land$:
\[
	\infer{\langle\chi_A({\sf u}\pi_0e),\chi_B({\sf u}\pi_1e)\rangle\rea A\land B}{
		\infer{\chi_A({\sf u}\pi_0e)\rea A\mbox{ and }\chi_B({\sf u}\pi_1e)\rea B}{
			\infer{(\foralle c\in\jump({\sf u}\pi_0e)\;c\rea A)\mbox{ and }(\foralle c\in\jump({\sf u}\pi_1e)\;c\rea B)}{
				\infer{\foralle d\in\jump(e)\;[\pi_0d\rea A\mbox{ and }\pi_1d\rea B]}{
					\foralle d\in\jump(e)\;d\rea A\land B
				}
			}
		}
	}
\]

For $\to$:
\[
\infer{\lamq a.\chi_B({\sf u}(\lambda x.xa)e)\rea A\to B}{
\infer{a\rea A\implies \chi_B({\sf u}(\lambda x.xa)e)\rea B}{
	\infer{a\rea A\implies \foralle d\in\jump({\sf u}(\lambda x.xa)e)\;d\rea B}{
		\infer{a\rea A\implies \foralle d\in\jump(e)\;(\lambda x.xa)d\rea B}{
			\infer{a\rea A\implies \foralle d\in\jump(e)\;da\rea B}{
				\infer{\foralle d\in\jump(e)\;[a\rea A\implies da\rea B]}{
					\foralle d\in\jump(e)\;d\rea A\to B
				}
			}
		}
	}
}
}
\]

Other constructions are trivial.

\subsection{Intuitionistic Logic}

The interpretation for $\to$ is the same as the usual Kleene realizability.
$(\to E)$:
\[
\infer{\lamq a.fa(xa)\rea D\vdash B}{
	\infer{a\rea D\implies fa(xa)\rea B}{
		\infer{a\rea D\mbox{ and }z\rea A\implies faz\rea B}{
			\infer{a\rea D\implies fa\rea A\to B}{
				f\rea D\vdash A\to B
			}
		}
		& &
		\infer{a\rea D\implies xa\rea A}{
			x \rea D\vdash A
		}
	}
}
\]

($\to I$):
\[
\infer{{\sf curry}(p):=\lamq y.\lamq z.p\langle y,z\rangle\rea D\vdash A\to B}{
	\infer{\lamq y.\lamq z.p\langle y,z\rangle\rea D\to(A\to B)}{
		\infer{y\rea D\implies \lamq z.p\langle y,z\rangle\rea A\to B}{
			\infer{y\rea D\text{ and }z\rea A\implies p\langle y,z\rangle \rea B}{
				\infer{x\rea D\land A\implies px\rea B}{
					\infer{p\rea D\land A\to B}{
						p\rea D,A\vdash B
					}
				}
			}
		}
	}
}
\]

For $(\land)$-rules:
\begin{align*}
\infer{\lamq x.\langle ax,bx\rangle\rea D\vdash A\land B}{
	\infer{x\rea D\implies\langle ax,bx\rangle\rea A\land B}{
		\infer{x\rea D\implies ax\rea A}{
			a\rea D\vdash A
		}
		& &
		\infer{x\rea D\implies bx\rea B}{
			b\rea D\vdash B
		}
	}
}
& &
\infer{\lamq x.\pi_i(cx)\rea D\vdash A_i}{
	\infer{x\rea D\implies\pi_i(cx)\rea D\vdash A_i}{
		\infer{x\rea D\implies cx\rea A_0\land A_1}{
			c\rea D\vdash A_0\land A_1
		}
	}
}
\end{align*}

For $\lor$-rules:
\[
\infer[\text{($\lor I$)}]{\lamq x.\iota\langle \underline{i},cx\rangle\rea D\vdash A_0\lor A_1}{
	\infer{x\rea D\implies \iota\langle\underline{i},cx\rangle\rea A_0\lor A_1}{
	\infer{x\rea D\implies \langle\underline{i},cx\rangle\reak A_0\lor A_1}{
		\infer{x\rea D\implies cx\rea A_i}{
			c\rea D\vdash A_i
		}
	}
	}
}
\]

For ($\lor E$):
\[
\infer{\chi_C({\sf u}[s,t]e)\rea C}{
\infer{\foralle d\in\jump({\sf u}[s,t]e)\; d\rea C}{
\infer{\foralle d\in[s,t]\jump(e)\;d\rea C}{
\infer[\text{($\lor E$)}]{(\foralle p\in\jump{e})\;[s,t]p:=\ifthen{\pi_0p}{s(\pi_1p)}{t(\pi_1p)}\rea C}{
	\infer{\foralle p\in\jump(e)\;p\reak A\lor B}{
		e\rea A\lor B
	}
	& &
	\infer{x\rea A\implies sx\rea C}{
		s\rea A\vdash C
	}
	& &
	\infer{x\rea B\implies tx\rea C}{
		t\rea B\vdash C
	}
}
}
}
}
\]

一応，$\lor$の除去規則($\lor E$)の実現について説明すると，$p$は$A\lor B$を実現しているので，$\pi_0p=\underline{0}$の場合は，$\pi_1p\rea A$である．
よって，$x=\pi_1p$を考えれば，$sx=s(\pi_1p)\rea C$を得る．
同様に，$\pi_0p\not=\underline{0}$の場合は，$\pi_1p\rea B$であるから，$x=\pi_1p$を考えれば，$tx=t(\pi_1p)\rea C$を得る．

$\Gamma$が空でない場合，$\lor$の除去規則($\lor E$)の実現は，一旦$\to$の導入規則($\to I$)を経由すると多少，楽になる．
つまり，$p,s,t$が($\lor E$)の上式を実現している場合，($\to I$)の実現の議論から，$s$と$t$を$\tilde{s}=\lamq yz.s\langle y,z\rangle$と$\tilde{t}=\lamq yz.t\langle y,z\rangle$に置き換えると，以下が成立する．
{\small
\begin{align*}
		\infer{a\rea D\implies (\foralle p\in\jump(e))\;pa\reak A\lor B}{
			\infer{e\rea D\implies e\rea A\lor B}{
				e\rea D\vdash A\lor B
			}
		}
	& &
		\infer{a\rea D\implies\tilde{s}a\rea A\to C}{
			\infer{\tilde{s}\rea D\vdash A\to C}{
				s\rea D,A\vdash C
			}
		}
	& &
		\infer{a\rea D\implies\tilde{t}a\rea B\to C}{
			\infer{\tilde{t}\rea D\vdash B\to C}{
				t\rea D,B\vdash C
			}
		}
\end{align*}
}

したがって，$a\rea D$が与えられているとき，以下が成立する．
\[
\infer{\chi_C({\sf u}[\tilde{s},\tilde{t}]e)a\rea C}{
\infer{(\foralle p\in\jump(e))\;[\tilde{s},\tilde{t}]pa=\ifthen{\pi_0(pa)}{\tilde{s}a(\pi_1(pa))}{\tilde{t}a(\pi_1(pa))}\rea C}{
	{(\foralle p\in\jump(e))\;pa\rea A\lor B}
	& &
	\infer{x\rea A\implies \tilde{s}ax\rea C}{
		\tilde{s}a\rea A\to C
	}
	& &
	\infer{x\rea B\implies \tilde{t}ax\rea C}{
		\tilde{t}a\rea B\to C
	}
}
}
\]

以上より，$\lamq a.\chi_C({\sf u}[\tilde{s},\tilde{t}]e)a$が$D\vdash C$を実現することが示された．

\medskip

まず，($\exists I$)と($\forall E$)の推論規則の実現について，$\Gamma$が空でない場合を考えても，記号が煩雑になるだけで証明は本質的に変わらないので，$\Gamma=\emptyset$であると仮定する．
$\Gamma$が空でない場合については，読者の演習問題とする．
特に変数条件のない($\exists I$)と($\forall E$)については，以下のように容易に実現できる．
\begin{align*}
\infer{\iota a\rea\exists uA(x,\bar{v})}{
\infer{(\exists x)\;a\rea A(x,\bar{v})}{
a\rea A(u,\bar{v})
}
}
& &
\infer{\chi_A(e)\rea A(u,\bar{v})}{
	\infer{(\foralle p\in\jump(e))(\forall u)\;p\rea A(u,\bar{v})}{
		e\rea \forall xA(x,\bar{v})
	}
}
\end{align*}

次に($\forall I$)について，議論を明確にするために$\Gamma=B$の場合を考えよう．
いま，$t(z,\bar{v})$が$B(\bar{v})\vdash A(z,\bar{v})$を実現すると仮定する．
このとき，変数$z$に$y$を代入してみよう．
変数条件から，$B$は$z$を変数として含まない，つまり$\bar{v}$の中に$z$は含まれない．
よって，文字通り$z$を$y$に置き換えると，$t(y,\bar{v})$は$B(\bar{v})\vdash A(y,\bar{v})$を実現する，と言い換えられる．
以上をまとめると，
\[
\infer{\lamq a.\iota ea\rea B(\bar{v})\vdash\forall xA(x,\bar{v})}{
\infer{a\rea B(\bar{v})\implies \iota ea\rea \forall xA(x,\bar{v})}{
\infer{a\rea B(\bar{v})\implies ea\reak \forall xA(x,\bar{v})}{
\infer{a\rea B(\bar{v})\implies (\forall y)\;ea\rea A(y,\bar{v})}{
\infer[\text{\small ($y$は任意なので)}]{(\forall y)\;[a\rea B(\bar{v})\implies ea\rea A(y,\bar{v})]}{
	\infer{a\rea B(\bar{v})\implies ea\rea A(y,\bar{v})}{
		\infer[\text{\small (変数条件より)}]{e\rea B(\bar{v})\vdash A(y,\bar{v})}{
			e\rea B(\bar{v})\vdash A(z,\bar{v})
		}
	}
}
}
}
}
}
\]

最後に($\exists E$)について議論する．
簡単のために$\Gamma$は省略する．
先程の議論のように，もし$t(z,\bar{v})$が$A(z,\bar{v})\vdash B(\bar{v})$を実現すると仮定する．
変数条件から$\bar{v}$の中には$z$は含まれないので，$z$に$y$を代入すると，$t(y,\bar{v})$は式$A(y,\bar{v})\vdash B(\bar{v})$を実現する．
これはどんな$y$についても成り立つから，任意の$y$について$t(y,\bar{v})$は$A(y,\bar{v})\vdash B(\bar{v})$を実現する．
以上より，
\[
\infer{\chi_B({\sf u}ep)\rea B(\bar{v})}{
\infer{\foralle d\in\jump({\sf u}ep)\;d\rea B(\bar{v})}{
\infer[\text{\small ($y$は任意なので)}]{(\forall a\in\jump(p))\;ea\rea B(\bar{v})}{
	\infer{(\foralle a\in\jump(p))(\exists u)\;a\rea A(u,\bar{v})}{
		p\rea \exists xA(x,\bar{v})
	}
	& &
		\infer{a\rea A(y,\bar{v})\implies ea\rea B(\bar{v})}{
			\infer[\text{\small (変数条件より)}]{e\rea A(y,\bar{v})\vdash B(\bar{v})}{
				e\rea A(z,\bar{v})\vdash B(\bar{v})
			}
		}
}
}
}
\]

\subsection{First order arithmetic}

We show that the number quantification is the same as the Kleene realizability.
For ($\forall^\N E$):
\[
\infer{\chi_A({\sf u}(\lambda x.xn)e)\rea A(n)}{
\infer{(\foralle d\in\jump({\sf u}(\lambda x.xn)e))\;d\rea A(n)}{
	\infer{(\forall n\in\N)(\foralle d\in\jump(e))\;(\lambda x.xn)d\rea A(n)}{
		\infer{(\foralle d\in\jump(e))(\forall n\in\N)\;dn\rea A(n)}{
			\infer{(\foralle d\in\jump(e))(\forall n)\;[a\rea \bfN(n)\implies da\rea A(n)]}{
				\infer{(\foralle d\in\jump(e))(\forall n)\;d\rea \bfN(n)\to A(n)}{
					e\rea (\forall x)\;\bfN(x)\to A(x)
				}
			}
		}
	}
}
}
\]

For ($\forall^\N I$):
\[
			\infer{\iota e\rea(\forall x)\;\bfN(x)\to A(x)}{
				\infer{(\forall n)\;e\rea\bfN(n)\to A(n)}{
					\infer{(\forall n)\;[n\rea\bfN(n)\implies en\rea A(n)]}{
						(\forall n\in\N)\;en\rea A(n)
					}
				}
			}
\]

For ($\exists^\N E$) and ($\forall^\N E$):
\[
			\infer{(\foralle d\in\jump(e))\;\pi_1d\rea A(\pi_0d)}{
				\infer{(\foralle d\in\jump(e))(\exists n)\;\pi_0d\rea\bfN(n)\;\mbox{ and }\;\pi_1d\rea A(n)}{
					\infer{(\foralle d\in\jump(e))(\exists n)\;d\rea\bfN(n)\land A(n)}{
						e\rea(\exists x)\;\bfN(x)\land A(x)
					}
				}
			}
\]

What's the difference between realizability and relative realizability?

\begin{lemma}
Let $x\in V^\jump$.
Then the following are equivalent:
\begin{enumerate}
\item $V^\jump\models\forall n\in\N\exists m\in\N\;P(n,m)$.
\item $P$ has a $\jump$-realizable choice.
\end{enumerate}
\end{lemma}

\[
\infer{\forall n\foralle r\in\jump\jump(u(\lambda x.x\ul{n})p)\exists m\;\pi_0r=\ul{m}\;\land\;\pi_1r\rea\langle n,\pi_0r\rangle\in x}{
	\infer{\forall n\foralle q\in\jump(p)\foralle r\in\jump((\lambda x.x\ul{n})q)\exists m\;\pi_0r=\ul{m}\;\land\;\pi_1r\rea\langle n,\pi_0r\rangle\in x}{
		\infer{\forall n\foralle q\in\jump(p)\foralle r\in\jump(q\ul{n})\exists m\;r\rea\mathbf{N}(m)\;\land\;\langle n,m\rangle\in x}{
			\infer{\forall n\foralle q\in\jump(p)\;q\ul{n}\rea\exists m\in\N\;\langle n,m\rangle\in x}{
				\infer{\forall n\foralle q\in\jump(p)\;[a\rea\mathbf{N}(n)\implies qa\rea\exists m\in\N\;\langle n,m\rangle\in x]}{
					\infer{\foralle q\in\jump(p)\forall n\;q\rea\mathbf{N}(n)\,\rightarrow\,\exists m\in\N\;\langle n,m\rangle\in x}{
						p\rea\forall n\in\N\exists m\in\N\;\langle n,m\rangle\in x
					}
				}
			}
		}
	}
}
\]

Now we discuss what the internal Baire space $(\N^\N)^\jump$ is.
\[(\N^\N)^\jump=\left\{x\in V^\jump:V^\jump\models\forall n\in\N\exists ! m\in\N\;\langle n,m\rangle\in x\right\}.\]

In particular, $p$ computes a name of a choice function $f\colon\om\to\jump\om$ for $x$.
Thus, the internal Baire space consists of $\jump$-realizable total functions on $\om$.

\subsection{Second order arithmetic}

Check $\exists^{\N^\N}$ and $\forall^{\N^\N}$:

\subsection{Equality axiom}

For reflexivity, for $n\in\N$, for any $a$, we have $a\rea n=n$ by definition.
We show reflexivity for sets $a=a\in V$ by induction on ranks ${\rm rank}(a)=\alpha$.
Assume that $e\rea b=b$ for any $b$ with ${\rm rank}(b)<\alpha$.
Then
\[
			\infer{\iota\langle f,e\rangle\rea b\in a}{
				\infer{\langle f,e\rangle\reak b\in a}{
					\infer{(\exists c)\;\spair{f}{c}\in a\; \mbox{ and }\; e\rea b=c}{
						\spair{f}{b}\in a & & e\rea b=b
					}
				}
			}
\]

Therefore,
\[
	\infer{\iota\langle \lambda x.\iota\langle x,e\rangle,\lambda x.\iota\langle x,e\rangle\rangle\rea a=a}{
		\infer{\forall f,b\;[\spair{f}{b}\in a\;\to\;(\lambda x.\iota\langle x,e\rangle)f\rea b\in a]}{
			\forall f,b\;[\spair{f}{b}\in a\;\to\;\iota\langle f,e\rangle\rea b\in a]
		}
	}
\]

Then, define $Q(e)=\iota\langle \lambda x.\iota\langle x,e\rangle,\lambda x.\iota\langle x,e\rangle\rangle$.
By recursion theorem, there is $e^\ast$ such that $Q(e^\ast)\simeq e^\ast$.
This implies that $e^\ast\rea b=b$ for any with ${\rm rank}(b)\leq\alpha$.

For symmetry:
\[
\]

なぜ可算選択が成り立たないのか考える．

\end{document}

\newpage